\documentclass{article}
\usepackage{amssymb}
 \usepackage{amsthm}
 \usepackage{caption}
\usepackage{amsmath,amssymb,amsopn,amsfonts,mathrsfs,amsbsy,amscd}
\usepackage{longtable}
\usepackage{multirow}
\usepackage[latin1]{inputenc}
\newcommand{\R}{\mathbb{R}}

\newcommand{\G}{{\mathfrak{g}}}

\newcommand{\h}{{\mathfrak{h}}}
\newcommand{\B}{{\mathfrak{b}}}

\newcommand{\om}{\omega}

\newcommand{\e}{\check{e}}
\newtheorem{Def}{Definition}
\newtheorem{theo}{Theorem}
\newtheorem{pr}{Proposition}
\newtheorem{Le}{Lemma}
\newtheorem{co}{Corollary}
\newtheorem{exem}{Example}
\newtheorem{remark}{Remark}

\usepackage{adjustbox}

\title{A complete classification of symplectic forms on six-dimensional  Frobeniusian real Lie algebras }
\author{T. A\"it Aissa, S. Elbourkadi and M. W. Mansouri\\
Department of Mathematics, Faculty of Sciences, Ibn Tofail University\\
	Analysis, Geometry and Applications Laboratory $($LAGA$)$\\ Kenitra, Morocco\\e-mail:
	mansourimohammed.wadia@uit.ac.ma\\tarik.aitaissa @uit.ac.ma\\said.elbourkadi @uit.ac.ma}
\begin{document}
	\maketitle
\begin{abstract}
In this paper, we give a complete classification of symplectic structures on  six-dimensional Frobeniusian solvable Lie algebras, up to symplectomorphism. We provide a scheme to
classify  the isomorphism classes of six-dimensional Frobeniusian solvable Lie algebras whose exact form has a Lagrangian ideal. We complete our classification by considering  Frobeniusian solvable Lie algebras without Lagrangian ideal.

\end{abstract}

key words: Symplectic form, Flat Lie algebras, Frobenius algebras.\\
 AMS Subject Class (2010): 17B30, 17B60, 53D05.

\section{Introduction and main result}
Let $\G$ be a finite-dimensional real Lie algebra. We say that
$(\G,\omega)$ is a {\it{symplectic Lie algebra}} if $\omega$ is a
non-degenerate skew-symmetric bilinear form on $\G$ and
\begin{equation}\label{cocy}
\big(\mathrm{d}\omega\big)(x,y,z):=\omega([x, y], z)+\omega([y, z], x)+\omega([z, x],y) = 0,
\end{equation}
for all $x$, $y$, $z\in\G$, where $\mathrm{d}$ is the Chevalley-Eilenberg differential. This is to say, $\omega$ is a non-degenerate
$2$-cocycle for the scalar cohomology of $\G$. Note that in such case,
$\G$ must be even-dimensional. We will then call $\omega$ a symplectic
structure on $\G$. There is a fundamental class of symplectic Lie algebras formed by \textit{Frobenius Lie algebras}, which admit a non-degenerate exact 2-form. Since the second cohomology group of a
Frobenius Lie algebra does not vanish in general, we may have symplectic structures that are not exacts. There is a one-to-one correspondence between symplectic Lie algebras and simply connected Lie groups with left invariant symplectic forms. In low dimensions, symplectic Lie algebras can be classified according to several approach  (\cite{Ova},
\cite{Com1}),  for the nilpotent case (\cite{Gom}, \cite{Burde2}, \cite{KGM}, \cite{Goz}, see also \cite{Fis}), and for the non-solvable case (also exact symplectic  Lie algebras) (\cite{AM}, \cite{Com2}).

Recall that two symplectic Lie algebras $(\G_1,\omega_1)$ and $(\G_2,\omega_2)$ are said to be
symplectomorphically equivalent if there exists an isomorphism of Lie algebras $\varphi : \G_1\longrightarrow\G_2$, which preserves the symplectic forms in the sense $\varphi^*\omega_2=\omega_1$.

A bilinear map $\nabla : \G\times\G\rightarrow\G$, written
as $(x,y)\mapsto\nabla_xy$, is called a \textit{connection} on $\G$. We define the torsion $\mathrm{T}=\mathrm{T}^\nabla$ as 
\begin{equation}
\mathrm{T}(x,y)=\nabla_xy-\nabla_yx-[x,y],~\forall~x,y\in\G
\end{equation}
and the curvature $\mathrm{R}=\mathrm{R}^\nabla$ as
\begin{equation}\label{Flatconnection}
\mathrm{R}(x,y)=\nabla_x\nabla_yz-\nabla_y\nabla_xz-\nabla_{[x,y]}z,~\forall~x,y,z\in\G.
\end{equation}
A \textit{ torsion-free} connection is one in which $\mathrm{T}\equiv0$.  It is equivalent for the mapping $\rho^\nabla: \G\rightarrow\mathrm{End}(\G)$ to be a representation of the Lie algebra $\G$ on itself if the curvature $\mathrm{R}$ of the connection $\nabla$ is zero. In this case, the connection $\nabla$  is called \textit{flat}.  A \textit{flat Lie algebra}  is a pair $(\G, \nabla)$, where $\nabla$ is a
torsion-free and flat connection on $\G$. \textit{Flat Lie algebras}  are known
under many different names.  \textit{Flat Lie algebras} are also called \textit{left-symmetric algebras} (\textit{LSAs} in short), \textit{Vinberg algebras}, \textit{Koszul algebras} or \textit{quasiassociative algebras}. For more details on \textit{left-symmetric algebras}, we refer the reader to the survey article \cite{Burde1} and the references therein. It is always hard to classify flat Lie algebras using algebraic isomorphisms, and it is one of the key problems of theory. The classification of 2-dimensional complex \textit{left-symmetric  algebras} was given in \cite{Bai2}, \cite{Burde3} and for the classification of 2-dimensional real \textit{left-symmetric  algebras} \cite{And}. The classification of 3-dimensional complex \textit{left-symmetric  algebras} was given in \cite{Bai}.

On the other hand, recall that two flat Lie algebras $(\G,\nabla)$ and $(\tilde{\G},\tilde{\nabla})$
are isomorphic  if there exists a linear isomorphism $\psi : \G\longrightarrow\tilde{\G}$ such
that, $\psi(\nabla_xy)=\tilde{\nabla}_{\psi(x)}\psi(y)$ for any $x,y\in\G$. 

Symplectic Lie algebras have a flat torsion-free connection  by nature, making them flat Lie algebras. 

It is known (\cite{Y}, \cite{MR}, \cite{DM}) that given a symplectic Lie algebra $(\G,\omega)$ the bilinear map given by 
\begin{equation}
\omega\big(\nabla_xy,z\big)=-\omega\big(y,[x,z]\big),~~\forall~x,y,z\in\G
\end{equation}
induces a flat torsion-free connection $\nabla$ on $\G$ that satisfies $\nabla_x y - \nabla_y  x = [x, y]$ on $\G$. If
in addition, the symplectic Lie algebra $(\G,\omega)$ admits a Lagrangian ideal $\mathfrak{j}$, then the
quotient algebra $\h=\G/\mathfrak{j}$ admits a flat torsion-free connection and the symplectic Lie
algebra $(\G,\omega)$ can be reconstructed from the flat Lie algebra $\h$ (see for instance \cite{BV}).

\textbf{Notation.}  For $\{e_i\}_{1\leq i\leq n}$  a basis of $\G$,
we denote by $\{e^i\}_{1\leq i\leq n}$ the
dual basis on $\G^\ast$ and  $e^{ij}$  the 2-form $e^i\wedge
e^j\in\wedge^2\G^*$. Set by  $\langle F \rangle:= \mathrm{span}_\R\{F\}$ the Lie
subalgebra  generated by the family $F$. A system of equations (\ref{Flatconnection}) corresponds to the \textit{LSA}-structure equations.

The main purpose of this article is to show the following theorem.

\begin{theo}\label{Principtheo}
Let $(\G,\omega)$ be a six-dimensional Frobeniusian  Lie algebra. Then, $(\G,\omega)$ is isomorphic to exactly one of the following symplectic Lie algebras$:$
\begin{center}
\captionof{table}{ Frobeniusian Lie algebra 
 for which the nilradical is four-dimensional.}\label{Theo2}
{\renewcommand*{\arraystretch}{1.4}
\small\begin{tabular}{lll}
			\hline
		Algebra& Symplectic structures&Remarks \\
			\hline
$N_{6,28}$&$\omega_1^{\pm}=e^{13}\pm(e^{15}\mp2e^{25})+e^{46}$\\
			&$\omega_2^{\pm}=\tau e^{12}+e^{13}\pm(e^{15}\mp2e^{25})+e^{46}$& $\tau\neq0$\\
$N_{6,29}^{\alpha\neq0,\beta}$&$\omega_1= e^{13}+e^{16}+e^{23}+\frac{\beta}{\alpha} e^{26}+ e^{45}$&\\
& $\omega_2=-\tau e^{12}+e^{13}+e^{16}+e^{23}+\frac{\beta}{\alpha} e^{26}+ e^{45}$&$\tau(\alpha-\beta)\neq0$\\
$N_{6,29}^{\alpha=-1,\beta=0}$&$\omega_1=e^{13}+e^{16}+e^{23}+ e^{45}$&\\
&$\omega_2=2\tau e^{12}+e^{13}+e^{16}+e^{23}+ e^{45}$&\\
&$\omega_3=e^{13}+e^{16}+e^{23}+ e^{45}+2e^{46}$&\\
&$\omega_4=2\tau^\prime e^{12}+e^{13}+e^{16}+e^{23}+ e^{45}+2e^{46}$&$\tau\tau^\prime\neq0$\\
$N_{6,29}^{\alpha\neq0,-1,\beta=0}$&$\omega_1=e^{13}+(\alpha-1)e^{16}+e^{23}+e^{45}$&\\
&$\omega_2=\tau e^{12}+e^{13}+(\alpha-1)e^{16}+e^{23}+e^{45}$&$\tau\neq0$\\
$N_{6,29}^{\alpha=0,\beta=-1}$&$\omega_1=e^{13}+e^{23}+e^{26}+e^{45}$&\\
&$\omega_2=-\tau e^{12}+ e^{13}+e^{23}+e^{26}+e^{45}$&\\
&$\omega_3= e^{13}+e^{23}+e^{26}+e^{45}+e^{56}$&\\
&$\omega_4=-\tau^\prime e^{12}+ e^{13}+e^{23}+e^{26}+e^{45}+e^{56}$&$\tau\tau^\prime\neq0$\\
$N_{6,29}^{\alpha=0,\beta\neq-1,0}$&$\omega_1=e^{13}+e^{23}+e^{26}+e^{45}$&\\
&$\omega_2=-\lambda\tau e^{12}+ e^{13}+e^{23}+e^{26}+e^{45}$&$\tau\neq0,\lambda\neq\pm1$\\
$N_{6,30}^\alpha$&$\omega_1^\pm=\pm(e^{13}+e^{45})+\alpha e^{16}+e^{26}$&\\
&$\omega_2^\pm=\tau e^{12}\pm(e^{13}+e^{45})+\alpha e^{16}+e^{26}$&$\tau\neq0$\\
$N_{6,32}^\alpha$&$\omega_1=e^{16}+e^{23}+e^{45}$&\\
&$\omega_2=\alpha\tau e^{12}+e^{16}+e^{23}+e^{45}$&$\tau\neq0$\\
$N_{6,33}$&$\omega_1=e^{13}+e^{23}+e^{26}+e^{45}$&\\
&$\omega_2=\tau e^{12}+ e^{13}+e^{23}+e^{26}+e^{45}$&$\tau\neq0$\\
$N_{6,34}^\alpha$&$\omega_1=e^{13}+e^{15}+(\alpha+1)e^{23}+e^{26}+e^{45}$&\\
&$\omega_2=\tau e^{12}+ e^{13}+e^{15}+(\alpha+1)e^{23}+e^{26}+e^{45}$&$\tau\neq0$
\\\hline
\end{tabular}}
\end{center}

\newpage
\begin{center}
 \captionof{table}{Frobeniusian Lie algebra 
 for which the nilradical is five-dimensional.}
{\renewcommand*{\arraystretch}{1.4}
\small\begin{tabular}{l l l}
			\hline
			Algebra&Symplectic structures&Remarks\\\hline
	$\G_{6,82}^{\alpha,\lambda,\lambda_1}$&$\omega=e^{16}+\frac{1}{2}e^{24}+\frac{1}{2}e^{35}$&$\alpha=2,\lambda=\lambda_1=0$\\
	$\G_{6,82}^{\alpha,\lambda,\lambda_1}$&$\omega_1=e^{16}+\frac{1}{2}e^{24}+\frac{1}{2}e^{35}$&\\
	&$\omega_2=e^{16}+\frac{1}{2}e^{24}+\frac{1}{2}e^{35}+e^{56}$&$\alpha=2, \lambda=0, \lambda_1=1$
	\\
	$\G_{6,82}^{\alpha,\lambda,\lambda_1}$&$\omega_1=e^{16}+\frac{1}{2}e^{24}+\frac{1}{2}e^{35}$&\\
	&$\omega_2=e^{16}+\frac{1}{2}e^{24}+\frac{1}{2}e^{35}+e^{45}$&$\alpha=2,\lambda=0,\lambda_1=2$\\
	$\G_{6,82}^{\alpha,\lambda,\lambda_1}$&$\omega_1=e^{16}+\frac{1}{2}e^{24}+\frac{1}{2}e^{35}$&\\	
	&$\omega_2=e^{16}+\frac{1}{2}e^{24}+e^{25}+\frac{1}{2}e^{35}$&$\alpha=2,\lambda=\lambda_1-2>0$
	\\
	$\G_{6,82}^{\alpha,\lambda,\lambda_1}$&$\omega_1=e^{16}+\frac{1}{2}e^{24}+\frac{1}{2}e^{35}$&\\	
	&$\omega_2=e^{16}+\frac{1}{2}e^{24}+\frac{1}{2}e^{35}+e^{45}$&$\alpha=2,\lambda=-\lambda_1+2$, $1<\lambda_1<2$\\
	$\G_{6,82}^{\alpha,\lambda,\lambda_1}$&$\omega_1=e^{16}+\frac{1}{2}e^{24}+\frac{1}{2}e^{35}$\\
	&$\omega_2=e^{16}+\frac{1}{2}e^{24}+\frac{1}{2}e^{35}+e^{45}$&\\
	&$\omega_3=e^{16}+\frac{1}{2}e^{24}+\frac{1}{2}e^{35}+e^{46}$&\\
	&$\omega_4=e^{16}+\frac{1}{2}e^{24}+\frac{1}{2}e^{35}+e^{45}+e^{46}$&$\alpha=2,\lambda=\lambda_1=1$\\
$\G_{6,82}^{\alpha,\lambda,\lambda_1}$&$\omega_1=e^{16}+\frac{1}{2} e^{24}+\frac{1}{2} e^{35}$&\\
&$\omega_2=e^{16}+\frac{1}{2} e^{24}+\frac{1}{2} e^{35}+e^{46}$&\\		
&$\omega_3=e^{16}+\frac{1}{2} e^{24}+e^{25}+\frac{1}{2} e^{35}$&\\
&$\omega_4=e^{16}+\frac{1}{2} e^{24}+e^{25}+\frac{1}{2} e^{35}+e^{46}$&$\alpha=2,\lambda=1,\lambda_1=3$
			\\
	$\G_{6,82}^{\alpha,\lambda,\lambda_1}$&$\omega_1=e^{16}+\frac{1}{2}e^{24}+\frac{1}{2}e^{35}$&\\
	&$\omega_2=e^{16}+\frac{1}{2}e^{24}+\frac{1}{2}e^{35}+e^{46}$&$\alpha=2, \lambda=1, \lambda_1>1,\lambda_1\neq 2,3$		
			\\
	$\G_{6,82}^{\alpha,\lambda,\lambda_1}$		&$\omega_1=e^{16}+\frac{1}{2}e^{24}+\frac{1}{2}e^{35}$&\\
	&$\omega_2=e^{16}+\frac{1}{2}e^{24}+\frac{1}{2}e^{35}+e^{56}$&$\alpha=2, 0<\lambda<1, \lambda_1=1$
			\\
$\G_{6,82}^{\alpha,\lambda,\lambda_1}$&$\omega=e^{16}+\frac{1}{2}e^{24}+\frac{1}{2}e^{35}$&$\alpha=2, \lambda\neq\pm(\lambda_1-2), 0\leq\lambda\leq\lambda_1$\\
	&& $(\lambda,\lambda_1)\neq \text{values previously obtained}$
	\\				
			$\G_{6,83}^{\lambda,\alpha}$&$\omega_1=e^{16}+\frac{1}{2}e^{24}+\frac{1}{2}e^{35}$&
			\\&$\omega_2=e^{16}+\frac{1}{2}e^{24}+\frac{1}{2}e^{35}+e^{46}$&
			\\&$\omega_3^\pm=e^{16}+\frac{1}{2}e^{24}+\frac{1}{2}e^{35}\pm e^{45}$&
			\\&$\omega_4=e^{16}+\frac{1}{2}e^{24}+\frac{1}{2}e^{35}\pm e^{45}+\tau e^{46}$&$\lambda=1,\alpha=2$, $\tau\neq0$
			\\
			$\G_{6,83}^{\lambda,\alpha}$&$\omega=e^{16}+\frac{1}{2}e^{24}+\frac{1}{2}e^{35}$&$\lambda>0$, $\lambda\neq,1,\alpha=2$
			\\
			$\G_{6,85}^{\lambda=1}$&$\omega_1^\pm
			=\pm(e^{16}+\frac{1}{2}e^{24}+\frac{1}{2}e^{35})$&\\
			&$\omega_2^\pm
			=\pm(e^{16}+\frac{1}{2}e^{24}+\frac{1}{2}e^{35})+e^{46}$&
			\\$\G_{6,85}^{\lambda=2}$&$\omega_1^\pm
			=\pm(e^{16}+\frac{1}{2}e^{24}+\frac{1}{2}e^{35})$&\\
			&$\omega_2^\pm
			=\pm(e^{16}+\frac{1}{2}e^{24}+\frac{1}{2}e^{35})+e^{45}$&
			\\$\G_{6,85}^{\lambda}$&$\omega^\pm=\pm(e^{16}+\frac{1}{2}e^{24}+\frac{1}{2}e^{35})$&$\lambda\geq0$, $\lambda\neq1,2$
			\\\hline
	\end{tabular}}
\end{center}

\begin{center}
{\renewcommand*{\arraystretch}{1.4}
\small\begin{tabular}{l l l}
			\hline
		Algebra& Symplectic structures&Remarks \\
			\hline
			$\G_{6,86}$&$\omega=e^{16}+\frac{1}{2}e^{24}+\frac{1}{2}e^{35}$&
			\\
			$\G_{6,87}$&$\omega^\pm=\pm(e^{16}+\frac{1}{2}e^{24}+\frac{1}{2}e^{35})$&
			\\
			$\G_{6,88}^{\mu_0,\alpha,\nu_0}$&$\omega_1=e^{16}-\frac{1}{2\mu_0}e^{24}-\frac{1}{2\mu_0}e^{35}$&\\
			&$\omega_2^{\pm}=e^{16}\pm e^{23}-\frac{1}{2\mu_0}e^{24}-\frac{1}{2\mu_0}e^{35}$&$\mu_0>0,\alpha=-2\mu_0, \nu_0=1$
			\\
			$\G_{6,88}^{\mu_0,\alpha,\nu_0}$&$\omega_1=e^{16}+\frac{1}{2\mu_0}e^{24}+\frac{1}{2\mu_0}e^{35}$&\\
			&$\omega_2^\pm=e^{16}+\frac{1}{2\mu_0}e^{24}+\frac{1}{2\mu_0}e^{35}\pm e^{45}$&$\mu_0>0,\alpha=2\mu_0, \nu_0=1$
			\\
		$\G_{6,88}^{\mu_0,\alpha,\nu_0}$&$\omega=e^{16}+\frac{1}{\alpha}e^{24}+\frac{1}{\alpha}e^{35}$&$\mu_0>0,\alpha\neq\pm 2\mu_0,\alpha\neq0, \nu_0=1$	
			\\
			$\G_{6,92}^{\alpha,\nu_0,\mu_0}$&$\omega=e^{16}+\frac{1}{2}e^{24}+\frac{1}{2}e^{35}$&$\alpha=2$, $0<\nu_0\leq\mu_0$ or $\nu_0=\mu_0=0$\\
			$\G_{6,92}^{\prime\prime}$&$\omega=e^{14}+e^{25}+e^{36}$&\\
	$\G_{6,92}^{\prime\prime\prime}$&$\omega=e^{14}+e^{25}+e^{36}$&		
			
			\\
	$\G_{6,94}^{\lambda=0}$&$\omega_1=e^{16}+\frac{1}{2} e^{25}+\frac{1}{2} e^{34}$&		
			\\
			&$\omega_2^\pm=e^{16}+\frac{1}{2} e^{25}+\frac{1}{2} e^{34}\pm e^{36}$&
			\\
$\G_{6,94}^{\lambda=-3}$	&$\omega_1=e^{16}-e^{25}-e^{34}$&\\
		&$\omega_2=-e^{15}-e^{24}-2e^{36}$&\\
			$\G_{6,94}^{\lambda=-\frac{1}{2}}$&$\omega_1=e^{16}+\frac{2}{3}e^{25}+\frac{2}{3}e^{34}$&\\
			&$\omega_2=e^{16}+e^{23}+\frac{2}{3}e^{25}+\frac{2}{3}e^{34}$&\\
			$\G_{6,94}^{\lambda}$&$\omega=e^{16}+\frac{1}{2+\lambda} e^{25}+\frac{1}{2+\lambda}e^{34}$&$\lambda\neq-\frac{1}{2},-3,-2,0$
			
			\\
			$\G_{6,95}$&$\omega_1^\pm
			=\pm(e^{16}+\frac{1}{2}e^{25}+\frac{1}{2}e^{34})$&\\
			&$\omega_2^\pm
			=\pm(e^{16}+\frac{1}{2}e^{25}+\frac{1}{2}e^{34}+\tau e^{36})$&$\tau\neq0$
			\\
			$\G_{6,96}$&$\omega=e^{16}+\frac{1}{3}e^{25}+\frac{1}{3}e^{34}$&		\\
			$\G_{6,97}$&$\omega^{\pm}=\pm(e^{16}+\frac{1}{4}e^{25}+\frac{1}{4}e^{34})$&\\	$\G_{6,98}^{h_1=0}$&$\omega_1=e^{16}+e^{25}+e^{34}$\\
			&$\omega_2=e^{16}+e^{25}+e^{34}+ e^{56}$&\\
			&$\omega_3^\pm=e^{16}+e^{25}+e^{34}\pm e^{46}$&\\
			&$\omega_4=e^{16}+e^{25}+e^{34}+e^{45}+\tau_1 e^{56}$&\\
			&$\omega_5=e^{16}+e^{25}+e^{34}+ e^{45}+\tau_2 e^{46}$&$\tau_1\in\R$, $\tau_2\neq0$\\
$\G_{6,98}^{h_1\neq0}$&$\omega=e^{16}+e^{25}+e^{34}-h_1 e^{35}+\tau_1 e^{45}+\tau_2 e^{46}$,&$\tau_1,\tau_2\in\R,~h_1\in\R^\ast$\\
$\G_{6,99}$&$\omega=e^{16}+\frac{1}{5}e^{25}+\frac{1}{5}e^{34}$&
\\\hline
\end{tabular}}
\end{center}

\end{theo}
In addition to the previous identified indecomposable Frobeniusian Lie algebras with Lagrangian ideals  in dimension $6$, two families of Frobenuisian Lie algebras can be distinguished. We refer to Proposition \ref{Principale2} for Frobeniusian indecomposable Lie algebras without Lagrangian  ideals, and Corollary \ref{decompoFrob} for Frobeniusian decomposable Lie algebras.

Our main result (Theorem \ref{Principtheo}) consists of three steps: 

We can distinguish two types of six-dimensional Fobeniusian Lie algebras: In the first type of Frobeniusian Lie algebras ($47$ algebras), the exact form has a Lagrangian ideal, whereas in the second, the exact form doesn't have a Lagrangian ideal ($8$ algebras). 
 
\begin{enumerate} 
\item For the first type, we present a scheme to classify the isomorphism classes of six-dimensional Frobeniusian  Lie algebras with Lagrangian ideals as follows: Let $(\G,\mathrm{d}\mu)$ be a six-dimensional Frobeniusian  Lie algebra, suppose that $\mathrm{d}\mu$ has a Lagrangian ideal $\mathfrak{j}$. It is known
that in the presence of an isotropic ideal $\mathfrak{j}$, the symplectic structure on $\G$ induces a flat torsion-free connection  on $\G/\mathfrak{j}$. As a consequence, $(\G/\mathfrak{j},\nabla^{\mathrm{d}\mu})$ is a flat Lie algebra with a right-identity element $``e"$. In fact, $\G$ is nothing more than a Lagrangian extension of a $3$-dimensional flat Lie algebra with a right-identity element. Then,
\begin{enumerate}
\item We classify all  $3$-dimensional flat Lie algebras  that possess a right-identity element (see Proposition \ref{Pr 3flat}).
\item The next step involves reconstructing the Frobeniusian Lie algebras, characterized by an exact form that includes a Lagrangian ideal (see Proposition \ref{Prclassifi}).
\item The classification of symplectic forms on Frobeniusian Lie algebras: When given a six-dimensional Frobeniusian Lie algebra $\tilde{\G}=(\G,\mathrm{d}\mu,\mathfrak{j})$ whose exact form has a Lagrangian ideal, every choice of the symplectic form $\omega$ over the Lie algebra $\tilde{\G}$ has a Lagrangian ideal (see Theorem \ref{Principale}). As a consequence, there is a one-to-one correspondence between the classes of isomorphisms
of symplectic Lie algebras with the same property of $\tilde{\G}$  and the triples $(\h,\nabla,[\alpha])$, where $[\alpha]\in H^2_{L,\rho}(\h,\h^\ast)$. Detailed explanations of this result are provided in Section \ref{Classifisection}, while additional details are provided in Proposition \ref{Prclassifi}.
\end{enumerate}

\item For the second type, the method essentially involves calculating symplectic forms up to symplectomorphism. As a consequence, any six-dimensional  Frobeniusian Lie algebra has one or two symplectic forms (non-symplectomorphically isomorphic); the first is exact, while the second is not.  In Proposition \ref{Principale2}, we provide a classification of Frobeniusian indecomposable Lie algebras without Lagrangian ideal by adding one Frobeniusian decomposable Lie algebra (Corollary \ref{decompoFrob}, algebra $\mathfrak{d}^\prime_{4,\delta}\oplus\mathfrak{aff}(1,\R)$) without Lagrangian ideal.
\item In \cite{Muba1} Mubarakzyanov classified real six-dimensional indecomposable solvable Lie algebras with five-dimensional nilradicals. There are several problems with this paper, which is frequently cited (see, \cite{Tur2}, \cite{Shaba}). A number of improvements have been made to Mubarakzyanov's paper by Shabanskaya and Thompson \cite{Shaba}. This is why we propose reconstructing the Frobeniusian Lie algebras with Lagrangian ideals to suit the Mubarakzyanov classification also of Turkowski \cite{Tur3} (see Tables \ref{isodecompo}, \ref{isoTur} and \ref{isoMuba}). In addition, we aim to separate possible Frobeniusian Lie algebras in accordance with their symplectic structures.
\end{enumerate}
The paper is structured as follows. Section \ref{sec1} introduces some fundamental properties of symplectic Lagrangian reduction and provides our first results in this regard.   In Section \ref{sec3}, we give a complete classification of flat Lie algebras with a right-identity element. In Section \ref{sec4},  we classify all symplectic structures on six-dimensional Frobeniusian Lie algebras without Lagrangian ideal, up to symplectomorphism. In Section \ref{Classifisection}, we give a complete classification of six-dimensional Frobeniusian Lie algebras with Lagrangian ideal, as well as their symplectic structures, up to symplectomorphism.  Sections \ref{AppenA} and \ref{AppenB} serve as appendices where we provide detailed computations necessary for the proofs of   Theorem \ref{Principale}, Proposition \ref{Prclassifi} (Appendix \ref{AppenA}). Appendix \ref{AppenB} is dedicated to determining the accurate list of six-dimensional Frobeniusian Lie algebras, as initially presented in \cite{Tur3},  \cite{Muba1} and \cite{Com2}. This list is rectified in \cite{Shaba} through a comparison with our updated list.
\section{Lagrangian extensions of flat Lie algebras}\label{sec1}
As a prelude to the main point of this paper, all the preliminaries are presented in this section. We briefly review some of the standard facts about Lagrangian extensions of flat Lie algebras and their relation to Lagrangian reductions \cite{BV}. An important result of Lagrangian symplectic extension theory is that the isomorphism classes of Lagrangian symplectic extensions of a flat Lie algebra $(\h,\nabla)$ are parametrized by a suitable restricted cohomology group
$H^2_{L,\nabla}(\h, \h^\ast)$.

\textit{Symplectic Lagrangian reduction.} Let $(\G,\omega)$ be a symplectic Lie algebra. An ideal $\mathfrak{j}$ of $(\G,\omega)$ is called \textit{isotropic} if $\mathfrak{j}\subset\mathfrak{j}^{\perp_{\omega}}$ with
\begin{equation*}
\mathfrak{j}^{\perp_{\omega}}=\big\{x\in\G~|~\omega(x,y)=0,~\forall y\in\mathfrak{j}\big\}.
\end{equation*}
If the orthogonal $\mathfrak{j}^{\perp_{\omega}}$ is an ideal in $\G$ we call $\mathfrak{j}$ a \textit{normal isotropic} ideal. If $\mathfrak{j}$ is a
maximal isotropic subspace $\mathfrak{j}$ is called a \textit{Lagrangian ideal}. Let $\mathfrak{j}$ a normal ideal of $(\G,\omega)$ and let $\h=\G/\mathfrak{j}$
denote the associated quotient Lie algebra. From $\omega$ we obtain a non-degenerate
bilinear pairing $\omega_\h$ between $\h$ and $\mathfrak{j}$, by declaring
\begin{equation*}
\omega_\h(\overline{x},u)=
\omega(x,u),~\forall \overline{x}\in\h, u\in\mathfrak{j},
\end{equation*}
where, for $x\in\G$, $\overline{x}$ denotes its class in $\h$. 
\begin{pr}\textsc{\cite{BV}}\label{Pr1isomo}
The homomorphism $\om_\h\in \mathrm{Hom}(\h,\mathfrak{j}^*)$,
$\overline{x}\longmapsto\om_\h(\overline{x},.) $, is an isomorphism and $\h$ carries a flat torsion-free connection
 defined by the equation
\begin{equation}\label{Pr1connex}\om_\h(\overline{\nabla}_{\overline{x}} \overline{y},u)=-\om(y,[{x},u]), ~~\forall
\overline{x},\,\overline{y}\in\h,\;u\in\mathfrak{j}.\end{equation}
\end{pr}

Let $(\h,\nabla)$ be a flat Lie algebra, that is, a Lie algebra endowed with a flat
torsion-free connection $\nabla$. Since $\nabla$ is a flat connection, the association $x \mapsto \nabla_x$ defines a representation $\h  \rightarrow \mathrm{End}(\h)$. We denote by $\rho^\nabla : \h \rightarrow \mathrm{End}(\h^\ast)$ the dual representation,
which satisfies
\begin{equation}\label{reprenew}
\rho^\nabla(x)\xi:=-\nabla^\ast_x\xi=-\xi\circ\nabla_x,~~x\in\h, \xi\in\h^\ast.
\end{equation}
Define $Z^2_\nabla(\h,\h^\ast)=Z^2_{\rho^\nabla}(\h,\h^\ast)$.  Every cocycle $\alpha\in Z^2_\nabla (\h, \h^\ast)$ thus gives rise to
a Lie algebra extension
\begin{center}
$0\longrightarrow\h^\ast\longrightarrow
\G_{\nabla,\alpha}\longrightarrow\h\longrightarrow 0$,
\end{center}
where the non-zero Lie brackets $[~ , ~]_{\G}$ for $\G = \G_{\nabla,\alpha}$ are defined on the vector
space direct sum $\G = \h \oplus \h^\ast$ by the formulas
\begin{equation}\label{Liebracket1}
[x,y]_\G = [x,y]_\h+\alpha(x,y),~~\text{for~all}~~x,y\in\h,
\end{equation}
\begin{equation}\label{Liebracket2}
[x,\xi]_\G =\rho(x)\xi, ~~\text{for~all}~x\in\h, \xi\in\h^\ast.~~~~~~
\end{equation}
We let $\omega_0$ be the non-degenerate alternating two-form on $\G$, which is defined
by the dual pairing of $\h$ and $\h^\ast$. i.e.,
\begin{equation}\label{formlagra}
\begin{cases}
	\omega_0(x,\xi)=\xi(x),\quad \text{ for all}~ x,\in \h, \xi\in\h^\ast,\\
	\omega_0(\xi_1,\xi_2)=0,\quad \text{ for all}~ \xi_1,\xi_2\in \h^\ast,\\
	\omega_0(x,y)=0,\quad \text{ for all}~ x,y\in \h.
	\end{cases}
\end{equation}

\begin{pr}
The form $\omega_0$ is symplectic for the Lie-algebra $\G_{\nabla,\alpha}$, if and only if
\begin{equation}\label{symexten}
\alpha(x,y)(z)+\alpha(y,z)(x)+\alpha(z,x)(y)=0,
\end{equation}
for all $x,y,z\in\h$.
\end{pr}
The condition $(\ref{symexten})$, known as the ``Bianchi identity".
\begin{Def}
 We call the symplectic Lie algebra $(\G_{\nabla,\alpha},\omega_0)$ the Lagrangian extension of the
flat Lie algebra $(\h, \nabla)$ with respect to $\alpha$.
\end{Def}

\begin{Le}\label{isomor1}
Two Lagrangian extensions over $\h$ with the same class in $H^2_\rho(\h,\h^\ast)$ are isomorphic. 
\end{Le}
\begin{proof}
If $[\alpha]=[\beta]\in H^2_\rho(\h,\h^\ast)$; that is, if $\beta = \alpha -\partial\sigma$, for some $\sigma\in\mathcal{C}^1(\h,\h^\ast)$. The following map
\begin{equation}
\Phi:\G_{\nabla,\alpha}\longrightarrow\G_{\nabla,\beta},~(x,\xi)\longmapsto (x,\xi+\sigma(x))
\end{equation}
is the required isomorphism of Lie algebras.  Using $(\ref{Liebracket1})$ and $(\ref{Liebracket2})$, we have
\begin{eqnarray}\label{eq1}
&&\Phi\Big([x+\xi_1,y+\xi_2]_{\G_{\nabla,\alpha}}\Big)=[x,y]_\h+\alpha(x,y)+\rho(x)\xi_2-\rho(y)\xi_1+\sigma([x,y]_\h)
\end{eqnarray}
and
\begin{eqnarray}\label{eq2}
&&[\Phi(x+\xi_1),\Phi(y+\xi_2)]_{\G_{\nabla,\beta}}=[x,y]_\h+\beta(x,y)+\rho(x)\xi_2-\rho(y)\xi_1+\rho(x)\sigma(y)-\rho(y)\sigma(x).
\end{eqnarray}
From $(\ref{eq1})$ and $(\ref{eq2})$, we obtain
\begin{eqnarray*}
&&\Phi\Big([x+\xi_1,y+\xi_2]_{\G_{\nabla,\alpha}}\Big)-[\Phi(x+\xi_1),\Phi(y+\xi_2)]_{\G_{\nabla,\beta}}=\alpha(x,y)-\beta(x,y)-\partial\sigma(x,y)=0.
\end{eqnarray*}
\end{proof}

\begin{Def}\textsc{\cite{BV}} A strong polarization of a symplectic Lie algebra $(\G,\omega)$ is a pair $(\mathfrak{a},N)$ consisting of a Lagrangian ideal $\mathfrak{a} \subset \G$ and
a complementary Lagrangian subspace $N \subset \G$. 
 The quadruple $(\G,\omega, \mathfrak{a},N)$ is then called a strongly polarized symplectic Lie algebra. 

An isomorphism
of strongly polarized symplectic Lie algebras $(\G_1,\omega_1, \mathfrak{a}_1,N_1) \rightarrow (\G_2,\omega_2, \mathfrak{a}_2,N_2)$ is an isomorphism of symplectic Lie algebras $(\G_1,\omega_1) \rightarrow (\G_2,\omega_2)$ which maps the strong polarization $(\mathfrak{a}_1,N_1)$ to the strong polarization $(\mathfrak{a}_2,N_2)$.

\end{Def}
\subsection{Change of flat Lie algebra}  See $\cite{BV}$, Theorem $4.2.1$  for more details regarding this small paragraph. Let $(\G_{\nabla,\alpha},\omega_0)$ be the Lagrangian extension of the flat Lie algebra $(\h, \nabla)$ with respect to $\alpha$. Let $\omega$ be a symplectic form on $\G_{\nabla,\alpha}$,  chosen arbitrarily, $\mathfrak{a}$ is a Lagrangian ideal of $(\G_{\nabla,\alpha},\omega)$ and $(\B=\G_{\nabla,\alpha}/\mathfrak{a},\nabla^\omega)$ the associated quotient flat Lie algebra.

Let $(\G_{\nabla,\alpha},\omega, \mathfrak{a},N)$ be a strongly polarized symplectic Lie algebra, $(\B, \nabla^\omega)$ its associated quotient flat Lie algebra and $\pi_\mathfrak{a} : \G_{\nabla,\alpha}\longrightarrow\mathfrak{a}$ be the projection map  induced by the strong
polarization $\G =\mathfrak{a}\oplus N$. According to Proposition \ref{Pr1isomo}, the map $\Gamma : \mathfrak{a}\longrightarrow\B^\ast,~x\mapsto \omega(x,.)$, serves as the identification of $\mathfrak{a}$ with $\B^\ast$, and this identification is induced by the symplectic form $\omega$. From Equation $(\ref{Pr1connex})$, we have
\begin{eqnarray*}
\rho(x)\circ\Gamma=\Gamma\circ
\mathrm{ad}_{\B,\mathfrak{a}}(x),
\end{eqnarray*}
$\mathrm{ad}_{\B,\mathfrak{a}}(x)(a)=[\tilde{x},a]$, for all $x\in\B$, $a\in\mathfrak{a}$. Furthermore, it demonstrates that the representation $\rho$ of $\B$ on $\B^\ast$, which belongs to the flat connection $\nabla^\omega$ by $(\ref{reprenew})$, corresponds to  $\mathrm{ad}_{\B,\mathfrak{a}}$.

Let $\pi_\B : N \longrightarrow \B$ be the isomorphism of vector spaces induced by the quotient map
$\G_{\nabla,\alpha}\longrightarrow\B$. The isomorphisms 
$\Gamma$ and $\pi_\B$ combine to form an isomorphism

\begin{equation}
\pi_\B\oplus\Gamma : N\oplus\mathfrak{a} \longrightarrow \B\oplus\B^\ast.
\end{equation}

Define
\begin{equation}
\tilde{\beta} = \Gamma\circ\tilde{\alpha} \in Z_{\nabla^{\omega}}(\B, \B^\ast) 
\end{equation}
to be the push-forward of $\tilde{\alpha}$, where $\tilde{\alpha}(x,y)=\pi_\B([\tilde{x},\tilde{y}])$, for all $x,y\in\B$. The map

\begin{equation}
\pi_\B\oplus\Gamma : (\G_{\nabla,\alpha},\omega) \longrightarrow (\G_{\nabla^\omega,\tilde{\beta}},\omega_{\tilde{\beta}})
\end{equation}
defines an isomorphism of symplectic Lie algebras.

 Let $(\G_{\nabla^{\omega},\beta},\omega_\beta)$ be the Lagrangian extension of the flat Lie algebra $(\B, \nabla^{\omega})$ with respect to $\beta\in Z^2_{\nabla^{\omega}}(\B,\B^\ast)$. Together with Lemma $\ref{isomor1}$, we therefore have:

\begin{Le}\label{pont}
The map 
\begin{center}
$\Lambda:=\pi_\B\oplus\Gamma : (\G_{\nabla,\alpha},\omega) \longrightarrow (\G_{\nabla^\omega,\beta},\omega_\beta),
$ 
\end{center}
defines an isomorphism of symplectic Lie algebras if and only if $\beta$ and $\tilde{\beta}$ have the same class in $H^2_{\rho^{\nabla^\omega}}(\B,\B^\ast)$. In particular, $\G_{\nabla,\alpha}$ and $\G_{\nabla^\omega,\tilde{\beta}}$ are isomorphic.
\end{Le}

A Lagrangian extension  of flat Lie algebra can generally take several isomorphic forms, i.e., $\G=\h_1\oplus\h_1^\ast \cong\h_2\oplus\h_2^\ast \cdot\cdot\cdot\cong \h_k\oplus\h_k^\ast $, and because of this, it is possible to have an isomorphim between two Lagrangian extensions over
  different flat Lie algebras, as shown in the following example.
  \begin{exem}
  Let $(\mathfrak{h}, \nabla^{\pm})$ be a three-dimensional flat Lie algebra, and let $\nabla^{\pm}$ be the connection on $\mathfrak{h}$ with non-zero products in the basis are
 \begin{equation}
 \nabla_{e_1}e_j=e_j, j=1,2,3,~\nabla_{e_2}e_1=e_2,~\nabla_{e_2}e_2= \pm e_3,~\nabla_{e_3}e_1=e_3,~\nabla_{e_3}e_2= e_2,~\nabla_{e_3}e_3= 2e_3.
 \end{equation}
 It is straightforward to verify that $(\mathfrak{h}, \nabla^{-})\ncong (\mathfrak{h}, \nabla^{+})$, $H^2_{\rho^{\nabla^{\pm}}}(\h,\h^\ast)=0$ and $\G_{\nabla^{-}}\cong \G_{\nabla^{+}}$; see Table $\ref{isoTur}$ $($algebras $\G_{6,7}$ and $\G_{6,8})$.

  \end{exem}
\begin{pr}\label{LEiso}
Let $(\G_{\nabla,\alpha},\omega)$ $($resp. $(\tilde{\G}_{\tilde{\nabla},\tilde{\alpha}},\tilde{\omega})$ $)$ be the Lagrangian extension of the flat Lie algebra $(\h,\nabla)$ $($resp. $(\tilde{\h},\tilde{\nabla})$ $)$ with respect to $\alpha$ $($resp. $\tilde{\alpha})$. Assume that there exists $(\Omega,\mathfrak{j})$ in $\G_{\nabla,\alpha}$ such that $\B=(\G_{\nabla,\alpha}/\mathfrak{j},\nabla^\Omega)\cong(\tilde{\h},\tilde{\nabla})$. Then, $\G_{\nabla^\Omega,\beta}:=\B\oplus\B^\ast$ is isomorphic to $\tilde{\G}_{\tilde{\nabla},\tilde{\alpha}}$ with $\beta\in Z^2_{\nabla^\Omega}(\B,\B^\ast)$. In particular, $\G_{\nabla,\big(\Lambda_\ast^{-1}\big)(\widehat{\beta})}$ and $\tilde{\G}_{\tilde{\nabla},\widehat{\alpha}}$ are isomorphic, where $\widehat{\beta}\in Z^2_{L,\nabla^\Omega}(\B,\B^\ast)$, such that $\widehat{\beta}(x,y)=\Gamma\circ\pi_\B([\tilde{x},\tilde{y}])$, and $\widehat{\alpha}=(\Lambda\circ\Psi)_\ast
\big(\widehat{\beta}\big)$.

\end{pr}
\begin{proof}
Let $(\Omega,\mathfrak{j})$ in $\G_{\nabla,\alpha}$, where $\Omega$ is a symplectic form on $\G_{\nabla,\alpha}$, and $\mathfrak{j}$ is a Lagrangian ideal of $(\G_{\nabla,\alpha},\Omega)$. Put $\B=(\G_{\nabla,\alpha}/\mathfrak{j},\nabla^\Omega)$, and assume that $(\B,\nabla^\Omega)$ and $(\tilde{\h},\tilde{\nabla})$ are isomorphic, i.e., there exists an isomorphism $\psi : \B\longrightarrow\tilde{\h}$ such that $\psi\big(\nabla^\Omega_xy)=
\tilde{\nabla}_{\psi(x)}\psi(y)$ for all $x,y\in\B$. Let $(\G_{\nabla^\Omega,\beta},\omega)$ be the Lagrangian extension of the flat Lie algebra $(\B,\nabla^\Omega)$ with respect to $\beta\in Z^2_{\nabla^\Omega}(\B,\B^\ast)$. Choose a strong polarization $(\G_{\nabla^\Omega,\beta},\omega,\mathfrak{a},N)$, then there exists $\beta_1 = \beta_{1(\G_{\nabla^\Omega,\beta},\omega, \mathfrak{a},N)}\in Z^2_{\nabla^\Omega}(\B, \B^\ast)$
satisfying (\ref{symexten}), such that $(\G_{\nabla^\Omega,\beta},\omega, \mathfrak{a},N)$ is isomorphic to $F(\B, \nabla^\Omega,\beta_1)$ (see \cite{BV}, Theorem $4.2.1$). Let $(\tilde{\G}_{\tilde{\nabla},\tilde{\alpha}},\tilde{\omega}, \tilde{\mathfrak{a}},\tilde{N})$ be a strong polarization of the symplectic
Lie algebra $(\tilde{\G}_{\tilde{\nabla},\tilde{\alpha}},\tilde{\omega})$, for which there exists  $\alpha_1\in Z^2_{L,\tilde{\nabla}}(\tilde{\h},\tilde{\h}^\ast)$ such that $\alpha_1:=\alpha_{1(\tilde{\G}_{\tilde{\nabla},\tilde{\alpha}},\tilde{\omega},\mathfrak{a},N)}=\psi_\ast(\beta_1)$ (push-forward of $\beta_1$). Using $(\ref{Liebracket1})$ and $(\ref{Liebracket2})$ it is easily verified that 
the map 
\begin{equation}\label{Iso18}
\Psi:\big(\G_{\nabla^{\Omega},\beta_1},\omega\big)\longrightarrow\big(\tilde{\G}_{\tilde{\nabla},\alpha_1},\tilde{\omega}\big),~(x,\eta)\longmapsto (\psi(x),\psi_\ast\eta)
\end{equation}
is the required isomorphism of symplectic Lie algebras.  Since $\psi$ is an isomorphism, then for any $\tilde{\alpha}\in Z^2_{L,\tilde{\nabla}}(\tilde{\h},\tilde{\h}^\ast)$, there exists $\beta\in Z^2_{L,\nabla^\Omega}(\B,\B^\ast)$ such that $\tilde{\alpha}=\psi_\ast\beta$. This implies that $\G_{\nabla^\Omega,\beta}$ and $\tilde{\G}_{\tilde{\nabla},\tilde{\alpha}}$ are isomorphic. According to Lemma $\ref{pont}$, the map 

\begin{equation}
\Lambda\circ\Psi: \big(\G_{\nabla,\big(\Lambda_\ast^{-1}\big)(\widehat{\beta})},\omega\big)\longrightarrow\big(\tilde{\G}_{\tilde{\nabla},\widehat{\alpha}},\tilde{\omega}\big),
\end{equation} 
is the required isomorphism of symplectic Lie algebras, where $\widehat{\beta}\in Z^2_{L,\nabla^\Omega}(\B,\B^\ast)$, such that $\widehat{\beta}(x,y)=\Gamma\circ\pi_\B([\tilde{x},\tilde{y}])$, and $\widehat{\alpha}=(\Lambda\circ\Psi)_\ast
\big(\widehat{\beta}\big)$.

\end{proof}

\begin{Le}\label{Remofiso}
 If two Lagrangian extensions $(\G_{\nabla,\alpha},\omega)$ and $(\G^\prime_{\nabla^\prime,\alpha^\prime},\omega^\prime)$ over $(\h,\nabla)$ $($resp. $(\h^\prime,\nabla^\prime))$ are isomorphic, then there exist a symplectic form  $\Omega$ and a Lagrangian ideal $\mathfrak{a}$ of $(\G_{\nabla,\alpha},\Omega)$ such that $(\G_{\nabla,\alpha}/\mathfrak{a},\nabla^{\Omega})$   and $(\h^\prime,\nabla^\prime)$ are isomorphic.
 
 \end{Le}
 \begin{proof}
 Let $\Phi : (\G_{\nabla,\alpha},\omega)\longrightarrow(\G^\prime_{\nabla^\prime,\alpha^\prime},\omega^\prime)$ be an isomorphism of symplectic Lie algebras, and $\Phi^{-1}_{\h^\prime} : \h^\prime =\G_{\nabla^\prime,\alpha^\prime}^\prime/\h^{\prime\ast}\longrightarrow \G_{\nabla,\alpha}/\mathfrak{a}$ the induced
map on quotients, where $\mathfrak{a}=\Phi^{-1}(\h^{\prime\ast})$. Then $\nabla^{\prime\prime}=(\Phi^{-1}_{\h^\prime})_{\ast}\nabla^\prime$   is the
associated quotient flat connection on $\G_{\nabla,\alpha}/\mathfrak{a}$ and $\Omega=(\Phi^{-1})^\ast\omega^\prime$
 is a sympelectic form on $\G_{\nabla,\alpha}$ that has $\mathfrak{a}$ as a Lagrangian ideal such that $(\G_{\nabla,\alpha}/\mathfrak{a},\nabla^{\Omega})\cong (\h^\prime,\nabla^\prime)$.
 \end{proof}

\begin{remark}
If $(\G_{\nabla,\alpha},\omega)$ has a unique Lagrangian extension, that is, $\G_{\nabla,\alpha}=\h\oplus\h^\ast$ and $(\h,\nabla)$ is the unique induced flat Lie algebra. Then $\G_{\nabla,\alpha}$ can not be isomorphic to any Lagrangian extension $($Uniqueness of the reconstruction$)$. There may be other Lagrangian extensions that are isomorphic to this one.
\end{remark}

We now recall the following cohomology group, which is a description of all Lagrangian extensions of $\h$ with corresponding flat Lie algebras $(\h, \nabla)$:

First, we define Lagrangian one- and two-cochains on $\h$ as
\begin{eqnarray*}
\mathcal{C}^1_L(\h,\h^\ast)&=&\Big\{\phi\in\mathcal{C}^1(\h,\h^\ast)~|~\phi(x)(y)-\phi(y)(x)=0,~\text{for all}~x,y\in\h\Big\}\\
\mathcal{C}^2_L(\h,\h^\ast)&=&\Big\{\alpha\in\mathcal{C}^2(\h,\h^\ast)~|~\sum_{cycl}\alpha(x,y)(z)=0,~\text{for all}~x,y,z\in\h\Big\}
\end{eqnarray*}
Then, as well, let $\rho = \rho^\nabla$ be the representation of $\h$ on $\h^\ast$ associated to $\nabla$, as defined in
(\ref{reprenew}). Denote by $\partial=\partial_\rho^i$ the corresponding coboundary operators for cohomology with $\rho$-coefficients. The coboundary operator $\partial : \mathcal{C}^1(\h,\h^\ast) \longrightarrow \mathcal{C}^2(\h,\h^\ast)$ maps the subspace $\mathcal{C}^1_L(\h,\h^\ast)$ into $\mathcal{C}^2_L(\h,\h^\ast)\cap Z^2_\rho(\h,\h^\ast)$ (see, Lemma 4.4.2 \cite{BV}).

Let $Z^2_{L,\rho}(\h, \h^\ast) = \mathcal{C}_L^2(\h, \h^\ast) \cap Z_\rho^2(\h, \h^\ast)$ denote the space of Lagrangian cocycles. We now define the Lagrangian extension cohomology group for the flat Lie algebra
$(\h, \nabla)$ as
\begin{equation*}
H^2_{L,\rho}(\h,\h^\ast)=\frac{Z^2_{L,\rho}(\h, \h^\ast)}{\partial\mathcal{C}^1_L(\h,\h^\ast)}.
\end{equation*}
A Natural map exists from $H^2_{L,\rho}(\h,\h^\ast)$ to the ordinary Lie algebra cohomology group $H^2_{\rho}(\h,\h^\ast)$.  This map does not have to be injective, in general,  see Table \ref{tableofcoho}.   The kernel $\kappa_L$ of the natural map 
\begin{equation*}
H^2_{L,\rho}(\h,\h^\ast)\longrightarrow H^2_{\rho}(\h,\h^\ast)
\end{equation*}
is given by
\begin{equation}
\kappa_L=\frac{B^2_{\rho}(\h,\h^\ast)\cap Z^2_{L,\rho}(\h,\h^\ast)}{B^2_{L,\rho}(\h,\h^\ast)},
\end{equation}
where $B^2_\rho(\h,\h^\ast) =\{ \partial\mu ~|~ \mu\in\mathrm{Hom}(\h, \h^\ast)\}$ is the set of ordinary two-coboundaries with $\rho$-coefficients and $B_{L, \rho}^2(\h, \h^\ast) = \{\partial\mu~ |~ \mu \in\mathcal{C}^1_L(\h, \h^\ast)\}$ is the set of two-coboundaries
for Lagrangian extension cohomology.

Similar to Lemma \ref{isomor1}, we can state:
\begin{Le}\label{isofromLrhotorho} Let $[\alpha]\in H^2_\rho(\h,\h^\ast)$. The map
\begin{equation}\label{isooftwodiffcoho}
\Psi: \G_{\nabla,\alpha^\prime}\longrightarrow \G_{\nabla,\beta},~~(x,\xi)\longmapsto (x,\xi+(\sigma-\mu)(x))
\end{equation}
 is the required isomorphism of Lie algebras, where $\beta=\alpha-\partial\sigma$ and $\alpha^\prime=\alpha-\partial\mu$, such that, $\sigma\in\mathcal{C}^1(\h,\h^\ast)$ and  $\mu\in\mathcal{C}^1_L(\h,\h^\ast)$.
\end{Le}

The following shows that the group $H^2_{L,\rho}(\h,\h^\ast)$ corresponds one-to-one to the isomorphism classes of Lagrangian extensions over a flat Lie algebra.

\begin{pr}\textsc{\cite{BV}}\label{1-1 corress}
The correspondence which associates to a symplectic Lie algebra with Lagrangian ideal $(\G, \omega, \mathfrak{j})$ the extension triple $(\h, \nabla, [\alpha]_{L,\rho})$ induces a bijection
between isomorphism classes of symplectic Lie algebras with Lagrangian ideal and isomorphism classes of flat Lie algebras with symplectic extension cohomology class.
\end{pr}
If two symplectic Lie algebras are probably isomorphic, finding isomorphism between them requires that we compare their flat structure and Lagrangian extension cohomology group to ensure that they are isomorphic.

 Let $(\G,\omega)$ be the Lagrangian extension of the flat Lie algebra $(\h,\nabla)$ with respect to $\alpha\in Z^2_{\nabla}(\h,\h^\ast)$, and assume that there exists a pair $(\tilde{\omega},\tilde{\mathfrak{a}})$ consisting of  a symplectic form $\tilde{\omega}$ on $\G$ and  a Lagrangian ideal $\tilde{\mathfrak{a}}$ of $(\G,\tilde{\omega})$. Let $(\tilde{\G},\omega)$ be the Lagrangian extension of the flat Lie algebra $(\tilde{\h},\nabla^{\tilde{\omega}})$ with respect to $\tilde{\alpha}\in Z^2_{\nabla^{\tilde{\omega}}}(\tilde{\h},\tilde{\h}^\ast)$, where $\tilde{\h}:=\G/\tilde{\mathfrak{a}}$.

 We have the following$:$
\begin{pr}\label{diffiso}
Two extensions $(\G_{\nabla,\alpha},\omega, \mathfrak{a})$ and $(\G^\prime_{\nabla^\prime,\alpha^\prime},\omega^\prime, \mathfrak{a}^\prime)$ over $(\h,\nabla)$ $($resp. $(\h^\prime,\nabla^\prime))$ are isomorphic if and only
if    the cohomology class
\begin{equation}
[(\Psi_{\tilde{\h}} )_\ast
\tilde{\alpha}-\alpha^\prime]\in H^2_{L,\rho}(\h^{\prime},\h^{\prime\ast})
\end{equation}
vanishes, where $\tilde{\alpha}(x,y)=\Gamma\circ\pi_{\tilde{\h}}([\tilde{x},\tilde{y}])$, for all $x,y\in\tilde{\h}$ and $\Psi : \tilde{\G}_{\nabla^{\tilde{\omega}},\tilde{\alpha}}=\tilde{\h}\oplus\tilde{\h}^\ast\longrightarrow
\G^\prime_{\nabla^\prime,\alpha^\prime}$ is an isomorphism of symplectic Lie algebras which is given in  Proposition $\ref{LEiso}$.

\end{pr}
\begin{proof}
According to Lemma $\ref{Remofiso}$, there exists $(\tilde{\omega},\tilde{\mathfrak{a}})$ in $\G$ such that $(\tilde{\h},\nabla^{\tilde{\omega}})$   and $(\h^\prime,\nabla^\prime)$ are isomorphic. Let $(\tilde{\G},\omega, \check{\mathfrak{a}})$ be a Lagrangian extension over $(\tilde{\h}, \nabla^{\tilde{\omega}})$. Take a strong polarization $(\tilde{\G},\omega, \check{\mathfrak{a}},N)$, and set $[\tilde{\alpha}_{(\G,\omega, \check{\mathfrak{a}})}]=
[\tilde{\alpha}_{(\G,\omega, \check{\mathfrak{a}},N)}]$, where $\tilde{\alpha}_{(\G,\omega,\check{\mathfrak{a}},N)}\in Z^2_{L,\nabla}(\tilde{\h},\tilde{\h}^\ast)$  (Due to the fact that this cohomology class is independent of the choice of strong polarization $N$). 

Let us suppose $\Psi : (\tilde{\G}_{\nabla^{\tilde{\omega}},\tilde{\alpha}},\omega,\check{\mathfrak{a}})\longrightarrow(\G^\prime_{\nabla^\prime,\alpha^\prime},\omega^\prime,\mathfrak{a}^\prime)$ is an isomorphism. Choose for $(\tilde{\G}_{\nabla^{\tilde{\omega}},\tilde{\alpha}},\omega,\check{\mathfrak{a}})$ a complementary Lagrangian subspace $N$.  Thus, $\Psi : (\tilde{\G}_{\nabla^{\tilde{\omega}},\tilde{\alpha}},\omega,\check{\mathfrak{a}},N)\longrightarrow(\G^\prime,\omega^\prime,\mathfrak{a}^\prime=\Psi(\check{\mathfrak{a}}),N^\prime=\Psi(N))$  is an isomorphism of strongly polarized symplectic Lie algebras, and we have $\alpha^\prime_{(\G^\prime,\omega^\prime,\mathfrak{a}^\prime,N^\prime)}=(\Psi_{\tilde{\h}}\big)_\ast\tilde{\alpha}_{(\tilde{\G},\omega,\check{\mathfrak{a}},N)}$.  In this sense, isomorphic extensions over isomorphic two flat Lie algebras possess the same cohomology class. It remains to prove that two extensions over isomorphic two flat Lie algebras of the same class are isomorphic. As shown by Proposition $\ref{LEiso}$, using the isomorphism of $(\ref{Iso18})$, this is true.

\end{proof}
\begin{exem}
Consider the Lie algebra $\G_{6,2}^b=\h\oplus\h^\ast$ given by Proposition $\ref{Prclassifi}$. To demonstrate that this algebra is isomorphic to $\mathfrak{r}_2\mathfrak{r}_2\oplus\mathfrak{aff}(1,\R)$ when $b\in(-\frac{4}{27},0)$ and to $\mathfrak{r}_2^\prime\oplus\mathfrak{aff}(1,\R)$ when $b\in(-\infty,-\frac{4}{27})\cup(0,+\infty)$, it is difficult (using algebraic isomorphisms). In order to establish these isomorphisms, we will apply Proposition $\ref{diffiso}$. Let
\begin{center}
 $(\Psi_\h):=\small\left( \begin {array}{ccc} a&2\,{a}^{2}+a&-1\\ \noalign{\medskip}
\sqrt { 3\,a^2+2\,a }& \left( -2\,a-1 \right) \sqrt {3\,{a}
^{2}+2\,a}&0\\ \noalign{\medskip}-2\,a-1&-4\,{a}^{2}-2\,a&-1
\end {array} \right),$
 \end{center} 
with $b=2a(2a+1)^2$ and $a\in(-\infty,-\frac{2}{3})\cup(0,+\infty)$. On the other hand, for any symplectic form $\omega$ on $\G^\prime=\mathfrak{r}_2^\prime\oplus\mathfrak{aff}(1,\R)$, $\mathfrak{a}=\langle e_3,e_4,e_6\rangle$ is a Lagrangian ideal of $(\G^\prime,\omega)$. It is easy to see that $\Psi_\h : (\h,\nabla)\longrightarrow(\G^\prime/\mathfrak{a},\nabla^\omega)$ defines an isomorphism of flat Lie algebras.

\begin{center}
$\Psi_\h:=\small\left( \begin {array}{ccc} {\frac {
\sqrt {-3\,{a}^{2}-2\,a+1}\,-a-1}{2}}&\frac{{a}^{2} \left( \sqrt {-3\,{a}^{2}-2\,a+1
}\,+3\,a+3 \right) }{  -\sqrt {-3\,{a}^{2}-2\,a+1}\,+3\,a+1  }&-1\\ \noalign{\medskip}\frac{a \left( \sqrt {-3\,{a}^{2}-2\,a+1}\,+3\,a+
3 \right) }{  \sqrt {-3\,{a}^{2}-2\,a+1}\,-3\,a-1 }&\frac{a
 \left( 2\,a\,\sqrt {-3\,{a}^{2}-2\,a+1} \,+\sqrt {-3\,{a}^{2}-2\,a+1}\,-a-1
 \right)}{   \sqrt {-3\,{a}^{2}-2\,a+1}\,-3\,a-1 }&-1
\\ \noalign{\medskip}a&-{a}^{2}-a&-1\end {array} \right)$,
\end{center}
with $b=-a^3-a^2$ and $a\in(-1,\frac{1}{3})$.  Similarly, let $\omega$ be an arbitrary symplectic form on $\G^\prime=\mathfrak{r}_2\mathfrak{r}_2\oplus\mathfrak{aff}(1,\R)$, the subspace $\mathfrak{a}=\langle e_2,e_4,e_6\rangle$ is a Lagrangian ideal of $(\G^\prime,\omega)$. It is now easily verified that the map $\Psi_\h :(\h,\nabla)\longrightarrow(\G^\prime/\mathfrak{a},\nabla^\omega)$ defines an isomorphism of flat Lie algebras.
\end{exem}
\begin{Le}
Let $(\G_{\nabla,\alpha},\omega_0)$ be the Lagrangian extension of the flat Lie algebra $(\h,\nabla)$ with respect to $\alpha$. Then, every symplectic form $\omega$ on $\G_{\nabla,[\alpha]_\rho}$ that has  a Lagrangian ideal $\mathfrak{j}=\h^\ast$ and such that its associated flat torsion-free connection $\nabla^\omega$ isomorphic to $\nabla$,  is symplectomorphically equivalent to 
 \begin{eqnarray*}
 \omega_{[\alpha]_{L,\rho}}(x,y+\xi)=-\chi(x,y)\oplus\omega_0(x,\xi),~x\in\h, \xi\in\h^\ast,
 \end{eqnarray*}
 where $\chi=(\sigma-\mu)- {}^t\left(\sigma-\mu\right)\in\bigwedge^2\h^\ast$ for some $\mu\in\mathcal{C}^1(\h,\h^\ast)$ and $\sigma\in\mathcal{C}^1_L(\h,\h^\ast)$.

\end{Le}
\begin{proof}
Putting $\beta=\alpha-\partial\sigma$ and $\alpha^\prime=\alpha-\partial\mu$, where $\mu\in\mathcal{C}^1(\h,\h^\ast)$ , $\sigma\in\mathcal{C}^1_L(\h,\h^\ast)$ and $\alpha\in Z^2_{L,\rho}(\h,\h^\ast)$.  Two strongly polarized symplectic Lie algebras $F(\h, \nabla, \alpha_1)$ and $F(\h, \nabla, \alpha_2)$ give rise to isomorphic extensions over $\h$ if $[\alpha_1] = [\alpha_2] \in H_{L,\rho}^2(\h, \h^\ast)$; that is, if $ \alpha_2 = \alpha_1-\partial\mu$, for some $\mu \in\mathcal{C}^1_L(\h, \h^\ast)$. Similarly to Lemma $\ref{isomor1}$, it is easily verified that then the map
\begin{equation}
\Phi: \G_{\nabla,\alpha_2}\longrightarrow \G_{\nabla,\alpha_1},~~(x,\xi)\longmapsto (x,\xi+\mu(x))
\end{equation}
is the required isomorphism of Lagrangian extensions over $\h$. Let $\beta=\alpha_1$, according to Lemma \ref{isofromLrhotorho}, $\G_{\nabla,\beta}$ and $\G_{\nabla,\alpha_2}$ are isomorphic (as Lie algebras), we have then
\begin{align*}
\omega_{[\alpha]_{L,\rho}}(x,y+\xi)&=\omega_0\big(\Phi(x),\Phi(y+\xi)\big)\\
&=\omega_0\big(x+(\sigma-\mu)(x),y+\xi+(\sigma-\mu)(y)\big)\\
&=\xi(x)+\omega_0\big(x,(\sigma-\mu)(y)\big)+\omega_0\big((\sigma-\mu)(x),y\big)\\
&=\xi(x)+\omega_0\big((\sigma-\mu)(x)+(\sigma-\mu)^\ast(x),y\big)\\
&=-\chi(x,y)+\xi(x),
\end{align*}
where, $\chi(x,y)=(\sigma-\mu)(x)(y)+(\sigma-\mu)^\ast(x)(y)$, and  $(\sigma-\mu)^\ast=-^t (\sigma-\mu)$ .
\end{proof}

\subsection{Exact Lagrangian extension}
\begin{pr}\label{exactLagrang}
Let $(\G_{\nabla,\alpha},\omega_0)$ be a Lagrangian extension of 
flat Lie algebra $(\h, \nabla)$ with respect to $\alpha$. Then, $\omega_0$ is exact if and only if 
\begin{enumerate}
\item $(\h,\nabla)$ has a right-identity element ``$e$".
\item $\alpha(x,y)(e)=\gamma\big([x,y]\big)$, for all $x,y\in\h$~ for some $\gamma\in\h^\ast$.
\end{enumerate}
Moreover, the flat torsion-free connection ``$\ast$" on $(\G_{\nabla,\alpha},\omega_0)$ is given by
\begin{align*}
x\ast y&=\nabla_xy+\alpha(x,\cdot)(y),\\
\xi_1\ast\xi_2&=0,\\
\xi\ast x&=\xi\circ(\nabla_x-\mathrm{ad}_x),\\
x\ast\xi&=-\xi\circ\mathrm{ad}_x,\end{align*}
for all $x,y\in\h$~ and~ $\xi,\xi_1,\xi_2\in\h^\ast$. 
\end{pr}
\begin{proof}
It is known that (see \cite{Y} and \cite{MR}) given a symplectic Lie algebra $(\G_{\nabla,\alpha},\omega_0)$, the product given by
\begin{equation*}
\omega_0\big((x+\xi_1)\ast(y+\xi_2),z+\xi\big)=-\omega_0\big(y+\xi_2,[x+\xi_1,z+\xi]\big),~~\forall x,y,z\in\h,~\xi,\xi_1,\xi_2\in\h^\ast
\end{equation*}
induces a flat torsion-free connection $``\ast"$ on $\G_{\nabla,\alpha}$ that satisfies \[ (x+\xi_1)\ast(y+\xi_2)-(y+\xi_2)\ast(x+\xi_1)=[x+\xi_1,y+\xi_2] .\]

Let $x,y,z\in\h$~ and ~$\xi,\xi_1,\xi_2\in\h^\ast$ , we have
\begin{equation*}
x\ast y= (x\ast y)_{|\h}+(x\ast y)_{|\h^\ast}.
\end{equation*}
Then,
\begin{align}\label{F1}
\omega_0\big((x\ast y)_{|\h},z\big)&=-\omega_0\big(y,[x,z]\big)\nonumber\\
&=-\omega_0\big(y,[x,z]_{\h}+\alpha(x,z)\big)\\
&=-\alpha(x,z)(y)\nonumber
\end{align}
and
\begin{align}\label{F2}
\omega_0\big((x\ast y)_{|\h^\ast},z\big)&=-(x\ast y)_{|\h^\ast}(z).
\end{align}
On the other hand,
\begin{eqnarray}\label{F3}
\omega_0\big((x\ast y)_{|\h^\ast},\xi\big)=-\omega_0\big(y,[x,\xi]\big)=\omega_0\big(y,\xi\circ\nabla_x\big)=
\xi\big(\nabla_xy\big).
\end{eqnarray}
From $(\ref{F1})$, $(\ref{F2})$ and $(\ref{F3})$, we obtain
\begin{align*}
x\ast y&=\nabla_x y+\alpha(x,\cdot)(y),~\forall x,y\in\h.
\end{align*}
Now we  define 
\begin{align*}
\xi_1\ast\xi_2&=(\xi_1\ast \xi_2)_{|\h}+(\xi_1\ast \xi_2)_{|\h^\ast}.
\end{align*}
Then,
\begin{equation}\label{F4}
(\xi_1\ast \xi_2)_{|\h^\ast}(x)=
-\omega_0\big((\xi_1\ast \xi_2)_{|\h^\ast},x\big)=\omega_0\big(\xi_2,[\xi_1,x]\big)=0
\end{equation}
and
\begin{equation}\label{F5}
\xi\big((\xi_1\ast \xi_2)_{|\h}\big)=
\omega_0\big((\xi_1\ast \xi_2)_{|\h},\xi)=-\omega_0\big(\xi_2,[\xi_1,\xi]\big)=0.
\end{equation}
From $(\ref{F4})$ and $(\ref{F5})$, we have $\xi_1\ast\xi_2=0$ for all $\xi_1,\xi_2\in\h^\ast$.

Let us consider,
\begin{align*}
\xi\ast x&=(\xi\ast x)_{|\h}+(\xi\ast x)_{|\h^\ast}.
\end{align*}
Then, we have
\begin{eqnarray}\label{F6}
\xi_1\big((\xi\ast x)_{|\h}\big)=\omega_0\big((\xi\ast x)_{|\h},\xi_1\big)=-\omega_0\big(x,[\xi,\xi_1]\big)=0
\end{eqnarray}
and
\begin{eqnarray}\label{F7}
\begin{split}
(\xi\ast x)_{|\h^\ast}(y)=-\omega_0\big((\xi\ast x)_{|\h^\ast},y\big)=\omega_0\big(x,[\xi,y]\big)=\xi\circ(\nabla_x-\mathrm{ad}_x)(y).
\end{split}
\end{eqnarray}
From $(\ref{F6})$ and $(\ref{F7})$, we obtain
\begin{equation*}
\xi\ast x=\xi\circ(\nabla_x-\mathrm{ad}_x),~~\forall x\in\h, \xi\in\h^\ast.
\end{equation*}
Let us consider,
\begin{align*}
x\ast \xi&=(x\ast \xi)_{|\h}+(x\ast \xi)_{|\h^\ast},
\end{align*}
\begin{eqnarray}\label{F8}
(x\ast \xi)_{|\h^\ast}(y)=-\omega_0\big((x\ast \xi)_{|\h^\ast},y\big)
=\omega_0\big(\xi,[x,y]+\alpha(x,y)\big)=-(\xi\circ\mathrm{ad}_x)(y).
\end{eqnarray}
Finally,
\begin{eqnarray}\label{F9}
\xi_1\big((x\ast \xi)_{|\h}\big)=\omega_0\big((x\ast \xi)_{|\h},\xi_1\big)=-\omega_0\big(\xi,[x,\xi_1]\big)=0.
\end{eqnarray}
From $(\ref{F8})$ and $(\ref{F9})$, we have
\begin{equation*}
x\ast\xi=-\xi\circ\mathrm{ad}_x,~\forall x\in\h, \xi\in\h^\ast.
\end{equation*}
From \cite{MR}. The Lagrangian extension $(\G_{\nabla,\alpha},\omega_0)$ is a Frobenius Lie algebra if and only if the corresponding flat torsion-free connection ``$\ast$" has a right-identity element.

Let $e^\prime=e^\prime_{|\h}+e^\prime_{|\h^\ast}$ be the right-identity element of ``$\ast$". We have
\begin{align*}
x\ast (e^\prime_{|\h}+e^\prime_{|\h^\ast})&=\nabla_xe^\prime_{|\h}+\alpha(x,\cdot)(e^\prime_{|\h})-e^\prime_{|\h^\ast}\circ\mathrm{ad}_x=x,\\
\xi\ast (e^\prime_{|\h}+e^\prime_{|\h^\ast})&=\xi\circ(\nabla_{e^\prime_{|\h}}-\mathrm{ad}_{e^\prime_{|\h}})=\xi.
\end{align*} 
This implies that,
\begin{equation}
\nabla_xe^\prime_{|\h}=x,~~~\text{and}~~~\alpha(x,\cdot)(e^\prime_{|\h})=e^\prime_{|\h^\ast}\circ\mathrm{ad}_x
\end{equation}
and
\begin{equation}
\varrho(e^\prime_{|\h})\xi=\xi,
\end{equation}
where, $\varrho(x)=\nabla_x-\mathrm{ad}_x$ is the right-multiplication by $x$ in the flat Lie algebra $(\h,\nabla)$. Therefore, we have $e^\prime_{|\h}:=e$ and $e^\prime_{|\h^\ast}:=\gamma$.
\end{proof}
A symplectic Novikov Lie algebra (\textit{SNLA}), is a symplectic Lie algebra $(\G,\omega)$ for which the associated flat torsion-free connection ``$\ast$" is Novikov. i.e., $(z\ast y)\ast x=(z\ast x)\ast y$, for  all $x,y,z\in\G$ \cite{AM2}.
\begin{pr}
The Lagrangian extension $(\G_{\nabla,\alpha},\omega_0)$ is an SNLA if and only if  the following conditions are satisfied:
\begin{enumerate}
\item  $(\h,\nabla)$ is Novikov associative flat Lie algebra.
\item $\alpha(zy,\cdot)(x)+\alpha(z,\cdot)(y)\circ\varrho(x)=\alpha(zx,\cdot)(y)+\alpha(z,\cdot)(x)\circ\varrho(y)$,
\end{enumerate}
for all $x,y,z\in\h$, $\xi\in\h^\ast$.
\end{pr}
\begin{proof}
Let $x,y,z\in\h$,  $\xi\in\h^\ast$, and $xy:=\nabla_xy$. Assume that ``$\ast$" is Novikov, we have
\begin{align*}
(z\ast y)\ast x=(z\ast x)\ast y&\Leftrightarrow \big(zy+\alpha(z,\cdot)(y)\big)\ast x=\big(zx+\alpha(z,\cdot)(x)\big)\ast y\\
&\Leftrightarrow (zy)x+\alpha(zy,\cdot)(x)+\alpha(z,\cdot)(y)\circ\varrho(x)=(zx)y+\alpha(zx,\cdot)(y)+\alpha(z,\cdot)(x)\circ\varrho(y)\\
&\Leftrightarrow \begin{cases}[\varrho(x),\varrho(y)]=0,\\
\alpha(zy,\cdot)(x)+\alpha(z,\cdot)(y)\circ\varrho(x)=\alpha(zx,\cdot)(y)+\alpha(z,\cdot)(x)\circ\varrho(y).
\end{cases}
\end{align*}
and
\begin{align*}
(x\ast y)\ast \xi=(x\ast \xi)\ast y&\Leftrightarrow \big(xy+\alpha(x,\cdot)(y)\big)\ast\xi=\big(-\xi\circ\mathrm{ad}_x\big)\ast y\\
&\Leftrightarrow-\xi\circ\mathrm{ad}_{xy}=-\xi\circ\mathrm{ad}_x\circ\varrho(y)\\
&\Leftrightarrow\xi\circ(\mathrm{ad}_{xy}-\mathrm{ad}_x\circ\varrho(y)=0.
\end{align*}
On the other hand, 
\begin{align*}
\big(\mathrm{ad}_{xy}-\mathrm{ad}_x\circ\varrho(y)\big)(z)=0&\Leftrightarrow(xy)z-z(xy)=x(zy)-(zy)x\\
&\Leftrightarrow
(xz)y-z(xy)=x(zy)-(zy)x\\
&\Leftrightarrow (xz)y-x(zy)+(zy)x-z(xy)=0
\\
&\Leftrightarrow (xz)y-x(zy)+(zx)y-z(xy)=0\\
&\Leftrightarrow 2\,(x,z,y)=0.
\end{align*}
It is straightforward to show that the condition $(\xi\ast x)\ast y=(\xi\ast y)\ast x$ is equivalent to
\begin{equation*}
\xi\circ[\varrho(x),\varrho(y)]=0,
\end{equation*}
for all $x,y\in\h$ and $\xi\in\h^\ast$.
\end{proof}

\begin{remark}
Note that every flat Lie algebra has at least one Lagrangian
extension, using the zero-cocycle $\alpha\equiv0$.   This is the semi-direct product Lagrangian extension $(\h\oplus_\nabla \h^\ast)$ or cotangent Lagrangian extension.  In this case, the cotangent Lagrangian extension is
\begin{enumerate}
\item  Frobenius Lie algebra if and only if $(\h,\nabla)$ has a right-identity element.
\item  An SNLA if and only if $(\h,\nabla)$ is Novikov associative flat algebra.
\item Complete if and only if $(\h,\nabla)$ is complete and $\mathcal{I}m(\nabla_x)\subset\ker\big(\varrho^{k-1}(x)\big)$,
 \end{enumerate}
where, $k>1$ is the index nilpotency  of $\varrho(x)$.
\end{remark}
\begin{pr}
Let $(\G_{\nabla,\alpha},\omega)$ be a Lagrangian extension of 
flat Lie algebra $(\h, \nabla)$ with respect to $\alpha$. Assume that, $\G_{\nabla,\alpha}$ is Frobeniusian, we have
\begin{enumerate}
\item If $\dim\h=3$. Then, $\G_{\nabla,\alpha}$  is non-nilpotent solvable Lie algebra.
\item If $\dim\h\geq4$. Then, $\G_{\nabla,\alpha}$ is non-nilpotent solvable Lie algebra or $\G_{\nabla,\alpha}$  has non-trivial Levi-Malcev decomposition. 
\end{enumerate}
\end{pr}
\begin{proof}
There is no flat Lie algebra (left-symmetric structure) on semisimple Lie algebra as is well known. If $\dim\h=3$, hence, over the real
field $\R$, besides 3-dimensional simple Lie algebras $\mathfrak{sl}_2(\R)$ and $\mathfrak{so}(3)$, up to isomorphisms, there are seven solvable Lie algebras. Assuming $\G_{\nabla,\alpha}$ is Frobeniusian, $\G_{\nabla,\alpha}$ has to be non-nilpotent (in general, not unimodular), so it is non-nilpotent solvable Lie algebra. We have two types of Lie algebras when $\dim\h=4$, the first of which is solvable, in this case $\G_{\nabla,\alpha}$ is necessarily solvable, and the second  is reductive, i.e., $\mathfrak{sl}_2(\R)\oplus\R e_4$ and $\mathfrak{so}(3)\oplus\R e_4$   and admits a flat torsion-free connections (see Theorem $1$, \cite{AM}). As $\G_{\nabla,\alpha}$ consists of a semisimple Lie algebra and is symplectic (then, $ \mathfrak{rad}(\G_{\nabla,\alpha})\neq0)$, it has a Levi-Malcev decomposition. When $\dim\h>4$, a consequence of what precedes follows immediately.

\end{proof}
\subsection{Bijective $1$-Cocycles}
\begin{Def}
Let $\G$ be a Lie algebra and $\rho: \G\longrightarrow\mathfrak{gl}(V)$ be a representation of $\G$.
A $1$-cocycle $\mathfrak{q}$ associated to $\rho$ is defined as a linear map from $\G$ to $V$ satisfying
\begin{equation}\label{cocydefi}
\mathfrak{q}\big([x,y]\big)=\rho(x)\mathfrak{q}(y)-\rho(y)\mathfrak{q}(x),~~\forall x,y\in\G.
\end{equation}
We denote it by $(\rho,\mathfrak{q})$. In addition, if $\mathfrak{q}$ is a linear isomorphism $($thus $\dim V =
\dim\G)$ , $(\rho,\mathfrak{q})$ is said to be bijective.
\end{Def}
Let $(\rho,\mathfrak{q})$ be a bijective $1$-cocycle, then it is easy to see that
\begin{equation}
x\ast y=\mathfrak{q}^{-1}\big(\rho(x)\mathfrak{q}(y)\big),~x,y\in\G
\end{equation}
defines a flat torsion-free connection on $\G$ (Medina, \cite{Medina1}).  There exists a bijection between the set of the isomorphism classes of bijective 1-cocycles of $\G$ and the set of flat Lie algebras on $\G$ (see, \cite{Bai},  Theorem 2.1).

\begin{Le}\label{dimh=5}
There is no flat torsion-free connections  on the following $5$-dimensional Lie algebras$:$
\begin{center}
$\mathfrak{so}(3)\oplus\R^2,~~\mathfrak{sl}_2(\R)\oplus\R^2,~~\mathfrak{so}(3)\oplus\mathfrak{aff}(1,\R),~~\mathfrak{sl}_2(\R)\oplus\mathfrak{aff}(1,\R)$~and~$2L_1\rtimes\mathfrak{sl}_2(\R)$.
\end{center}

\end{Le}
\begin{proof}
A Helmstetter result \cite{Hel} shows that a perfect Lie algebra (i.e., $[\G,\G]=\G)$ does not support a flat torsion-free connection. Then $2L_1\rtimes\mathfrak{sl}_2(\R)$ has no flat torsion-free connection.  Let $\{e_1,e_2,e_3\}\oplus\{e_4,e_5\}$ be a basis of $\mathfrak{s}\oplus \mathfrak{r}$, where $\mathfrak{s}=\mathfrak{so}(3)~\text{or}~\mathfrak{sl}_2(\R)$ and $\mathfrak{r}=\R^2~\text{or}~\mathfrak{aff}(1,\R)$.  Suppose that, $\h=\mathfrak{s}\oplus \mathfrak{r}$ has a flat torsion-free connection $\nabla$.  The first step is to classify all representations of $\h :$

If $\mathfrak{s}=\mathfrak{sl}_2(\R)$. We have only two possibilities for $\h$. In the first case, $\h$ is irreducible, and in the second case, $\h =V_1\oplus V_2$ , where $V_1$ (resp. $V_2$) (as an $\mathfrak{sl}_2(\R)$-module) is isomorphic to the $2$-dimensional (resp. $3$-dimensional) natural representation of $\mathfrak{sl}_2(\R)$.

 As matrices, $\nabla_{e_3}$ is similar to $\mathrm{diag}(-4, -2,0, 2,4)$ or to $\mathrm{diag}(1,-1,-3,-1,1)$
and $\nabla_{e_1}$, $\nabla_{e_2}$ are nilpotent. Indeed, if $\h$ is irreducible, it is (as an $\mathfrak{sl}_2(\R)$-module) a highest weight module with basis $v_i$ such that 
\begin{equation}\label{repre1}
		\nabla_{e_3} v_j=(-4+2j)v_j, \nabla_{e_1}v_j=v_{j+1}, \nabla_{e_2}v_j=j(5-j)v_{j-1} ~\text{and}~ \nabla_{e_2}v_0=0 ~\text{for}~ j=0,\ldots,4.
	\end{equation}
As a result of this basis, we are able to obtain the first part of our claim. It should be noted that this basis does not satisfy the torsion-free condition (i.e., $\nabla_xy-\nabla_yx=[x,y]$). In the second case,
choose a basis according to $V_1\oplus V_2$, where $V_j,~j=1,2$, is a highest weight module for $\mathfrak{sl}_2(\R)$.

If $\mathfrak{s}=\mathfrak{so}(3)$. The only possibility for $\h$ is that it is irreducible, it is (as an $\mathfrak{so}(3)$-module) a highest weight module with basis $v_i$ such that
\begin{eqnarray}\label{repre2}
&&\nabla_{e_1}v_1=\tfrac{1}{2}v_4,~\nabla_{e_1}v_2=-\tfrac{1}{2}v_3,~\nabla_{e_1}v_3=2v_2-v_5,~\nabla_{e_1}v_4=-2v_1,~\nabla_{e_1}v_5=3v_3,\nonumber\\
&&\nabla_{e_2}v_1=\tfrac{1}{2}v_3,~\nabla_{e_2}v_2=-\tfrac{1}{2}v_4,~\nabla_{e_2}v_3=-2v_1,~\nabla_{e_2}v_4=-2v_2-v_5,~\nabla_{e_2}v_5=3v_4,\\
&&\nabla_{e_3}v_1=2v_2,~\nabla_{e_3}v_2=-2v_1,~\nabla_{e_3}v_3=v_4,~\nabla_{e_3}v_4=-v_3,~\nabla_{e_3}v_5=0.\nonumber
\end{eqnarray}
We should note that, in any case, $\nabla_{e_4}v_i=v_i$  and $\nabla_{e_5}v_i=v_i$ (when $\mathfrak{r}=\R^2$) or $\nabla_{e_4}v_i=v_i$ and $\nabla_{e_5}v_i=0$,  (when $\mathfrak{r}=\mathfrak{aff}(1,\R)$) $i=1,\ldots,5$.

The second step is to classify all the 1-cocycles associated with each representation. However, a short calculation shows that any $1$-cocycle $\mathfrak{q}$ associated with any representation of $\h$ and satisfying $(\ref{cocydefi})$ has $\mathrm{rank}(\mathfrak{q}) = 4$ (that is, a bijective 1-cocycle $(\rho,\mathfrak{q})$ does not exist). Thus, $\mathfrak{s}\oplus\mathfrak{r}$ cannot admit a flat torsion-free connection.
\end{proof}
The following result shows that, any  $10$-dimensional symplectic  Lie algebra  admitting a Lagrangian ideal is necessarily solvable.
\begin{Le}\label{10is sol}
Let $(\G,\omega)$ be a  $10$-dimensional  symplectic  Lie algebra that has a Lagrangian ideal. Then $\G$ is solvable.
\end{Le}
\begin{proof}
Let $(\G=\mathfrak{s}\ltimes\mathfrak{r},\omega)$ be a 10-dimensional exact symplectic  Lie algebra, $\mathfrak{j}$ be a Lagrangian ideal of $(\G,\omega)$. Then, $(\G,\omega)$ is a Lagrangian extension of a  $5$-dimensional flat Lie algebra. Recall (from Proposition 1.3.1, \cite{BV}) that the associated flat torsion-free
connection $\nabla$ on the quotient Lie algebra $\h = \G/\mathfrak{j}$ satisfies the relation
\begin{equation}
\omega_\h(\nabla_{\overline{x}}\overline{y},a)=-\omega(y,[x,a]),~\text{for all}~\overline{x},\overline{y}\in\h,~a\in\mathfrak{j}.
\end{equation}
 Since $\G$ has a non-trivial Levi-Malcev decomposition, the flat Lie algebra $(\h,\nabla)$ (as a Lie algebra) must contain a semisimple part. It follows that $\h\cong\mathfrak{s}^\prime\oplus\mathfrak{d}$ where $\mathfrak{s}^\prime$ is a three-dimensional semisimple Lie algebra, and $\mathfrak{d}$ is a two-dimensional Lie algebra, or $\h$ has non-trivial Levi-Malcev decomposition, in this case $\h\cong2L_1\rtimes\mathfrak{sl}_2(\R)$. By Lemma \ref{dimh=5}, these forms of Lie algebras do not support a flat torsion-free connection, which implies that, $\mathfrak{s}\equiv0$.

\end{proof}

Together with Lemma \ref{10is sol}, we therefore have:

\begin{theo}
There is no Lagrangian ideal for any $10$-dimensional symplectic non-solvable Lie algebra.
\end{theo}

The following immediate consequence appears from \cite{AM}. Any eight-dimensional exact symplectic non-solvable Lie algebra $\G$ has a unique Lagrangian ideal $\mathfrak{j}$. Observe that every quotient $\h$ of an exact symplectic Lie algebra $\G$ must have a right-identity element as well.  
The uniqueness of the Lagrangian ideal in  any eight-dimensional exact symplectic non-solvable Lie algebra $\G$ has the following remarkable consequence:
\begin{pr}\label{Pr from eight}
The correspondence which associates to an eight-dimensional exact symplectic non-solvable Lie algebra the extension triple $(\h,\nabla, [\alpha]_{L,\rho})$, which is defined with respect to the Lagrangian ideal $\mathfrak{nil}(\G)$ of $(\G,\omega)$, induces a bijection between isomorphism classes of eight-dimensional exact symplectic non-solvable  Lie algebras and isomorphism classes of
 four dimensional flat reductive\footnote{A Lie algebra $\h$ is called reductive if the following equivalence conditions hold$:$
\begin{enumerate}
\item it is the direct sum $\h\simeq\mathfrak{s}\oplus\mathfrak{a}$ of a semisimple Lie algebra $\mathfrak{s}$ and an abelian Lie algebra $\mathfrak{a}$;
\item its adjoint representation is completely reducible: every invariant subspace has an invariant complement. 
\end{enumerate}} Lie algebras with symplectic extension cohomology class.
\end{pr}

\section{Three-dimensional flat Lie algebras}\label{sec3}

It is well known that there is no left-symmetric algebra structure on a semisimple Lie algebra. Therefore, over the real
field $\R$, besides $3$-dimensional simple Lie algebras $\mathfrak{sl}_2(\R)$ and $\mathfrak{so}(3)$, up to isomorphisms, there
are the following (non-isomorphic) Lie algebras (Mubarakzyanov's classification, we only give the nonzero products):
\begin{enumerate}
\item[] $3\G_1$, abelian.
\item[] $3\G_{2,1}\oplus\G_1=\langle e_1,e_2,e_3~|~[e_3,e_2]=e_2\rangle$, decomposable solvable.
\item[] $\h_3=\langle e_1,e_2,e_3~|~[e_1,e_2]=e_3\rangle$, Heisenberg-Weyl algebra, nilpotent.
\item[] $\G_{3,2}=\langle e_1,e_2,e_3~|~[e_1,e_3]=e_1,~[e_2,e_3]=e_1+e_2\rangle$, solvable.
\item[] $\G_{3,3}=\langle e_1,e_2,e_3~|~[e_1,e_3]=e_1,~[e_2,e_3]=e_2\rangle$, solvable.
\item[] $\G_{3,4}=\langle e_1,e_2,e_3~|~[e_1,e_3]=e_1,~[e_2,e_3]=\alpha e_2, -1\leq\alpha<1,\alpha\neq 0\rangle$, solvable, Poincar\'e algebra $\mathfrak{p}(1,1)$ when $\alpha=-1$.
\item[] $\G_{3,5}=\langle e_1,e_2,e_3~|~[e_1,e_3]=\beta e_1-e_2,~[e_2,e_3]=e_1+\beta e_2, \beta\geq 0\rangle$, solvable.
\end{enumerate}
 
\begin{Le}\label{righ-identity} Let $(\h,\nabla)$ be a flat Lie algebra on three-dimensional Lie algebra $\G_{3,j}$. If $(\h,\nabla)$ has a  right-identity element, we may assume that the right-identity 
\begin{enumerate}
\item For Lie algebra $3\G_1$ is, $e=e_3$.
\item For Lie algebra $3\G_{2,1}\oplus\G_1$ is, $e=e_1,~e=e_2,~e=e_1+e_2$ or $e=\lambda e_3 $, with $\lambda\in\R^\ast$.
\item For Lie algebra $\h_3$ is, $e=e_2$ or $e=e_3$.
\item For Lie algebra $\G_{3,2}$ is, $e=e_1$, $e=e_2$ or $e=\mu e_3$, with $\mu\in\R^\ast$.
\item For Lie algebra $\G_{3,3}$ is, $e=e_1$ or $e=\eta e_3$, with $\eta\in\R^\ast$.
\item For Lie algebra $\G_{3,4}$ is, $e=e_1$, $e=e_2$, $e=e_1+e_2$ or $e=\gamma e_3$, with $\gamma\in\R^\ast$ and $0<|\alpha|<1$.\\\phantom{xxxxxxxxxxxxxxxxxxx~~~}$e=e_1$, $e=e_1+e_2$ or $e=\nu e_3$, with $\nu>0$,  when, $\alpha=-1$.
\item For Lie algebra $\G_{3,5}$ is, $e=e_1$ or $e=\delta_1 e_3$, with $\delta_1\in\R^\ast$ and $\beta>0$.\\
\phantom{xxxxxxxxxxxxxxxxxxx~~~}$e=e_1$ or $e=\delta_2e_3$, with $\delta_2>0$, when $\beta=0$.

\end{enumerate}

\end{Le}
\begin{proof}
Two flat Lie algebras $(\h, \nabla)$ and $(\h, \tilde{\nabla})$ are isomorphic if and
only if there is a $\psi \in\mathrm{Aut}(\h)$ such that $\tilde{\nabla}_x= \psi\circ  \nabla_{\psi^{-1}(x)} \circ \psi^{-1}$. If $e$ is a right-identity element of $(\h,\nabla)$, then $\varrho(e)=\mathrm{I}_{\h}$, implies $\tilde{\varrho}(\psi(e)) =\psi\circ\varrho(e)\circ\psi^{-1}=\mathrm{I}_{\h}$, i.e., the flat Lie algebra $(\h,\tilde{\nabla})$ has right-identity $\psi(e)$. Lie algebra $3\G_{2,1}\oplus\G_1$ is the only one shown in the proof, but the same procedure will be applied to all remaining algebras. The automorphism group of $3\G_{2,1}\oplus\G_1$ is
\begin{equation}\label{groupauto}
\begin{small} \mathrm{Aut}(3\G_{2,1}\oplus\G_1)=\left\{\psi=\left( \begin {array}{ccc} x_{{11}}&0&x_{{13}}\\ \noalign{\medskip}0
&x_{{22}}&x_{{23}}\\ \noalign{\medskip}0&0&1\end {array} \right)~|~x_{11}x_{22}\neq0
\right\}.
\end{small}
\end{equation}
Let $e=\lambda_1e_1+\lambda_2 e_2+\lambda_3 e_3:=(\lambda_1,\lambda_2,\lambda_3)$, and note that $e\neq0$. If $\lambda_3\neq 0$, then let $x_{23}=-\frac{x_{22}\lambda_2}{\lambda_3}$, $x_{13}=-\frac{x_{11}\lambda_1}{\lambda_3}$ and we have, $\psi(e)=(0,0,\lambda_3)$ with $\lambda_3\neq0$. 
If $\lambda_3=0$.\\
$\ast$ \textbf{Case} 1. If $\lambda_2=0$ implies $\lambda_1\neq0$ and $\psi(e)=(1,0,0)$ with $x_{11}=\frac{1}{\lambda_1}$.\\ $\ast$ \textbf{Case} 2, $\lambda_2\neq0$ and $\lambda_1=0$, then $\psi(e)=(0,1,0)$ with $x_{22}=\frac{1}{\lambda_2}$.\\
$\ast$ \textbf{Case} 3. If $\lambda_1\lambda_2\neq0$, then $\psi(e)=(1,1,0)$ with $x_{11}=\frac{1}{\lambda_1}$ and $x_{22}=\frac{1}{\lambda_2}$.

 As a result, $e=e_1,~e=e_2,~e=\lambda e_3$ or $e=e_1+e_2$ with $\lambda\in\R^\ast$. It is easy to verify that, for all $\psi\in\mathrm{Aut}(\mathcal{A}_2)$, $\psi(e_1)\neq e_2,\lambda e_3, e_1+e_2$, each and every one of them in the same way.
\end{proof}
Based on explicit calculations, we are now able to classify the flat Lie algebras (left-invariant affine structures) on each Lie algebra given above.
\begin{pr}\label{Pr 3flat}
Let $(\h,\nabla)$ be a three-dimensional flat Lie algebra which admits a right-identity elements. Then $\h$  is isomorphic to one of the following algebras:
\end{pr}
{\renewcommand*{\arraystretch}{1.3}
\captionof{table}{Three-dimensional flat Lie algebras with right-identity element.}\label{tableleftt}
\small\begin{longtable}{clc}
			\hline
		Algebra& flat torsion-free connection&Lie algebra\\
			\hline
		$\h_1$& $\nabla_{e_3}e_j=e_j,~j=1,2,3$, $\nabla_{e_1}e_3=e_1$, $\nabla_{e_2}e_3=e_2$, $e=e_3$. &\multirow{1}{*}{$3\G_1$}	
		\\
	$\h_{2}^{b}$&$\nabla_{e_3}e_j=e_j,~j=1,2,3$, $\nabla_{e_1}e_1=e_1+e_2$, $\nabla_{e_1}e_2=be_3$, $\nabla_{e_1}e_3=e_1$, &\\
		&$\nabla_{e_2}e_1=be_3$, $\nabla_{e_2}e_2=be_1-be_3$, $\nabla_{e_2}e_3=e_2$, $b\in\R^\ast$, and $e=e_3$.&\\
$\h_{3}$&$\nabla_{e_3}e_j=e_j,~j=1,2,3$, $\nabla_{e_1}e_3=e_1$, $\nabla_{e_2}e_1=e_1$, $\nabla_{e_2}e_3=e_2$.&
		\\
		$\h_{4}^a$&$\nabla_{e_1}e_j=e_j, j=1,2,3,~\nabla_{e_2}e_1=e_2,~\nabla_{e_3}e_1=e_3,~\nabla_{e_3}e_2=e_2,$ &\multirow{1}{*}{$3\G_{2,1}\oplus\G_1$}	
		\\
		&$\nabla_{e_3}e_3= e_2+ae_3$, $a\in\R^\ast$ and $e=e_1$.&\\
	$\h_5$&$\nabla_{e_1}e_j=e_j, j=1,2,3,~\nabla_{e_2}e_1=e_2,~\nabla_{e_3}e_1=e_3,~\nabla_{e_3}e_2=e_2,$ &	\\
		&$\nabla_{e_3}e_3=e_3$, $e=e_1$.&\\
		$\h_6$&$\nabla_{e_1}e_j=e_j, j=1,2,3,~\nabla_{e_2}e_1=e_2,~\nabla_{e_3}e_1=e_3,~\nabla_{e_3}e_2=e_2,~e=e_1.$ 
		\\
		$\h_7$&$\nabla_{e_1}e_j=e_j, j=1,2,3,~\nabla_{e_2}e_1=e_2,~\nabla_{e_2}e_2= e_3,~\nabla_{e_3}e_1=e_3,$&
		\\
		&$\nabla_{e_3}e_2= e_2,~\nabla_{e_3}e_3= 2e_3$, $e=e_1$.&
		\\
		$\h_8$&$\nabla_{e_1}e_j=e_j, j=1,2,3,~\nabla_{e_2}e_1=e_2,~\nabla_{e_2}e_2=- e_3,~\nabla_{e_3}e_1=e_3,$&
		\\
		&$\nabla_{e_3}e_2= e_2,~\nabla_{e_3}e_3= 2e_3$, $e=e_1$.&
		\\
		
		$\h_9^{a}$&$\nabla_{e_1}e_1=e_1-e_2,~\nabla_{e_1}e_2=e_2,~\nabla_{e_1}e_3=-e_2+e_3,~\nabla_{e_2}e_1=e_2,$&
		\\
&$\nabla_{e_3}e_1=-e_2+e_3,~\nabla_{e_3}e_2=e_2,~\nabla_{e_3}e_3=ae_2+e_3,$&
		\\
		&$a\in\R$~~ and~~ $e=e_1+e_2$.&
		\\
		$\h_{10}^{\lambda,a,b}$&$\nabla_{e_1}e_1=\frac{a}{b} e_1-\frac{\lambda (a-b)}{b} e_3,~\nabla_{e_1}e_2=e_2,~\nabla_{e_1}e_3=\frac{1}{\lambda} e_1,~\nabla_{e_2}e_1=e_2$,&\\
&$\nabla_{e_2}e_3=\frac{1}{\lambda} e_2,~\nabla_{e_3}e_1=\frac{1}{\lambda} e_1,~\nabla_{e_3}e_2=\frac{1+\lambda}{\lambda} e_2,~\nabla_{e_3}e_3=\frac{1}{\lambda}e_3,$\\
&  $\lambda,a,b\in\R^\ast,~b\neq a,$ ~~and~~$e=\lambda e_3$.&
		\\
		$\h_{11}^{\lambda}$&$\nabla_{e_1}e_1= e_1,~\nabla_{e_1}e_2=e_2,~\nabla_{e_1}e_3=\frac{1}{\lambda} e_1,~\nabla_{e_2}e_1=e_2$,&\\
&$\nabla_{e_2}e_3=\frac{1}{\lambda} e_2,~\nabla_{e_3}e_1=\frac{1}{\lambda} e_1,~\nabla_{e_3}e_2=\frac{1+\lambda}{\lambda} e_2,~\nabla_{e_3}e_3=\frac{1}{\lambda}e_3,$\\
&  $\lambda\in\R^\ast$~~and~~$e=\lambda e_3$.&
		\\
		$\h_{12}^{\lambda}$&$\nabla_{e_1}e_1= \lambda e_3,~\nabla_{e_1}e_2=e_2,~\nabla_{e_1}e_3=\frac{1}{\lambda} e_1,~\nabla_{e_2}e_1=e_2$,&\\
&$\nabla_{e_2}e_3=\frac{1}{\lambda} e_2,~\nabla_{e_3}e_1=\frac{1}{\lambda} e_1,~\nabla_{e_3}e_2=\frac{1+\lambda}{\lambda} e_2,~\nabla_{e_3}e_3=\frac{1}{\lambda}e_3,$\\
&  $\lambda\in\R^\ast$~~and~~$e=\lambda e_3$.&
		\\
		$\h_{13}^{\lambda}$&$\nabla_{e_1}e_1= e_1,~\nabla_{e_1}e_3=\frac{1}{\lambda} e_1$,~$\nabla_{e_2}e_3=\frac{1}{\lambda} e_2$,~$\nabla_{e_3}e_1=\frac{1}{\lambda} e_1$,~$\nabla_{e_3}e_2=\frac{1+\lambda}{\lambda} e_2$&\\
&$\nabla_{e_3}e_3=\frac{1}{\lambda}e_3,$  $\lambda\in\R^\ast$~~and~~$e=\lambda e_3$.&
		\\
		$\h_{14}^{\lambda}$&$\nabla_{e_1}e_3=\frac{1}{\lambda} e_1$,~$\nabla_{e_2}e_3=\frac{1}{\lambda} e_2$,~$\nabla_{e_3}e_1=\frac{1}{\lambda} e_1$,~$\nabla_{e_3}e_2=\frac{1+\lambda}{\lambda} e_2$,~$\nabla_{e_3}e_3=\frac{1}{\lambda}e_3,$&\\
&  $\lambda\in\R^\ast$~~and~~$e=\lambda e_3$.&
		\\
			
		\multirow{1}{*}{$\h_{15}$}&$\nabla_{e_1}e_1=e_3$,~$\nabla_{e_1}e_2= e_1$,~$\nabla_{e_2}e_1=e_1-e_3$,~$\nabla_{e_2}e_2= e_2$,~$\nabla_{e_2}e_3= e_3$,&\multirow{1}{*}{$\h_3$}		\\
		&$\nabla_{e_3}e_2= e_3$~~and~~$e=e_2$.&\\
		\multirow{1}{*}{$\h_{16} $}&$\nabla_{e_1}e_2= e_1$,~$\nabla_{e_2}e_1=e_1-e_3$,~$\nabla_{e_2}e_2= e_2$,~$\nabla_{e_2}e_3= e_3$, $\nabla_{e_3}e_2= e_3$,&\\
		&$e=e_2$.&\\
		\multirow{1}{*}{$\h_{17}^{\mu}$}&$\nabla_{e_1}e_3=\frac{1}{\mu}e_1$, $\nabla_{e_2}e_3=\frac{1}{\mu}e_2$, $\nabla_{e_3}e_1=\frac{1-\mu}{\mu}e_1$, $\nabla_{e_3}e_2=-e_1+\frac{1-\mu}{\mu}e_2$,&\multirow{1}{*}{$\G_{3,2}$}		\\
	&$\nabla_{e_3}e_3=\frac{1}{\mu}e_3$, $\mu\in\R^\ast$~~and~~$e=\mu e_3$.&	\\
		\multirow{1}{*}{$\h_{18}^{\eta}$}&$\nabla_{e_1}e_3=\frac{1}{\eta}e_1$, $\nabla_{e_2}e_3=\frac{1}{\eta}e_2$, $\nabla_{e_3}e_1=\frac{1-\eta}{\eta}e_1$, $\nabla_{e_3}e_2=\frac{1-\eta}{\eta}e_2$, &\multirow{1}{*}{$\G_{3,3}$}		\\
		&$\nabla_{e_3}e_3=\frac{1}{\eta}e_3$, $\eta\in\R^\ast$~~and~~$e=\eta e_3$.&
		\\
			
			\multirow{1}{*}{$\h_{19}^{\alpha,\gamma}$}&$\nabla_{e_1}e_3=\frac{1}{\gamma}e_1$, $\nabla_{e_2}e_3=\frac{1}{\gamma}e_2$, $\nabla_{e_3}e_1=(-1+\frac{1}{\gamma})e_1$, $\nabla_{e_3}e_2=(-\alpha+\frac{1}{\gamma})e_2$&	\multirow{1}{*}{$\G_{3,4}^{0<|\alpha|<1}$}\\
	& $\nabla_{e_3}e_3=\frac{1}{\gamma}e_3$, $\gamma\in\R^\ast$~~and~~$e=\gamma e_3$.&	
\\
		\multirow{1}{*}{$\h_{20}^{\gamma}$}&$\nabla_{e_1}e_3=\frac{1}{\gamma}e_1$, $\nabla_{e_2}e_2=e_1$, $\nabla_{e_2}e_3=\frac{1}{\gamma}e_2$, $\nabla_{e_3}e_1=(-1+\frac{1}{\gamma})e_1$,&	\\
	&$\nabla_{e_3}e_2=(-\frac{1}{2}+\frac{1}{\gamma})e_2$, $\nabla_{e_3}e_3=\frac{1}{\gamma}e_3$, $\gamma\in\R^\ast$~~and~~$e=\gamma e_3$.&			
		\\
		\multirow{1}{*}{$\h_{21}^{\nu}$}&$\nabla_{e_1}e_3=\frac{1}{\nu}e_1$, $\nabla_{e_2}e_3=\frac{1}{\nu}e_2$, $\nabla_{e_3}e_1=\frac{1-\nu}{\nu} e_1$, $\nabla_{e_3}e_2=\frac{1+\nu}{\nu} e_2$,&\multirow{1}{*}{$\G_{3,4}^{\alpha=-1}$}\\
		&$\nabla_{e_3}e_3=\frac{1}{\nu} e_3$, $\nu\in\R^{\ast +}$~~and~~$e=\nu e_3$.&
		\\
		\multirow{1}{*}{$\h_{22}^{\beta,\delta_1}$}&$\nabla_{e_1}e_3=\frac{1}{\delta_1}e_1$, $\nabla_{e_2}e_3=\frac{1}{\delta_1}e_2$, $\nabla_{e_3}e_1=\frac{-\beta\delta_1+1}{\delta_1} e_1+e_2$, $\nabla_{e_3}e_2=-e_1+\frac{-\beta\delta_1+1}{\delta_1} e_2$,&\multirow{1}{*}{$\G_{3,5}^{\beta>0}$}	\\
		&$\nabla_{e_3}e_3=\frac{1}{\delta_1}e_3$, $\beta,\in\R^{\ast +}$, $\delta_1\in\R^\ast$~~and~~$e=\delta_1 e_3$.&
		\\
			
		\multirow{1}{*}{$\h_{23}^{\delta_2}$}&$\nabla_{e_1}e_3=\frac{1}{\delta_2}e_1$, $\nabla_{e_2}e_3=\frac{1}{\delta_2}e_2$, $\nabla_{e_3}e_1=\frac{1}{\delta_2} e_1+e_2$, $\nabla_{e_3}e_2=-e_1+\frac{1}{\delta_2}e_2$&\multirow{1}{*}{$\G_{3,5}^{\beta=0}$}	\\	
		
		&$\nabla_{e_3}e_3=\frac{1}{\delta_2}e_3$, $\delta_2\in\R^{\ast +}$~~and~~$e=\delta_2 e_3$.&
		\\\hline
			\end{longtable}}

\begin{proof}
 The first step in the procedure is to find all the possible expressions of the right-identity element $($Lemma $\ref{righ-identity}$ illustrates how this can be applied$)$. The \textit{LSA}-structure is given by $27$ structure constants via $\nabla_{e_1}, \nabla_{e_2}, \nabla_{e_3}$. The
condition $[x, y] = x \cdot y - y \cdot x$ determines 9 structure constants by linear equations. The \textit{LSA}-structure equations
\begin{equation}\label{LSAproperty}
\nabla_x\nabla_y-\nabla_y\nabla_x-\nabla_{[x,y]}=0
\end{equation}
 is equivalent to quadratic equations in the structure constants. In general, they are quite difficult to solve. It will be easier to calculate if there is a right-identity element. Now we will present a complete proof of Lie algebra $3\G_{2,1}\oplus\G_1$. Similarly, we will prove the other cases  using the same procedure. Let $\varrho(x)$  the right-multiplication by
$x$ in the flat Lie algebra $(\h,\nabla)$. We have
 \begin{equation}\label{rightsimplify}
 \varrho(e)=\mathrm{I}_\h~~~~\text{and}~~~~\mathrm{Tr}(\varrho(x))=3.
 \end{equation}
This determines another $9$ structure constants by linear equations. Let $\nabla_{e_1} = (a_{ij})$, $\nabla_{e_2} = (b_{ij})$, $\nabla_{e_3} = (c_{ij})$ with $i, j = 1,\ldots,4$ . Using $\nabla_x-\varrho(x)=\mathrm{ad}_x$ we obtain

\begin{eqnarray}
&&\nabla_{e_1}e_1=a^j_1e_j,~\nabla_{e_1}e_2=a^j_2e_j,~\nabla_{e_1}e_3=a^j_3e_j,\nonumber\\
&&\nabla_{e_2}e_1=a^j_2e_j,~\nabla_{e_2}e_2=b^j_2e_j,~\nabla_{e_2}e_3=b^j_3e_j,~\\&&\nabla_{e_3}e_1=a^j_3e_j,~\nabla_{e_3}e_2=e_2+b^j_3e_j,~\nabla_{e_3}e_3=c^j_3e_j,\nonumber
\end{eqnarray}
for $j=1,2,3$, where $d^j_ke_j=\sum_{j=1}^3d_{jk}e_j$ and $a^j_k, b^j_k, c^j_k\in\R$.

$\ast$ \textbf{Case} 1: Flat Lie algebra with  central right-identity $e=e_1$. Using $(\ref{rightsimplify})$ we have $\varrho(e_1)=\mathrm{I}_{\h}$. This determines another $9$ structure constants by linear equations. The
remaining \textit{LSA}-structure equations $(\ref{LSAproperty})$ then are almost trivial. It is easy to see that they have two solutions
\begin{flalign}\label{solution1}
\nabla_{e_2}e_2&=(-{\tfrac {ac}{2}}+{\tfrac {{b}^{2}}{2}}+c)e_1-be_2+ce_3,~\nabla_{e_2}e_3=-{\tfrac { \left(  \left( a-2 \right) c-{b}^{2} \right) b}{2\,c}}e_1+{\tfrac { \left( a-2 \right) c-3\,{b}^{2}}{2
\,c}}e_2+be_3,&\nonumber \\
\nabla_{e_3}e_2&=-{\tfrac { \left(  \left( a-2 \right) c-{
b}^{2} \right) b}{2\,c}}e_1+
{\tfrac {ac-3\,{b}^{2}}{2\,c}}e_2+be_3,~\nabla_{e_3}e_3={\tfrac { \left( -{a}^{2}+4 \right) {c}^{2}+4
\,c{b}^{2}+{b}^{4}}{4\,{c}^{2}}}e_1-{\tfrac {b \left( {b}^{2}+c \right) }{{c}^{2}}}e_2+ae_3,&\\
\nabla_{e_1}e_j&=\nabla_{e_j}e_1=e_j,~j=1,2,3.&\nonumber
\end{flalign}
\begin{flalign}\label{solution2}
\nabla_{e_2}e_3&=a^\prime e_2,~\nabla_{e_3}e_2=(a^\prime +1)e_2,~\nabla_{e_3}e_3=a^\prime(a^\prime-b^\prime)e_1+c^\prime e_2+b^\prime e_3,&
\end{flalign}
with $~\nabla_{e_1}e_j=\nabla_{e_j}e_1=e_j,$ for $j=1,2,3$. Where, $a,a^\prime,b,b^\prime,c^\prime\in\R$, $c\in\R^\ast$.   The automorphism group of $3\G_{2,1}\oplus\G_1$ is
\begin{equation}
\begin{small}
\mathrm{Aut}(3\G_{2,1}\oplus\G_1)=\left\{\psi=\left( \begin {array}{ccc} x_{{11}}&0&x_{{13}}\\ \noalign{\medskip}
0&x_{{22}}&x_{{23}}\\ \noalign{\medskip}0&0&1\end {array} \right) 
~|~x_{1 1}x_{2 2} \neq0
\right\}.
\end{small}
\end{equation}
 To begin with, will simplify solution $(\ref{solution2})$, remember that two flat  algebras $(\h,\nabla)$ and $(\h,\tilde{\nabla})$ are isomorphic if and only if there is $\psi\in\mathrm{Aut}(3\G_{2,1}\oplus\G_1)$, such that $\tilde{\nabla}_x=\psi\circ\nabla_{\psi^{-1}(x)}\circ\psi^{-1}$. We have
 \begin{equation}\label{appisomo}
 \tilde{\nabla}_{e_2}=\psi\circ\nabla_{\psi^{-1}(e_2)}\circ\psi^{-1}~~\text{and}~~\tilde{\nabla}_{e_3}=\psi\circ\nabla_{\psi^{-1}(e_3)}\circ\psi^{-1},
\end{equation}
with $x_{11}=1$, $x_{23}=0$ and $x_{13}=a^\prime$, knowing the value of $c$ makes choosing $x_{22}$ preferable, this respects  $\varrho(e_1)=\mathrm{I}_\h$. Due to this, we have
\begin{equation}\label{struc1}
\tilde{\nabla}_{e_1}e_j=e_j, j=1,2,3,~\tilde{\nabla}_{e_2}e_1=e_2,~\tilde{\nabla}_{e_3}e_1=e_3,~\tilde{\nabla}_{e_3}e_2=e_2,~\tilde{\nabla}_{e_3}e_3=\epsilon e_2+(-2a+b)e_3,
\end{equation}
where, $a,b\in\R$, and $\epsilon=0,1$.
Using (\ref{appisomo}), we will extract all flat torsion-free connections that are not isomorphic, we obtain flat Lie algebras $\h_{4}^{a}, \h_{5}$ and $\h_{6}$ as given in Proposition \ref{Pr 3flat}.

 Using the first  solution $(\ref{solution1})$, by applying $(\ref{appisomo})$ we have
\begin{equation}\label{struc2}
\tilde{\nabla}_{e_1}e_j=e_j, j=1,2,3,~\tilde{\nabla}_{e_2}e_1=e_2,~\tilde{\nabla}_{e_2}e_2=\pm e_3,~\tilde{\nabla}_{e_3}e_1=e_3,~\tilde{\nabla}_{e_3}e_2= e_2,~\tilde{\nabla}_{e_3}e_3= 2e_3,
\end{equation}
with $x_{11}=1,~x_{22}=\frac{1}{\sqrt{|c|}}$, $x_{2 3} =\frac{bx_{2 2}}{c}$, $x_{1 3} =\frac{ac - b^2 - 2c}{2c}$. As a final note, in this case, we can easily show that the two flat torsion-free connections defined in $(\ref{struc1})$ and $(\ref{struc2})$ are not isomorphic no matter what $\psi\in\mathrm{Aut}(3\G_{2,1}\oplus\G_1)$.

$\ast$ \textbf{Case} 2: Flat Lie algebra with $e=e_1+e_2$.  Using $(\ref{rightsimplify})$ we have $\varrho(e_1)+\varrho(e_2)=\mathrm{I}_{\h}$. This determines another $9$ structure constants by linear equations. The
remaining \textit{LSA}-structure equations $(\ref{LSAproperty})$ are very simple. It is easy to see that they have
a unique solution, which is given by 
\begin{equation}\label{struc3}
\begin{split}
\nabla_{e_1}e_1=e_1-e_2,~\nabla_{e_1}e_2=e_2,~\nabla_{e_1}e_3=-(a+1)e_2+e_3,~\nabla_{e_2}e_1=e_2,~\nabla_{e_2}e_3=ae_2,\\
\nabla_{e_3}e_1=-(a+1)e_2+e_3,~\nabla_{e_3}e_2=(a+1)e_2,~\nabla_{e_3}e_3=-(a^2+a)e_1+be_2+(2a+1)e_3,
\end{split}
\end{equation}
where $a,b\in\R$. Furthermore, the flat torsion-free connection $(\ref{struc3})$ (with $x_{1 1} = 1, x_{2 2} = 1, x_{1 3} = a$ and this respects  $\varrho(e_1+e_2)=\mathrm{I}_\h$ ) can even be simplified by applying $(\ref{appisomo})$, resulting in 
\begin{equation}
\begin{split}
\tilde{\nabla}_{e_1}e_1=e_1-e_2,~\tilde{\nabla}_{e_1}e_2=e_2,~\tilde{\nabla}_{e_1}e_3=-e_2+e_3,~\tilde{\nabla}_{e_2}e_1=e_2,\\
\tilde{\nabla}_{e_3}e_1=-e_2+e_3,~\tilde{\nabla}_{e_3}e_2=e_2,~\tilde{\nabla}_{e_3}e_3=a^\prime e_2+e_3,
\end{split}
\end{equation}

$\ast$ \textbf{Case} 3: Flat Lie algebra with $e=\lambda e_3$. Assume first that $\lambda\neq0$. Then $\varrho(\lambda e_3)=\mathrm{I}_\h$ determines
$7$ structure constants. It is straightforward to solve the remaining \textit{LSA}-structure equations and to see that they have a unique solution, which is given by
\begin{eqnarray*}
&&\nabla_{e_1}e_1=ae_1-\lambda b(b-a) e_3,~\nabla_{e_1}e_2=be_2,~\nabla_{e_1}e_3=\tfrac{1}{\lambda} e_1,~\nabla_{e_2}e_1=be_2,\\
&&\nabla_{e_2}e_3=\tfrac{1}{\lambda} e_2,~\nabla_{e_3}e_1=\tfrac{1}{\lambda} e_1,~\nabla_{e_3}e_2=\tfrac{1+\lambda}{\lambda} e_2,~\nabla_{e_3}e_3=\tfrac{1}{\lambda}e_3,
\end{eqnarray*}
where $a,b\in\R$ and $\lambda\in\R^\ast$. As a matter of cohomological calculus (see, Table \ref{tableofcoho}), we will extract  the following flat Lie algebras, which are given by  $\h_{10}^{\lambda,a,b}$, $\h_{11}^{\lambda}$, $\h_{12}^{\lambda}$, $\h_{13}^{\lambda}$ and $\h_{14}^{\lambda}$  in Proposition \ref{Pr 3flat}.

$\ast$ \textbf{Case} 4: Flat Lie algebra with $e=e_2$. After a short calculation we obtain a contradiction. This contradiction arises from $\varrho(e_2)=\mathrm{I}_\h$.
\end{proof}

\section{Six-dimensional Frobeniusian Lie algebras}\label{sec4}

The following remarkable result summarizes our main observation: 
\begin{theo}\label{Principale}
Let $(\G,\mathrm{d}\mu)$ be a  Frobeniusian Lie algebra of dimension $n\leq6$. Assume that $\mathrm{d}\mu$ has a Lagrangian ideal. Then, for every choice of symplectic form $\omega$ on six-dimensional Frobeniusian Lie algebra $\G$, the symplectic Lie algebra $(\G,\omega)$ has  at least one Lagrangian ideal. 
\end{theo}
\begin{proof}
We can conclude this statement by analyzing symplectic forms in general position for each Fobeniusian Lie algebra so that their exact form admits a Lagrangian ideal. Table $\ref{cocycleofall}$ shows symplectic forms in general position, and Proposition $\ref{Prclassifi}$ gives all possible Lagrangian ideals for each form.
\end{proof}

\begin{remark}\label{Remark1-1}
As a result of Theorem $\ref{Principale}$, when we consider the one-to-one correspondence between the isomorphism classes of Lagrangian extensions over a flat Lie algebra and the group $H^2_{L,\rho}(\h,\h^\ast)$, then the classification of all symplectic structures, up to a symplectomorphism on a six-dimensional Frobeniusian Lie algebra $(\G,\mathrm{d}\mu)$ for which $\mathrm{d}\mu$ admits a Lagrangian ideal, reduces to a classification of all flat torsion-free connections derived from the symplectic structure.
\end{remark}

Two symplectic Lie algebras $(\G_1,\omega_1)$ and $(\G_2,\omega_2)$ are said to be
symplectomorphically equivalent if there exists an isomorphism of Lie algebras $\varphi : \G_1\longrightarrow\G_2$, which preserves the symplectic forms in the sense $\varphi^*\omega_2=\omega_1$.  As a result of direct computation, we have:
\begin{pr}\label{Principale2}
Let $(\G,\omega)$ be a six-dimensional Frobeniusian indecomposable Lie algebra whose exact form does not admit a Lagrangian ideal. Then, $(\G,\omega)$ is isomorphic to exactly one of the following symplectic Lie algebras$:$
\begin{center}
\captionof{table}{Six-dimensional Frobeniusian indecomposable Lie algebras 
without Lagrangian ideal.}\label{symwithnoLag}
{\renewcommand*{\arraystretch}{1.4}
\small\begin{tabular}{l l }
\hline
		Algebra& Symplectic structure \\
			\hline
			$N^{\alpha\neq 0,\beta}_{6,35}$&$\omega^{\pm}_1=e^{16}+\frac{\beta}{\alpha}e^{26}\pm(e^{23}+\frac{1}{2}e^{45})$\\
&$\omega^{\pm}_2=\mu e^{12}+e^{16}+\frac{\beta}{\alpha}e^{26}\pm(e^{23}+\frac{1}{2}e^{45})$, $\mu\neq 0$\\
$N^{\alpha}_{6,37}$&$\omega^{\pm}_1=\pm (e^{16}+2e^{23}+e^{45})$\\
&$\omega^{\pm}_2=\nu e^{12}\pm (e^{16}+2e^{23}+e^{45})$, $\nu\neq 0$\\
$\G_{6,89}$&$\omega^\pm_1=\pm(e^{16}+\frac{1}{2}e^{24}+\frac{1}{2}e^{35})$\\
&$\omega^\pm_2=\pm(e^{16}+\frac{1}{2}e^{24}+\frac{1}{2}e^{35}+e^{46})$\\
$\G_{6,90}$&$\omega^{\pm}=\pm(e^{16}+\frac{1}{2}e^{24}+\frac{1}{2}e^{35})$\\
$\G_{6,92}^{\prime}$&$\omega^{\pm}=\pm(e^{16}+\frac{1}{2} e^{24}+\frac{1}{2}e^{35})$\\
$\G_{6,93}^{\nu_0=0}$&$\omega_1=e^{16}+\frac{1}{2}e^{24}+\frac{1}{2}e^{35}$\\
&$\omega_2=e^{16}+\frac{1}{2}e^{24}+\lambda e^{26}+\frac{1}{2}e^{35}$,~$\lambda\neq0$\\
$\G_{6,93}^{\nu_0\neq0}$&$\omega=e^{16}+\frac{1}{2}e^{24}+\frac{1}{2}e^{35}$
			\\\hline
\end{tabular}}
\end{center}
\end{pr}
\begin{proof}
The proof uses the automorphisms of symplectic Lie algebras. We can say that $\omega_1$ and $\omega_2$ are not symplectomorphically isomorphic (As a result of exactness).

\end{proof}

Proposition $\ref{Prclassifi}$ has a consequence, which we will now discuss.

\begin{co}\label{decompoFrob}
Let $(\G,\omega)$ be a six-dimensional Frobeniusian  decomposable Lie algebra. Then, $(\G,\omega)$ is isomorphic to exactly one of the following symplectic Lie algebras$:$
\begin{center}
\captionof{table}{Six-dimensional Frobeniusian decomposable Lie algebras.}\label{decoFro}
{\renewcommand*{\arraystretch}{1.4}
\small\begin{tabular}{l l l}
\hline
		Algebra& Symplectic structure&Remarks \\
			\hline
$\mathfrak{r}_2\mathfrak{r}_2\oplus\mathfrak{aff}(1,\R)$&$\omega=e^{12}+e^{34}+e^{56}+\tau_1 e^{13}+\tau_2 e^{15}+\tau_3 e^{35}$&$\tau_1,\tau_2,\tau_3\in\R$\\
$\mathfrak{r}_2^\prime\oplus\mathfrak{aff}(1,\R)$&$\omega= e^{14}+e^{23}+e^{56}+\tau_1 e^{12}+\tau_2 e^{15}+\tau_3 e^{25}$&$\tau_1,\tau_2,\tau_3\in\R$\\
$\mathfrak{d}_{4,1}\oplus\mathfrak{aff}(1,\R)$&$\omega_1=e^{12}-e^{34}+e^{56}$&\\
&$\omega_2=e^{12}-e^{34}+\tau e^{45}+e^{56}$&\\
&$\omega_3=e^{12}+e^{24}-e^{34}+\tau^\prime e^{45}+e^{56}$&\\
&$\omega_4=e^{12}+\tau^{\prime\prime}e^{24}+e^{25}-e^{34}+\tau^\prime e^{45}+e^{56}$&$\tau\neq0$, $\tau^{\prime\prime},\tau^\prime\in\R$\\
$\mathfrak{d}_{4,2}\oplus\mathfrak{aff}(1,\R)$&$\omega_1=e^{12}-e^{34}+e^{56}$&\\
&$\omega_2=e^{12}-e^{34}+\tau e^{45}+e^{56}$&\\
&$\omega_3=e^{12}-\frac{1+\epsilon}{\epsilon}e^{14}+e^{23}-e^{34}+e^{56}$&\\
&$\omega_4=e^{12}-\frac{1+\epsilon}{\epsilon}e^{14}+e^{23}-e^{34}+\tau^\prime e^{45}+e^{56}$&$\tau\tau^\prime\neq0$, $\epsilon=\pm1$\\
$\mathfrak{d}_{4,\lambda}\oplus\mathfrak{aff}(1,\R)$&$\omega_1=e^{12}-e^{34}+e^{56}$&\\
&$\omega_2=e^{12}-e^{34}+\tau e^{45}+e^{56}$&$\lambda\geq\frac{1}{2}$, $\lambda\neq 1,2$, $\tau\neq0$\\
$\mathfrak{d}_{4,\delta}^\prime\oplus\mathfrak{aff}(1,\R)$&$\omega=\pm(e^{12}-\delta e^{34})+\gamma e^{45}+e^{56}$&$\gamma\in\R, \delta>0$\\
$\mathfrak{h}_4\oplus\mathfrak{aff}(1,\R)$&$\omega_1^\pm=\pm(e^{12}-e^{34})+e^{56}$&\\
&$\omega_2^\pm
=\pm(e^{12}-e^{34})+\tau e^{45}+e^{56}$& $\tau\neq0$ 			
\\\hline
\end{tabular}}
\end{center}
\end{co}
\begin{proof}
Except for $\mathfrak{d}^\prime_{4,\delta\neq0}\oplus\mathfrak{aff}(1,\R)$, all Frobeniusian decomposable Lie algebras given in Table $\ref{decoFro}$ admit a Lagrangian ideal. From \cite{Ova}, any symplectic form on $\mathfrak{d}^\prime_{4,\delta\neq0}$ is symplectomorphically equivalent to $\omega=\pm(e^{12}-\delta e^{34})$. As a consequence, any symplectic form on  $\mathfrak{d}^\prime_{4,\delta\neq0}\oplus\mathfrak{aff}(1,\R)$ is symplectomorphically equivalent to 
\begin{equation*}
\omega=\pm(e^{12}-\delta e^{34})+\gamma e^{45}+e^{56},~\gamma\in\R.
\end{equation*} 
\end{proof}
\section{Classification of Frobeniusian Lie algebras with Lagrangian ideals}\label{Classifisection}
In this section, we give the classification of six-dimensional Frobeniusian Lie algebras that have a Lagrangian ideal. We describe the main idea of our classification in the following paragraphs:\\
Every symplectic Lie algebra $(\G,\omega)$, which has a Lagrangian ideal $\mathfrak{j}$, arises as a Lagrangian extension of a flat Lie algebra \cite{BV}. In the quotient Lie algebra $\h=\G/\mathfrak{j}$, the flat torsion-free  connection $\nabla=\nabla^\omega$ satisfies the following relation
\begin{equation}\label{torsion-free}
\omega_\h(\nabla_{\overline{x}}\overline{y},z)=-\omega(y,[x,z]),~~\forall \overline{x},\overline{y}\in\h, z\in\mathfrak{j}.
\end{equation}

Let $(\G,\mathrm{d}\mu)$ be a six-dimensional Frobeniusian Lie algebra. Assume that, $\mathrm{d}\mu$ has a Lagrangian ideal $\mathfrak{j}$, then $(\G/\mathfrak{j},\nabla^{\mathrm{d}\mu})\cong(\h,\nabla,e)$ is a flat Lie algebra with a right-identity element $``e"$. Thus, the classification of this kind of Lie algebra reduces to that of three-dimensional  flat Lie algebras  with right-identity elements. Our main results fall into six steps:
\begin{enumerate}
\item[$1$.] Using Lemma \ref{righ-identity}, we can classify all flat Lie algebras with right-identity elements. This step is summarized in Table \ref{tableleftt}.

\item[$2$.] The reconstruction of Frobeniusian Lie algebras with Lagrangian ideal: Upon completion of the first step, we obtain the list of Frobeniusian Lie algebras with Lagrangian ideal in Proposition \ref{Prclassifi}. Combining Proposition \ref{exactLagrang}, Table \ref{tableleftt} and  \ref{tableofcoho} successively is required to complete this step. Consequently, we obtain all the symplectic forms that are not symplectomorphically equivalent, such that their associated flat torsion-free connections $\nabla^\omega$ arise from a flat Lie algebra with a right-identity element $(\h,\nabla,e)$.
\item[$3$.] We analyze symplectic forms on six-dimensional Frobeniusian
Lie algebras given by Proposition \ref{Prclassifi}. The proof follows by working on each Lie algebra. We first compute the 2-cocycles, (i.e., the 2-forms which verifies (\ref{cocy}) ) the next step is to compute the
rank of $\omega$. If $\omega$ has maximal rank, that is, $\wedge^3\omega\neq0$ then $\G$ will be endowed with a symplectic structure. Table \ref{cocycleofall} gives a description of all the 2-cocycles.

\item[$4.$] Using Table  \ref{cocycleofall}. For each symplectic Lie algebra described by Proposition \ref{Prclassifi}, we determine all possible Lagrangian ideals (As a result of theorem \ref{Principale}).

\item[$5.$] Let $\G_{\nabla,\alpha}=\tilde{\h}\oplus\tilde{\h}^\ast$ be a six-dimensional Frobeniusian Lie algebra (taken from Proposition \ref{Prclassifi}),  $\omega$ its symplectic form in general position (see Table \ref{cocycleofall}), and $\mathfrak{j}$ be a Lagrangian ideal of $(\G,\omega)$.  Then, $(\h=\G/\mathfrak{j},\psi^\ast\nabla^{\omega})$ is a flat Lie algebra, where, \footnote{To classify flat torsion-free connections $\nabla^{\omega}$ up to isomorphism, we use $\psi$.} $\psi : \h\longrightarrow\h$ is an automorphism, and $\psi^\ast\nabla^{\omega}_x=\psi\circ\nabla^{\omega}_{\psi^{-1}(x)}\circ\psi^{-1}$ for all $x\in\h$ (Recall  that the  flat torsion-free connection $\nabla^{\omega}$ on the quotient Lie algebra $\h=\G/\mathfrak{j}$ satisfies the relation $(\ref{torsion-free})$).
\begin{enumerate}
\item[$(i)$]  It is finished if  $(\h,\psi^\ast\nabla^{\omega})$ has a right-identity element, and $(\h,\psi^\ast\nabla^{\omega})\cong(\tilde{\h},\nabla)$. In accordance with Remark \ref{Remark1-1} and Table \ref{tableofcoho}, the non-symplectomorphically equivalent symplectic forms on $\G$ are given by
\begin{equation}\label{Geneform}
\omega_{[\alpha]_{L,\rho}}(x,y+\xi)=-\chi(x,y)\oplus\omega(x,\xi),~x\in\h, \xi\in\h^\ast,
 \end{equation}
 where $\chi=(\sigma-\mu)- {}^t\left(\sigma-\mu\right)\in\bigwedge^2\h^\ast$ for some $\sigma\in\mathcal{C}^1(\h,\h^\ast)$, and $\mu\in\mathcal{C}^1_L(\h,\h^\ast)$. Class $\alpha$ in $H^2_{L,\rho}(\h,\h^\ast)$ designed by $[\alpha]_{L,\rho}$.
\item[$(ii)$] If $(\B,\overline{\nabla})=(\h,\psi^\ast\nabla^{\omega})$ has no right-identity element or $(\B,\overline{\nabla})\ncong(\tilde{\h},\nabla)$, Table \ref{flatfromsymplectic} shows the classification of flat Lie algebras of this type. We compute the Lagrangian extension cohomology group $H^2_{L,\rho}(\B,\B^\ast)$ for the flat Lie algebra $\B$ (see Table \ref{cohoforflatsym}). Every symplectic structure on $\G_{\overline{\nabla},\alpha^\prime}=\B\oplus\B^\ast$ is thus symplectomorphically equivalent  to $\omega_{[\alpha^\prime]_{L,\rho}}$, where $\omega_{[\alpha^\prime]_{L,\rho}}$ is given as $(\ref{Geneform})$.  As a final step, an isomorphism is then applied to $\G_{\overline{\nabla},\alpha^\prime}$ to return $\G_{\nabla,\alpha}$ (see Table \ref{isofrom new to new}).
\end{enumerate}
\item[$6.$] The last step, and in order to get Turkowski  algebras \cite{Tur3} (Mubarakzyanov    algebras \cite{Muba1}), we send the new Frobeniusian Lie algebras (see Tables \ref{isoMuba} and \ref{isoTur}).
\end{enumerate}
\begin{remark}
Proposition $\ref{Prclassifi}$ provides a complete classification of six-dimensional Frobeniusian Lie algebras, as well as their symplectic structures. Despite the fact that it is not very common (but does exist), there are probably two algebras that are isomorphic. Based on Proposition $\ref{LEiso}$, we are able to locate them and eliminate one of them $($see Tables $\ref{isoTur}$ and $\ref{isoMuba})$.
\end{remark}
This method leads to the following result.

\begin{pr}\label{Prclassifi}
Let $(\G,\omega)$ be a six-dimensional Frobeniusian Lie algebra for which its exact form possesses a Lagrangian ideal. Then, $(\G,\omega)$ is isomorphic to exactly one of the following symplectic Lie algebras$:$
\end{pr}
\begin{flushleft}
{\renewcommand*{\arraystretch}{1.4}
\small
\begin{tabular}{l l }
$\G_{6,1}$ :&	$[e_1,e_4]=-e_6,~[e_2,e_5]=-e_6,~[e_3,e_4]=-e_4,~[e_3,e_5]=-e_5,~[e_3,e_6]=-e_6$,	\\
&$\mathfrak{j}_1=\langle e_4,e_5,e_6\rangle,~\mathfrak{j}_2=\langle e_1,e_2,e_6\rangle,~\mathfrak{j}_3=\langle e_2,e_4,e_6\rangle,~\mathfrak{j}_4=\langle e_1,e_5,e_6\rangle,$\\
&$\omega_{[\alpha]_{L,\rho}}=\tau_1 e^{12}+\tau_2 e^{13}+\tau_3 e^{23}+\omega_0,~\tau_1,\tau_2,\tau_3,\in\R.$\\
$\G^{b}_{6,2}:$&$[e_1,e_4]=-e_4-e_6,~[e_1,e_5]=-e_4,~[e_1,e_6]=-be_5,~[e_2,e_4]=-be_5,~[e_2,e_5]=-e_6,$\\
&$[e_2,e_6]=-be_4+be_5,~[e_3,e_4]=-e_4,~[e_3,e_5]=-e_5,~[e_3,e_6]=-e_6,~b\in\R^\ast$,\\
&$\mathfrak{j}=\mathfrak{nil}(\G^{b}_{6,2})=\langle e_4,e_5,e_6\rangle,$\\
&$\omega_{[\alpha]_{L,\rho}}=\tau_3 e^{12}-\tau_1 e^{13}-\tau_2 e^{23}-\omega_0,~\tau_1,\tau_2,\tau_3\in\R.$\\
$\G_{6,3}:$&$[e_1,e_4]=-e_6,[e_2,e_4]=-e_5,~[e_2,e_5]=-e_6,~[e_3,e_4]=-e_4,~[e_3,e_5]=-e_5,~[e_3,e_6]=-e_6,$\\
\end{tabular}}
\end{flushleft}		
\begin{flushleft}
{\renewcommand*{\arraystretch}{1.4}
\small\begin{tabular}{l l }
 $\G_{6,8}:$&$[e_2,e_3]=-e_2,~[e_1,e_4]=-e_4,~[e_1,e_5]=-e_5,~[e_1,e_6]=-e_6,~[e_2,e_5]=-e_4,~[e_2,e_6]=e_5,$\\
 &$[e_3,e_5]=-e_5,~[e_3,e_6]=-e_4-2\,e_6,$\\
 &$\mathfrak{j}=\langle e_4,e_5,e_6\rangle,$\\
 &$\omega_{[\alpha]_{L,\rho}}=\tau e^{13}+\omega_0,$\\
 &$\omega_{[\alpha^\prime]_{L,\rho}}=\tau^\prime e^{13}+e^{14}-e^{16}+e^{25}-e^{36},~\tau,\tau^\prime\in\R.$\\
$ \G_{6,9}^{a}:$&$[e_2,e_3]=-e_2,~[e_1,e_4]=-e_4,~[e_1,e_5]=e_4-e_5+e_6,~[e_1,e_6]=-e_6,~[e_2,e_5]=-e_4,$\\
 &$[e_3,e_5]=e_4-e_5-ae_6,~[e_3,e_6]=-e_4-e_6,~a\in\R,$\\
 &$\mathfrak{j}_1=\langle e_4,e_5,e_6\rangle,~\mathfrak{j}_2=\langle e_2,e_4,e_6\rangle,$\\
\end{tabular}}
\end{flushleft}	
\begin{flushleft}
{\renewcommand*{\arraystretch}{1.4}
\small\begin{tabular}{l l }
$\G_{6,12}^{\lambda=-2}:$&$[e_2,e_3]=-e_2,~[e_1,e_4]=\tfrac{1}{2}e_6,~[e_1,e_5]=-e_5,~[e_1,e_6]=2 e_4,~[e_2,e_5]=-e_4+\tfrac{1}{2} e_6,~[e_3,e_4]=\tfrac{1}{2}e_4,$\\
 &$[e_3,e_5]=-\tfrac{1}{2}e_5,~[e_3,e_6]=\tfrac{1}{2}e_6,$\\
 &$\mathfrak{j}_1=\langle e_4,e_5,e_6\rangle,~\mathfrak{j}_2=\langle e_2,e_4,e_6\rangle,$\\
 &$\omega_{[\alpha]_{L,\rho}}=2\,\tau e^{13}+\omega_0,~\tau\neq0,$\\
  &$\omega_{[\alpha^\prime]_{L,\rho}}=2\,\tau^\prime e^{13}-e^{45}+2\,e^{56}+\omega_0,~\tau,\tau^\prime\in\R.$\\
 $ \G_{6,13,1}^{\lambda\neq -\frac{1}{2}}:$&$[e_2,e_3]=-e_2,~[e_1,e_4]=-e_4-\tfrac{1}{\lambda}e_6,~[e_2,e_5]=-\tfrac{1}{\lambda}e_6,~[e_3,e_4]=-\tfrac{1}{\lambda}e_4,~[e_3,e_5]=-\tfrac{1+\lambda}{\lambda}e_5,$\\
 &$[e_3,e_6]=-\tfrac{1}{\lambda}e_6,~\lambda\in\R^\ast,~\text{with}~\lambda\neq\{\pm1,-2\},$\\
 &$\mathfrak{j}_1 =\langle e_4,e_5,e_6\rangle,~\mathfrak{j}_2 =\langle e_2,e_4,e_6\rangle,$\\
\end{tabular}}
\end{flushleft}	
 \begin{flushleft}
{\renewcommand*{\arraystretch}{1.4}
\small\begin{tabular}{l l }
$\G_{6,14,1}^{\lambda=-2}:$ &$[e_2,e_3]=-e_2,~[e_1,e_4]=\tfrac{1}{2}e_6,~[e_2,e_5]=\tfrac{1}{2}e_6,~[e_3,e_4]=\tfrac{1}{2}e_4,~[e_3,e_5]=-\tfrac{1}{2}e_5,~[e_3,e_6]=\tfrac{1}{2}e_6,$\\
 &$\mathfrak{j}_1 =\langle e_4,e_5,e_6\rangle,~\mathfrak{j}_2 =\langle e_1,e_2,e_6\rangle,~\mathfrak{j}_3 =\langle e_1,e_5,e_6\rangle,~\mathfrak{j}_4 =\langle e_2,e_4,e_6\rangle,$\\
 &$\omega_{[\alpha]_{L,\rho}}=2\,\tau e^{13}+\omega_0,$\\
 &$\omega_{[\alpha^\prime]_{L,\rho}}=-2\,\tau^\prime e^{13}+e^{45}+\omega_0,~\tau,\tau^\prime\in\R.$\\
$\G_{6,14,2}^{\lambda=-1}:$ &$[e_1,e_2]=ae_4,~[e_1,e_3]=be_5,~
 [e_2,e_3]=-e_2+be_4,~[e_1,e_4]=e_6,~[e_2,e_5]=e_6,
 ~[e_3,e_4]=e_4,$\\
 &$[e_3,e_6]=e_6,
 ~a\in\R,b\in\R^\ast,$\\
 &$\mathfrak{j}_1 =\langle e_4,e_5,e_6\rangle,~\mathfrak{j}_2 =\langle e_2,e_4,e_6\rangle,$\\
 &$\omega_{[\alpha]_{L,\rho}}=\omega_0,$\\
\end{tabular}}
\end{flushleft}
 \begin{flushleft}
{\renewcommand*{\arraystretch}{1.4}
\small\begin{tabular}{l l }
$\G_{6,17,3}^{\mu=\frac{1}{3}}:$&$[e_1, e_2] = be_5,~ [e_1, e_3] = e_1,~ [e_1, e_4] = -3e_6,~[e_2, e_3] = e_1 + e_2,~[e_2, e_5] = -3e_6,$\\
 &$[e_3, e_4] = -2e_4 + e_5,~ [e_3, e_5] = -2e_5,~ [e_3, e_6] = -3e_6,~b\in\R,$\\
&$\mathfrak{j}_1 =\langle e_4,e_5,e_6\rangle,\mathfrak{j}_2 =\langle e_1,e_5,e_6\rangle,$\\
 &$\omega_{[\alpha]_{L,\rho}}=e^{14}+e^{25}+e^{36}.$\\
$\G_{6,18,1}^{\eta\neq\{\frac{1}{2},\frac{1}{3}\}}:$&$[e_1,e_3]=e_1,~[e_2,e_3]=e_2,~
[e_1,e_4]=-\tfrac{1}{\eta}e_6,~[e_2,e_5]=-\tfrac{1}{\eta}e_6,~[e_3,e_4]=-\tfrac{1-\eta}{\eta}e_4,~[e_3,e_5]=-\tfrac{1-\eta}{\eta}e_5, $\\
&$[e_3,e_6]=-\tfrac{1}{\eta}e_6,~\eta\in\R^\ast,~~\text{with}~\eta\neq1,$
\\
&$\mathfrak{j}_1 =\langle e_4,e_5,e_6\rangle,\mathfrak{j}_2 =\langle e_1,e_5,e_6\rangle,\mathfrak{j}_3 =\langle e_2,e_4,e_6\rangle,\mathfrak{j}_4 =\langle e_1,e_2,e_6\rangle,$\\
&$\omega_{[\alpha]_{L,\rho}}=e^{14}+e^{25}+e^{36}.$\\
\end{tabular}}
\end{flushleft}
\begin{flushleft}
{\renewcommand*{\arraystretch}{1.4}
\small\begin{tabular}{l l }
$\G_{6,19,1}^{\alpha=\frac{1}{3},\gamma=\frac{3}{2}}:$&$[e_1,e_3]=e_1,~[e_2,e_3]=\tfrac{1}{3} e_2,~[e_1,e_4]=-\tfrac{2}{3} e_6,~
[e_2,e_5]=-\tfrac{2}{3}e_6,~[e_3,e_4]=\tfrac{1}{3}e_4,~[e_3,e_5]=-\tfrac{1}{3}e_5,$\\
&$[e_3,e_6]=-\tfrac{2}{3}e_6,$\\
&$\mathfrak{j}_1 =\langle e_4,e_5,e_6\rangle,\mathfrak{j}_2 =\langle e_1,e_2,e_6\rangle,\mathfrak{j}_3 =\langle e_1,e_5,e_6\rangle,\mathfrak{j}_4 =\langle e_2,e_4,e_6\rangle,$\\
&$\omega_{[\alpha]_{L,\rho}}=e^{14}+e^{25}+e^{36},$\\
&$\omega_{[\alpha^\prime]_{L,\rho}}=e^{24}+\omega_0.$\\
$\G_{6,19,1}^{\alpha=1-\frac{1}{\gamma},\gamma}:$&$[e_1,e_3]=e_1,~[e_2,e_3]=(1-\tfrac{1}{\gamma}) e_2,~[e_1,e_4]=-\tfrac{1}{\gamma}e_6,~
[e_2,e_5]=-\tfrac{1}{\gamma}e_6,~[e_3,e_4]=(1-\tfrac{1}{\gamma})e_4,$\\
&$[e_3,e_5]=(1-\tfrac{2}{\gamma})e_5,~[e_3,e_6]=-\tfrac{1}{\gamma}e_6,~\gamma>\tfrac{1}{2}~~\text{and}~~\gamma\neq1,2,\tfrac{3}{2}$\\
&$\mathfrak{j}_1 =\langle e_4,e_5,e_6\rangle,\mathfrak{j}_2 =\langle e_1,e_2,e_6\rangle,\mathfrak{j}_3 =\langle e_1,e_5,e_6\rangle,\mathfrak{j}_4 =\langle e_2,e_4,e_6\rangle,$\\
\end{tabular}}
\end{flushleft}
\begin{flushleft}
{\renewcommand*{\arraystretch}{1.4}
\small\begin{tabular}{l l }

$\G_{6,19,3}^{\alpha=-1+\frac{1}{\gamma},\gamma}:$&$[e_1,e_3]=e_1+be_5,~[e_2,e_3]=(-1+\tfrac{1}{\gamma})e_2+be_4,~[e_1,e_4]=-\tfrac{1}{\gamma}e_6,~
[e_2,e_5]=-\tfrac{1}{\gamma}e_6,~[e_3,e_4]=(1-\tfrac{1}{\gamma})e_4,$\\
&$[e_3,e_5]=-e_5,~[e_3,e_6]=-\tfrac{1}{\gamma}e_6,~b\in\R,~~\gamma>\tfrac{1}{2},~~\text{with}~~\gamma\neq1,$\\
&$\mathfrak{j}_1 =\langle e_4,e_5,e_6\rangle,\mathfrak{j}_2 =\langle e_1,e_2,e_6\rangle_{b=0},\mathfrak{j}_3 =\langle e_1,e_5,e_6\rangle,\mathfrak{j}_4 =\langle e_2,e_4,e_6\rangle,$\\
&$\omega_{[\alpha]_{L,\rho}}=e^{14}+e^{25}+e^{36}.$\\
$\G^{\alpha,\gamma=\tfrac{1}{2}}_{6,19,4}:$&$[e_1,e_3]=e_1+ce_4,~[e_2,e_3]=\alpha e_2,~[e_1,e_4]=-2e_6,~
[e_2,e_5]=-2e_6,~[e_3,e_4]=-e_4,~[e_3,e_5]=(\alpha-2)e_5,$\\
&$[e_3,e_6]=-2e_6,~c\in\R~\text{with}~0<|\alpha|<1,~~\text{with}~~\alpha\neq\tfrac{1}{2},$\\
&$\mathfrak{j}_1 =\langle e_4,e_5,e_6\rangle,\mathfrak{j}_2 =\langle e_1,e_2,e_6\rangle_{c=0},\mathfrak{j}_3 =\langle e_1,e_5,e_6\rangle_{c=0},\mathfrak{j}_4 =\langle e_2,e_4,e_6\rangle,$\\
&$\omega_{[\alpha]_{L,\rho}}=e^{14}+e^{25}+e^{36}.$\\
\end{tabular}}
\end{flushleft}

\begin{flushleft}
{\renewcommand*{\arraystretch}{1.4}
\small\begin{tabular}{l l }
$\G_{6,20,1}^{\gamma\neq\{1,\frac{2}{3},\frac{2}{5}\}}:$&$[e_1,e_3]=e_1,~[e_2,e_3]=\tfrac{1}{2}e_2,~[e_1,e_4]=-\tfrac{1}{\gamma}e_6,~[e_2,e_4]=-e_5,~[e_2,e_5]=-\tfrac{1}{\gamma}e_6,~[e_3,e_4]=(1-\tfrac{1}{\gamma})e_4,$\\
&$[e_3,e_5]=(\tfrac{1}{2}-\tfrac{1}{\gamma})e_5,~[e_3,e_6]=-\tfrac{1}{\gamma}e_6,~\gamma\in\R^\ast, ~\text{with}~ \gamma\neq-2,\tfrac{4}{3},$\\
&$\mathfrak{j}_1 =\langle e_4,e_5,e_6\rangle,\mathfrak{j}_2 =\langle e_1,e_5,e_6\rangle,$\\
&$\omega_{[\alpha]_{L,\rho}}=e^{14}+e^{25}+e^{36}.$\\
$\G_{6,20,1}^{\gamma=-2}:$&$[e_1,e_3]=e_1,~[e_2,e_3]=\tfrac{1}{2}e_2,~[e_1,e_4]=\tfrac{1}{2}e_6,~[e_2,e_4]=-e_5,~[e_2,e_5]=\tfrac{1}{2}e_6,~[e_3,e_4]=\tfrac{3}{2}e_4,$\\
&$[e_3,e_5]=e_5,~[e_3,e_6]=\tfrac{1}{2}e_6,$\\
&$\mathfrak{j}_1 =\langle e_4,e_5,e_6\rangle,\mathfrak{j}_2 =\langle e_1,e_5,e_6\rangle,$\\
&$\omega_{[\alpha]_{L,\rho}}=e^{14}+e^{25}+e^{36},$\\
\end{tabular}}
\end{flushleft}

\begin{flushleft}
{\renewcommand*{\arraystretch}{1.4}
\small\begin{tabular}{l l }
$\G^{\nu=1}_{6,21,3}:$&$[e_1,e_2]=b e_4,~[e_1,e_3]=e_1,~[e_2,e_3]=-e_2,~[e_1,e_4]=-e_6,~[e_2,e_5]=-e_6,$\\
&$[e_3,e_5]=-2e_5,~[e_3,e_6]=-e_6,~b\in\R^\ast,$\\
&$\mathfrak{j}_1 =\langle e_4,e_5,e_6\rangle, \mathfrak{j}_2 =\langle e_2,e_4,e_6\rangle,$\\
&$\omega_{[\alpha]_{L,\rho}}=e^{14}+e^{25}+e^{36}.$
\\
$\G^{\nu=1}_{6,21,3,b=0}:$&$[e_1,e_3]=e_1,~[e_2,e_3]=-e_2,~[e_1,e_4]=-e_6,~[e_2,e_5]=-e_6,~[e_3,e_5]=-2e_5,~[e_3,e_6]=-e_6,$\\
&$\mathfrak{j}_1 =\langle e_4,e_5,e_6\rangle,\mathfrak{j}_2 =\langle e_1,e_2,e_6\rangle,\mathfrak{j}_3 =\langle e_1,e_5,e_6\rangle, \mathfrak{j}_4 =\langle e_2,e_4,e_6\rangle,$\\
&$\omega_{[\alpha]_{L,\rho}}=\tau e^{12}+\omega_0,$\\
&$\omega_{[\alpha^\prime]_{L,\rho}}=\tau^\prime e^{12}+e^{34}+\omega_0,~~\tau,\tau^\prime
\in\R.$\\
\end{tabular}}
\end{flushleft}

\begin{proof}
This is the general scheme of the proof. We take a flat Lie algebra, say $(\h,\nabla)$, from the ones obtained in Proposition \ref{Pr 3flat}. This flat Lie algebra has a basis   $\mathcal{B}=\{e_1,e_2,e_3\}$. Let $(\G_{\nabla,\alpha},\omega)$ be the Lagrangian extension of the flat Lie algebra $(\h,\nabla)$ with respect to $\alpha\in Z^2_{L,\rho}(\h,\h^\ast)$. Let $\sigma\in\mathcal{C}^1(\h,\h^\ast)$ and $\alpha^\prime\in Z^2_{L,\rho}(\h,\h^\ast)$, such that, $\alpha^\prime=\alpha-\partial_\rho\sigma$. In this case, $\G_{\nabla,\alpha}\cong\G_{\nabla,\alpha^\prime}$, thus, Lie algebra $\G_{\nabla,\alpha^\prime}$ will be the subject of the scheme, this Lie algebra is isomorphic to a family
of Lie algebras in Proposition \ref{Prclassifi}, and their Lie brackets are given by $(\ref{Liebracket1})$ and $(\ref{Liebracket2})$. By doing so, we collect all the Lie algebras that are isomorphic to $\G_{\nabla,\alpha^\prime}$. The list of these isomorphisms can be found in Table \ref{tableofcoho}.

After that, we proceed to classify symplectic forms on $\G_{\nabla,\alpha^\prime}$, that have a flat torsion-free connection $\nabla^\omega$ coming from $(\h,\nabla)$. This is accomplished by using one-to-one correspondences between Lagrangian extensions over flat Lie algebras and the group $H^2_{L,\rho}(\h,\h^\ast)$ (see, Proposition \ref{1-1 corress}). Let $\mu\in\mathcal{C}^1_L(\h,\h^\ast)$, such that, $\alpha-\partial_\rho\mu=\alpha^{\prime\prime}$ for some $\alpha^{\prime\prime} \in Z^2_{L,\rho}(\h,\h^\ast)$. It is easily verified that then the map 
\begin{equation}
\Phi:(\G_{\nabla,\alpha^{\prime\prime}},\omega)\longrightarrow (\G_{\nabla,\alpha^\prime},\omega),~~(x,\xi)\longmapsto (x,\xi+(\sigma-\mu)(x))
\end{equation}
 is the required isomorphism of Lagrangian extensions over $\h$. Table \ref{tableofcoho} shows the Lagrangian extension cohomology group $H^2_{L,\rho}(\h,\h^\ast)$, for all flat Lie algebras with right-identity elements. Therefore,  every symplectic form $\omega$ on $\G_{\nabla,\alpha^\prime}$ that has flat torsion-free connection $\nabla^\omega$ coming from $(\h,\nabla)$  is symplectomorphically equivalent to 
 \begin{eqnarray*}
 \omega_{[\alpha]_{L,\rho}}(x,y+\xi)=-\chi(x,y)\oplus\omega(x,\xi),~x\in\h, \xi\in\h^\ast,
 \end{eqnarray*}
 where $\chi=(\sigma-\mu)- {}^t\left(\sigma-\mu\right)\in\bigwedge^2\h^\ast$.

  For an illustration, we compute two different examples, the others cases are treated in the Appendix \ref{AppenA}. These examples have the advantage of using all the techniques needed in the general case. Let us apply the scheme above to $(\h_1,\nabla)$ (see, Table \ref{tableleftt}), with its basis $\mathcal{B}= \{e_1, e_2, e_3\}$, where the flat torsion-free connection $\nabla$ is given by
 \begin{equation*}
 \nabla_{e_1}e_3=e_1,~ \nabla_{e_2}e_3=e_2, \nabla_{e_3}e_j=e_j,~j=1,2,3.
 \end{equation*}
Let $(\G_{\nabla,\alpha},\omega)$ be the Lagrangian extension of the flat
Lie algebra $(\h_1, \nabla)$ with respect to $\alpha\in Z^2_{L,\rho}(\h_1,\h_1^\ast)$.   
Let $\mathcal{B}^\ast=\{e^1,e^2,e^3\}$ be the dual basis of $\mathcal{B}$ that generates $\h^\ast$. Let $\alpha\in C^2(\h_1,\h_1^\ast)$  given by
\begin{center}
$\alpha(e_i,e_j)=\sum_{k=1}^3 \alpha_{ijk}e^k,~1\leq i<j\leq 3,~~\alpha_{ijk}\in\R.$
\end{center}
It follows that $\alpha\in Z^2_{L,\rho}(\h_1,\h_1^\ast)$ if and only if $\alpha$ is as follows:
\begin{center}
$\alpha(e_1,e_2)=(\alpha_{135}-\alpha_{234})e^3,~\alpha(e_1,e_3)=
\sum^3_{k=1}\alpha_{13k}e^k,~\alpha(e_2,e_3)=\sum^3_{k=1}\alpha_{23k}e^k$.
\end{center}
It is straightforward to verify that, $H^2_\rho(\h_1,\h^\ast_1)=0$, in this case, $[\alpha]_\rho=0$, that is, $\alpha^\prime=0$.

Let $\sigma\in\mathcal{C}^1(\h_1,\h^\ast_1)$, and $x,y,z\in\h_1$. Then
\begin{align*}
\big(\partial_\rho\sigma\big)(x,y)&=\rho(x)\sigma(y)-\rho(y)\sigma(x)-\sigma([x,y])\\
&=-\sigma(y)\circ\nabla_x+\sigma(x)\circ\nabla_y-\sigma([x,y])
\end{align*} 
and, hence,
\begin{align}\label{partialforsigma}
\big(\partial_\rho\sigma\big)(x,y)(z)&=-\sigma(y)(\nabla_xz)+\sigma(x)(\nabla_yz)-\sigma([x,y])(z).
\end{align}
Let $\mu\in\mathcal{C}^1_L(\h_1,\h^\ast_1)$, i.e., $\mu(e_j)=\sum^3_{k=1}\mu_{jk}e^k$. Using (\ref{partialforsigma}) we compute
\begin{eqnarray*}
&&\big(\partial_\rho\mu\big)(e_1,e_3,e_1)=\mu_{11},~~
\big(\partial_\rho\mu\big)(e_1,e_3,e_2)=\mu_{12},~~
\big(\partial_\rho\mu\big)(e_1,e_3,e_3)=0,\\
&&\big(\partial_\rho\mu\big)(e_2,e_3,e_1)=\mu_{12},~~
\big(\partial_\rho\mu\big)(e_2,e_3,e_2)=\mu_{22},~~
\big(\partial_\rho\mu\big)(e_2,e_3,e_3)=0,\\
&&\big(\partial_\rho\mu\big)(e_1,e_2,e_j)=0,~j=1,2,3.
\end{eqnarray*}
This shows that 
\begin{center}
$B^2_{L,\rho}(\h_1,\h^\ast_1)=\left\{\mu\in\mathrm{Hom}(\bigwedge^2\h_1,\h^\ast_1)~\Big|~
    \begin{array}{l}
      \mu(e_1,e_2,e_j)=0,~~
       \mu(e_1,e_3,e_3)=0,\\
       \mu(e_2,e_3,e_3)=0.
    \end{array}
\right\}$.
\end{center}
In particular,
\begin{center}
$H^2_{L,\rho}(\h_1,\h^\ast_1)\cong\left\{\mu\in\mathrm{Hom}(\bigwedge^2\h_1,\h^\ast_1)~\Big|~
    \begin{array}{l}
      \mu(e_1,e_2,e_3)=\tau_1,~~
       \mu(e_1,e_3,e_3)=\tau_2,\\
       \mu(e_2,e_3,e_1)=-\tau_1,~
       \mu(e_2,e_3,e_3)=\tau_3.
    \end{array}
\right\}$.
\end{center}
Due to the fact that $H^2_\rho(\h_1,\h^\ast_1)=0$, the form $\omega$ verifies the conditions of Proposition \ref{exactLagrang}, therefore it is an exact form.  By identifying the basis $\{e_4,e_5,e_6\}$ with the dual basis
of $\{e_1,e_2,e_3\}$. As $\chi(x,y)=\tau_1 e^{12}+\tau_2 e^{13}+\tau_3e^{23}$, any symplectic form $\omega$ on $\G_{\nabla,0}$ whose flat torsion-free connection $\nabla^\omega$ coming from $(\h_1,\nabla)$ is symplectomorphically isomorphic to 
\begin{equation}\label{allofthemsympl}
\omega_{[\alpha]_{L,\rho}}=\tau_1 e^{12}+\tau_2 e^{13}+\tau_3e^{23}+e^{14}+e^{25}+e^{36}.
\end{equation}

Study symplectic forms on $\G_{\nabla,0}$. The Lie brackets over the basis $\mathcal{B}_0=\{e_1,e_2,e_3,e_4,e_5,e_6\}$ are given
\begin{eqnarray*}
&&[e_1,e_4]=-e_6,~[e_2,e_5]=-e_6,~[e_3,e_4]=-e_4,~[e_3,e_5]=-e_5,~[e_3,e_6]=-e_6,
\end{eqnarray*}
Taking the dual basis
$\{e^1,e^2,e^3,e^4,e^5,e^6\}$, the Maurer-Cartan equations of the algebra are
\[\mathrm{d}e^1=\mathrm{d}e^2=\mathrm{d}e^3=0,~\mathrm{d}e^4=-e^{34},~\mathrm{d}e^5-e^{35},~\mathrm{d}e^6=-e^{14}-e^{25}-e^{36}.\]
 Now define an element $\omega\in\bigwedge^2(\G_{\nabla,0})^\ast$
in general position by
\begin{equation*}
\omega=\omega_{12}e^{12}+\omega_{13}e^{13}+\omega_{23}e^{23}+\omega_{34}e^{34}+\omega_{35}e^{35}+\omega_{36}\mathrm{d}e^6,~~\omega_{36}\neq0.
\end{equation*}
The following is a list of possible Lagrangian ideals
\[\mathfrak{j}_1=\langle e_4,e_5,e_6\rangle,~\mathfrak{j}_2=\langle e_1,e_2,e_6\rangle,~\mathfrak{j}_3=\langle e_2,e_4,e_6\rangle,~\mathfrak{j}_4=\langle e_1,e_5,e_6\rangle.\]

The flat torsion-free   connection $\nabla^\omega$ that meets the relation (\ref{torsion-free}) in the quotient Lie algebra $\h=\G_{\nabla,0}/\mathfrak{j}_1$  coinsides with the flat torsion-free connection of $\nabla$ associated with $(\h_1,\nabla)$. In this case, we obtain $(\h=\G_{\nabla,0}/\mathfrak{j}_1,\nabla^\omega)\cong(\h_1,\nabla)$ ($\psi=\mathrm{I}_3$). Therefore, any symplectic form on $\G_{\nabla,0}$ is symplectomorphically equivalent to  $\omega_{[\alpha]_{L,\rho}}$ given by $(\ref{allofthemsympl})$.

The next example will be different from the previous one. Lie algebra $\G_{6,11,2}^{\lambda\neq-1}=\h_{11,2}^{\lambda\neq-1}\oplus(\h_{11,2}^{\lambda\neq-1})^\ast$ with its basis $\mathcal{B} = \{e_1,e_2, e_3, e_4,e_5,e_6\}$ where the non-zero Lie brackets are given in Proposition \ref{Prclassifi}. The Lagrangian extension cohomology group and the relative cohomology of the flat Lie algebra $\h_{11,2}^{\lambda\neq-1}$ are given in Table \ref{tableofcoho}, and its flat torsion-free connection $\nabla$ is given in Proposition \ref{Pr 3flat} (Table \ref{tableleftt}). Using the same procedure as the previous example, any symplectic structure on $\G_{6,11,2}^{\lambda\neq-1}$ which has a flat torsion-free connection $\nabla^\omega$ comes from that of $(\h_{11,2}^{\lambda\neq-1},\nabla)$ is is symplectomorphically equivalent to
\begin{equation}\label{sympexem2}
\omega_{[\alpha]_{L,\rho}}=-\lambda\tau e^{13}+e^{14}+e^{25}+e^{36},~\tau\in\R,~\lambda\neq0.
\end{equation}
Now define an element $\omega\in\bigwedge^2(\G_{6,11,2}^{\lambda\neq-1})^\ast$ in general position, in this case, we have two families based on the value of $\lambda$.

$\bullet$ If $\lambda\neq\pm1$. Then, 
\[\omega=\omega_{13}e^{13}+\omega_{23}e^{23}+\omega_{15}(e^{15}+\tfrac{1+\lambda}{\lambda}e^{35})+\omega_{34}(\lambda e^{14}+\lambda e^{25}+e^{34})+\lambda\omega_{36}\mathrm{d}e^6,\] 
with, $\Omega=\omega_{36}(\lambda\omega_{34}+\omega_{36})\neq0$.
Note that, $\mathfrak{j}=\langle e_4,e_5,e_6\rangle$ is a Lagrangian ideal of $(\G_{6,11,2}^{\lambda\neq-1},\omega)$, and the quotient flat Lie algebra  $(\h=\G_{6,11,2}^{\lambda\neq\pm1}/\mathfrak{j},\psi^\ast\nabla^\omega)$ is isomorphic to $(\h_{11,2}^{\lambda\neq\pm1},\nabla)$, where $\psi$ is given by
\begin{eqnarray*}
\psi :\h\longrightarrow\h : \overline{ e}_1\mapsto \overline{e}_1,~\overline{e}_2\mapsto \overline{e}_2,~\overline{e}_3\mapsto\tfrac{\omega_{15}}{\Omega}\overline{e}_2+\overline{e}_3.
\end{eqnarray*}

As an immediate reason,  any symplectic structure on $\G_{6,11,2}^{\lambda\neq\pm1}$ is symplectomorphically equivalent to the form given by $(\ref{sympexem2})$.

$\bullet$ If $\lambda=1$. In this case, symplectic structures on $\G_{6,11,2}^{\lambda=1}$ have the following form:
\[\omega=\omega_{13}e^{13}+\omega_{23}e^{23}+\omega_{34}(e^{14}+ e^{25}+e^{34})+\omega_{15}(e^{15}+2 e^{35})+\omega_{26}(-e^{24}+e^{26})+\omega_{36}\mathrm{d}e^6,\] 
with $\Omega=(\omega_{3 4} + \omega_{3 6})(\omega_{1 5}\omega_{2 6} - \omega_{3 4}\omega_{3 6} - \omega_{3 6}^2)\neq0$.

Let us perform the scheme above for the  Lagrangian ideal $\mathfrak{j}=\langle e_4,e_5,e_6\rangle$ of $(\G_{6,11,2}^{\lambda=1},\omega)$.  It is straightforward to verify that the associated flat torsion-free connection $~\overline{\nabla}=\psi^\ast\nabla^\omega$ on the quotient Lie algebra $\B=\G_{6,11,2}^{\lambda=1}/\mathfrak{j}$ is 
\begin{eqnarray*}
&&\overline{\nabla}_{\overline{e}_1}\overline{e}_1=\overline{e}_1,~\overline{\nabla}_{\overline{e}_1}\overline{e}_2=-\epsilon\overline{e}_1+\overline{e}_2+\epsilon\overline{e}_3,~\overline{\nabla}_{\overline{e}_1}\overline{e}_3=\overline{e}_1,\\
&&\overline{\nabla}_{\overline{e}_2}\overline{e}_1=-\epsilon\overline{e}_1+\overline{e}_2+\epsilon\overline{e}_3,~\overline{\nabla}_{\overline{e}_2}\overline{e}_3=-\epsilon\overline{e}_1+\overline{e}_2+\epsilon\overline{e}_3,\\
&&\overline{\nabla}_{\overline{e}_3}\overline{e}_1=\overline{e}_1,~\overline{\nabla}_{\overline{e}_3}\overline{e}_2=-\epsilon\overline{e}_1+2\overline{e}_2+\epsilon\overline{e}_3,~\overline{\nabla}_{\overline{e}_3}\overline{e}_3=\overline{e}_3,
\end{eqnarray*}
where, $\epsilon=0~\text{or}~1$, and $\psi$ is given in Table $\ref{flatfromsymplectic}$ (see, case $\B_1$). Note that, $(\B,\overline{\nabla})\cong(\h_{11,2}^{\lambda=1},\nabla)$ when $\epsilon=0$, in this case, we have previously classified the symplectic structures on $\G_{6,11,2}^{\lambda=1}$ whose flat torsion-free connection comes from $(\h_{11,2}^{\lambda=1},\nabla)$. We will examine the situation when, $\epsilon=1$. Furthermore, and in this case, $(\B_{\epsilon=1},\overline{\nabla})$ has no right-identity element. We have, $H^2_\rho(\B_{\epsilon=1},\B_{\epsilon=1}^\ast)=0$ (see for instant,  Table \ref{cohoforflatsym}, case $\B_{1}$), and  the Lagrangian extension cohomology group
\begin{center}
$H^2_{L,\rho}\big(\B_{\epsilon=1},\B_{\epsilon=1}^\ast\big)=\Big\{\mu\in\mathrm{Hom}\big(\bigwedge^2\B_{\epsilon=1},\B_{\epsilon=1}^\ast\big)~\big|~\mu(e_1,e_3,e_1)=\tau^\prime,~
\mu(e_1,e_3,e_3)=\tau^\prime\Big\}$.
\end{center}
In this regard, it is straightforward to verify  that, $\chi(x,y)=\tau^\prime e^{13}$. As a consequence, any symplectic structure on $\G_{\overline{\nabla},\alpha^\prime}=\B_{\epsilon=1}\oplus\B_{\epsilon=1}^\ast$ whose flat torsion-free connection isomorphic to $\overline{\nabla}$ of $(\B_{\epsilon=1},\overline{\nabla})$ is  symplectomorphically equivalent to
\[\omega_{[\alpha^\prime]_{L,\rho}}=\tau^\prime e^{13}+e^{14}+e^{25}+e^{36}.\]
Then, by a symplectomorphism (see for instant Table \ref{isofrom new to new}), we send the symplectic structure $\omega_{[\alpha^\prime]_{L,\rho}}$ to the starting Lie algebra $\G_{6,11,2}^{\lambda=1}$, and we find all the non-symplectomorphically equivalent symplectic forms on $\G_{6,11,2}^{\lambda=1}$. This last algebra should be sent to the list given in \cite{Tur3} or \cite{Muba1}. The list of such isomorphisms are given in Tables  \ref{isoTur} and \ref{isoMuba}. In this way, the proof is complete.

\end{proof}
\begin{remark}
\begin{enumerate}
\item  $\G_{6,22,2}^{\beta=\frac{1}{2\delta_1}}$ with $a\neq 0$ and $\G_{6,18,2}^{\eta=\frac{1}{2}}$ $($with $b^2-ac<0$, $abc\neq0)$ represent new Lie algebras that don't appear in \textsc{\cite{Muba1}} and \textsc{\cite{Shaba}}, and their nilradical is isomorphic to $ [e_2, e_4] = e_1, [e_3, e_5] = e_1$, so they belong to Table $7$ of \textsc{\cite{Muba1}}.

\item The eigenvalues of the adjoint representation of $\G_{6,22,2}^{\beta=\frac{1}{2\delta_1}}$ are  $\{0,-\frac{1}{\delta_1},\frac{-1+2i\delta_1}{2\delta_1},-\frac{1+2i\delta_1}{2\delta_1}\}$, and this algebra is closer to Mubarakzyanov algebras $\G_{6,88}$ and $\G_{6,92}$. The adjoint representation of the algebra $\G_{6,88}$ has six distinct eigenvalues, indicating that it is never isomorphic to $\G_{6,22,2}^{\beta=\frac{1}{2\delta_1}}$. For $a=0$, we have already proven the following isomorphism $\G_{6,22,2}^{\beta=\frac{1}{2\delta_1}}\cong \G_{6,92}$ $($see Table $\ref{isoMuba})$. If $a\neq0$,  we are comparing algebra $\G_{6,22,2}^{\beta=\frac{1}{2\delta_1}}$, which has a unique symplectic structure and admits one Lagrangian ideal, with algebra $\G_{6,92}$, wich also has a unique symplectic structure that includes two Lagrangian ideals. Therefore, $\G_{6,22,2}^{\beta=\frac{1}{2\delta_1}}$ and $\G_{6,92}$ are not isomorphic.

\item The algebra $\G_{6,22,2}^{\beta=\frac{1}{2\delta_1}}$ with $a\neq0$ $(\text{resp.}~\G_{6,18,2}^{\eta=\frac{1}{2}}$ $($with $b^2-ac<0$, $abc\neq0))$ is represented by $\G_{6,92}^{\prime\prime}$ $(\text{resp.}~\G_{6,92}^{\prime\prime\prime})$ in Theorem $\ref{Principtheo}$.
\end{enumerate}

\end{remark}
\newpage
\section{Appendix A}\label{AppenA}
\subsection{Cohomological properties of flat Lie algebras}

\begin{Le}
The Lagrangian extension cohomology group and the relative cohomology of any three-dimensional real flat  Lie algebra  $(\h,\nabla)$ with a right-identity element are presented in the following table$:$
\end{Le}
{\renewcommand*{\arraystretch}{1.23}
\captionof{table}{Cohomological properties of flat Lie algebras with right-identity element.}\label{tableofcoho}
\small\begin{longtable}{lll}
			\hline
		Flat Lie algebra& $H^2_{L,\rho}(\h,\h^\ast)$&$H^2_{\rho}(\h,\h^\ast)$\\\hline
		$\h_1$&$\left[\begin{array}{l}\alpha_{12}\mapsto \tau_1 e^\ast_3\\ \alpha_{13}\mapsto \tau_2 e^\ast_3\\\alpha_{23}\mapsto -\tau_1 e_1^\ast+\tau_3 e^\ast_3\end{array}\right] $&$0$
		\\
		$\h_2^{b}$&$\left[\begin{array}{l}
		\alpha_{12}\mapsto -\tau_1be_1^\ast+b(\tau_1+\tau_2)e_2^\ast+\tau_3e_3^\ast\\
	\alpha_{13}\mapsto\tau_3e_2^\ast-\tau_1e_3^\ast\\
	\alpha_{23}\mapsto-\tau_2e_3^\ast	
		\end{array}\right]$&$0$
		\\
	$\h_3$	&$\left[\begin{array}{l} \alpha_{12}\mapsto \tau_1 e_3^\ast\\
\alpha_{13}\mapsto\tau_2 e_3^\ast\\ \alpha_{23}\mapsto-\tau_1 e_1^\ast+\tau_3e_3^\ast\end{array}\right] 
$&$0$
		\\
	$\h_{4}^{a}$	& $\left[\alpha_{13}\mapsto\tau e_1^\ast\right]$&$0$\\
	$\h_{5}$&$\left[\alpha_{13}\mapsto \tau e_1^\ast\right]$&$0$\\
	$\h_{6}$&$\left[\alpha_{13}\mapsto\tau e_1^\ast\right]$&$0$\\
	$\h_{7}$&$\left[\alpha_{13}\mapsto\tau e_1^\ast\right]$&$0$\\
	$\h_{8}$&$\left[\alpha_{13}\mapsto\tau e_1^\ast\right]$&$0$\\
	$\h_{9}^{a}$&$\left[\alpha_{13}\mapsto\tau e_1^\ast\right]$&$0$\\
	$\h_{10,1}^{a,b,\lambda=-1}$&$\left[\alpha_{13}\mapsto\tau e_3^\ast\right]$&$0$\\
	$\h_{10,2}^{a,b,\lambda\neq-1}$&$\left[\alpha_{13}\mapsto\tau e_3^\ast\right]$&$0$\\
	$\h_{11,1}^{\lambda=-1}$&$\left[\begin{array}{l} \alpha_{12}\mapsto a e_1^\ast\\
\alpha_{13}\mapsto\tau e_3^\ast\\\alpha_{23}\mapsto b e_3^\ast\end{array}\right]$&$\left[\begin{array}{l}\alpha_{12}\mapsto ae_1^\ast\\
\alpha_{23}\mapsto be_3^\ast\end{array}\right]$\\
	$\h_{11,2}^{\lambda\neq-1}$&$\left[\alpha_{13}\mapsto\tau e_3^\ast\right]$&$0$\\
	$\h_{12}^{\lambda}$&$\left[\alpha_{13}\mapsto\tau e_3^\ast\right]$&$0$\\
	$\h_{13,1}^{\lambda\neq-\frac{1}{2}}$&$\left[\alpha_{13}\mapsto\tau e_3^\ast\right]$&$0$\\
	$\h_{13,2}^{\lambda=-\frac{1}{2}}$&$\left[\begin{array}{l} \alpha_{12}\mapsto a e_2^\ast\\
	\alpha_{13}\mapsto\tau e_3^\ast\\ \alpha_{23}\mapsto b e_2^\ast\end{array}\right]$&$\left[\begin{array}{l}\alpha_{12}\mapsto a e_2^\ast\\ \alpha_{23}\mapsto b e_2^\ast\end{array}\right]$\\
	$\h_{14,1}^{\lambda\neq\{-\frac{1}{2},-1\}}$&$\left[\alpha_{13}\mapsto\tau e_3^\ast\right]$&$0$\\
	$\h_{14,2}^{\lambda=-1}$&$\left[\begin{array}{l}\alpha_{12}\mapsto a e_1^\ast\\\alpha_{13}\mapsto b e_2^\ast+ \tau e_3^\ast\\\alpha_{23}\mapsto b e_1^\ast\end{array}\right]$&$\left[\begin{array}{l}\alpha_{12}\mapsto a e_1^\ast\\\alpha_{13}\mapsto b e_2^\ast\\\alpha_{23}\mapsto b e_1^\ast\end{array}\right]$
	\\
	$\h_{14,3}^{\lambda=-\frac{1}{2}}$&$\left[\begin{array}{l}\alpha_{12}\mapsto a e_2^\ast\\\alpha_{13}\mapsto  \tau e_3^\ast\\\alpha_{23}\mapsto b e_2^\ast\end{array}\right]$&$\left[\begin{array}{l}\alpha_{12}\mapsto a e_2^\ast\\\alpha_{23}\mapsto b e_2^\ast\end{array}\right]$\\
	$\h_{15}$&$\left[\begin{array}{l}\alpha_{12}\mapsto \tau_1 e_3^\ast\\\alpha_{13}\mapsto  \tau_1 e_2^\ast\\\alpha_{23}\mapsto \tau_2 e_2^\ast\end{array}\right]$&$0$
	\\
	$\h_{16}$&$\left[\begin{array}{l}\alpha_{12}\mapsto \tau_1 e_3^\ast\\\alpha_{13}\mapsto  \tau_1 e_2^\ast\\\alpha_{23}\mapsto \tau_2 e_2^\ast\end{array}\right]$&$0$\\
	$\h_{17,1}^{\mu\neq\{\frac{1}{2},\frac{1}{3}\}}$&$0$&$0$\\
	$\h_{17,2}^{\mu=\frac{1}{2}}$&$\left[\alpha_{13}\mapsto  a e_1^\ast\right]$&$\left[\alpha_{13}\mapsto  a e_1^\ast\right]$\\
	$\h_{17,3}^{\mu=\frac{1}{3}}$&$\left[\alpha_{12}\mapsto  b e_2^\ast\right]$&$\left[\alpha_{12}\mapsto  b e_2^\ast\right]$\\
	$\h_{18,1}^{\eta\neq\{\frac{1}{2},\frac{1}{3}\}}$&$0$&$0$\\
	$\h_{18,2}^{\eta=\frac{1}{2}}$&$\left[\begin{array}{l}\alpha_{13}\mapsto  a e_1^\ast+be_2^\ast\\\alpha_{23}\mapsto  b e_1^\ast+ce_2^\ast\end{array}\right]$&$\left[\begin{array}{l}\alpha_{13}\mapsto  a e_1^\ast+be_2^\ast\\\alpha_{23}\mapsto  b e_1^\ast+ce_2^\ast\end{array}\right]$\\
$\h_{18,3}^{\eta=\frac{1}{3}}$&$\left[\alpha_{12}\mapsto  a e_1^\ast+be_2^\ast\right]$&$\left[\alpha_{12}\mapsto  a e_1^\ast+be_2^\ast\right]$\\
$\h_{19,1,\alpha,\gamma}^{\alpha\neq\{\frac{1}{2\gamma},-1+\frac{1}{\gamma},\gamma\neq\frac{1}{2}\}}$&$0$&$0$\\
$\h_{19,2}^{\alpha=\frac{1}{2\gamma},\gamma\neq\frac{1}{2}}$&$[\alpha_{23}\mapsto  a e_2^\ast]$&$[\alpha_{23}\mapsto  a e_2^\ast]$\\
$\h_{19,3}^{\alpha=-1+\frac{1}{\gamma},\gamma\neq\{\frac{1}{2},1\}}$&$\left[\begin{array}{l}\alpha_{13}\mapsto  b e_2^\ast\\\alpha_{23}\mapsto  b e_1^\ast \end{array}\right]$&$\left[\begin{array}{l}\alpha_{13}\mapsto  b e_2^\ast\\\alpha_{23}\mapsto  b e_1^\ast \end{array}\right]$\\
$\G^{\alpha,\gamma=\frac{1}{2}}_{19,4}$&$[\alpha_{13}\mapsto  c e_1^\ast ]$&$[\alpha_{13}\mapsto  c e_1^\ast ]$\\
$\h^{\alpha=-\frac{1}{2}+\frac{1}{2\gamma},\gamma\neq\{1,\frac{1}{2}\}}_{19,5}$&$[\alpha_{12}\mapsto  d e_2^\ast ]$&$[\alpha_{12}\mapsto  d e_2^\ast ]$\\
$\h_{19,6}^{\alpha=\gamma=\frac{1}{2}}$&$\left[\begin{array}{l}\alpha_{12}\mapsto  a e_2^\ast\\\alpha_{13}\mapsto  b e_1^\ast  \end{array}\right]$&$\left[\begin{array}{l}\alpha_{12}\mapsto  a e_2^\ast\\\alpha_{13}\mapsto  b e_1^\ast  \end{array}\right]$\\
$\h_{19,7}^{\alpha=-2+\frac{1}{\gamma},\gamma\neq\{1,\frac{1}{2}\}}$&$[\alpha_{12}\mapsto  a e_1^\ast  ]$&$[\alpha_{12}\mapsto  a e_1^\ast  ]$\\
$\h_{20,1}^{\gamma\neq\{1,\frac{2}{3},\frac{2}{5}\}}$&$0$&$0$\\
$\h_{20,2}^{\gamma=\frac{2}{3}}$&$\left[\begin{array}{l}\alpha_{13}\mapsto a e_2^\ast\\\alpha_{23}\mapsto a e_1^\ast\end{array}\right]$&$\left[\begin{array}{l}\alpha_{13}\mapsto a e_2^\ast\\\alpha_{23}\mapsto a e_1^\ast\end{array}\right]$\\
$\h_{20,3}^{\gamma=1}$&$[\alpha_{23}\mapsto b e_2^\ast]$&$[\alpha_{23}\mapsto b e_2^\ast]$\\
$\h_{21,1}^{\nu\neq\{\frac{1}{2},1\},\nu>0}$&$\left[\begin{array}{l}\alpha_{12}\mapsto \tau e_3^\ast\\\alpha_{23}\mapsto -\tau e_1^\ast\end{array}\right]$&$0$\\
$\h_{20,4}^{\gamma=\frac{2}{5}}$&$[\alpha_{12}\mapsto c e_1^\ast]$&$[\alpha_{12}\mapsto c e_1^\ast]$
\\
$\h_{21,2}^{\nu=\frac{1}{2}}$&$\left[\begin{array}{l}\alpha_{12}\mapsto \tau e_3^\ast\\\alpha_{13}\mapsto a e_1^\ast\\\alpha_{23}\mapsto -\tau e_1^\ast\end{array}\right]$&$[\alpha_{13}\mapsto a e_1^\ast]$\\
$\h_{21,3}^{\nu=1}$&$\left[\begin{array}{l}\alpha_{12}\mapsto b e_1^\ast+\tau e_3^\ast\\\alpha_{23}\mapsto -\tau e_1^\ast\end{array}\right]$&$[\alpha_{12}\mapsto b e_1^\ast]$\\
$\h_{22,1}^{\beta\neq\frac{1}{2\delta_1}}$&$0$&$0$\\
$\h_{22,2}^{\beta=\frac{1}{2\delta_1}}$&$[\alpha_{23}\mapsto ae_2^\ast]$&$[\alpha_{23}\mapsto ae_2^\ast]$
\\
$\h_{23}^{\beta=0}$&$\left[\begin{array}{l}\alpha_{12}\mapsto \tau e_3^\ast\\\alpha_{13}\mapsto \tau e_2^\ast\end{array}\right]$&$0$
		\\\hline
		\end{longtable}}

\newpage
\subsection{Symplectic structures on six-dimensional Frobeniusian Lie algebras with Lagrangian ideal}
\begin{center}
\captionof{table}{Symplectic structures on six-dimensional Frobeniusian Lie algebras with Lagrangian ideal}\label{cocycleofall}
{\renewcommand*{\arraystretch}{1.2}
	\small\begin{tabular}{l l l}\hline
		Algebra&Symplectic form&Conditions on $\omega_{ij}$\\\hline
		\multirow{1}{*}{$\G_{6,1}$}&$\omega=\omega_{12}e^{12}+\omega_{13}e^{13}+\omega_{23}e^{23}+\omega_{34}e^{34}+\omega_{35}e^{35}+\omega_{36}\mathrm{d}e^6$&$\omega_{36}\neq0$\\
		\multirow{1}{*}{$\G^{b}_{6,2}$}&$\omega=\omega_{12}e^{12}+\omega_{13}e^{13}+\omega_{23}e^{23}+\omega_{34}(e^{14}+e^{15}+be^{26}+e^{34})$\\
		&$~~~~~~+\omega_{35}(be^{16}+be^{24}-be^{26}+e^{35})+\omega_{36}\mathrm{d}e^6$&$\Omega_1\neq0$\\
		\multirow{1}{*}{$\G_{6,3}$}&$\omega=\omega_{12}e^{12}+\omega_{13}e^{13}+\omega_{23}e^{23}+\omega_{34}e^{34}+\omega_{35}(e^{24}+e^{35})+\omega_{36}\mathrm{d}e^6$&$\omega_{36}\neq0$\\
		\multirow{1}{*}{$\G^{a}_{6,4}$}&$\omega=\omega_{13}e^{13}+\omega_{23}e^{23}+\omega_{16}(-ae^{14}-e^{15}+e^{16}-ae^{25})$\\
		&~~~~~~~~$+\omega_{35}(e^{15}+e^{35})+\omega_{36}\mathrm{d}e^4$&$\omega_{36}(a\omega_{16}-\omega_{36})\neq0$\\
		$\G_{6,5}$&$\omega=\omega_{13}e^{13}+\omega_{23}e^{23}+\omega_{35}(e^{15}+e^{35})+\omega_{16}(e^{16}+e^{36})+\omega_{25}\mathrm{d}e^4$&$\omega_{25}(\omega_{25}+\omega_{16})\neq0$\\
		$\G_{6,6}$&$\omega=\omega_{13}e^{13}+\omega_{23}e^{23}+\omega_{16}e^{16}+\omega_{35}(e^{15}+e^{35})+\omega_{36}\mathrm{d}e^4$&$\omega_{36}\neq0$\\
		$\G_{6,7}$&$\omega=\omega_{13}e^{13}+\omega_{23}e^{23}+\omega_{16}(-2\,e^{14}+e^{16}-2\,e^{25})$&\\
		&$~~~~~~~~+\omega_{35}(e^{15}+e^{26}+e^{35})+\omega_{36}\mathrm{d}e^4$&$\Omega^+\neq0$\\
		$\G_{6,8}$&$\omega=\omega_{13}e^{13}+\omega_{23}e^{23}+\omega_{16}(-2\,e^{14}+e^{16}-2\,e^{25})$&\\
		&$~~~~~~~~+\omega_{35}(e^{15}-e^{26}+e^{35})+\omega_{36}\mathrm{d}e^4$&$\Omega^-\neq0$\\
		$\G^{a}_{6,9}$&$\omega=\omega_{13}e^{13}+\omega_{23}e^{23}+\omega_{16}(-e^{14}-(a+1)e^{15}+e^{16}-e^{25})$&\\
		&$~~~~~~~~+\omega_{35}(e^{15}+e^{35})+\omega_{36}\mathrm{d}e^4$&$\omega_{36}(\omega_{16}-\omega_{36})\neq0$\\
	 $\G_{6,10,1}^{a,b,\lambda=-1}$	&$\omega=\omega_{13}e^{13}+\omega_{15}e^{15}+\omega_{23}e^{23}+\omega_{34}(-\frac{a}{b}e^{14}+\frac{b-a}{b}e^{16}-e^{25}+e^{34})$&\\
	 &$~~~~~~~~+\omega_{36}\mathrm{d}e^6$&$\Omega_2\neq0$\\
	 $\G_{6,10,2}^{a,b,\lambda\neq-1}$&$\omega=\omega_{12}e^{12}+\omega_{13}e^{13}+\omega_{34}(\frac{a\lambda}{b}e^{14}+\frac{\lambda^2(b-a)}{b}e^{16}+\lambda e^{25}+e^{34})$&\\
	 &$~~~~~~~~+\omega_{35}(\frac{\lambda}{1+\lambda}e^{15}+e^{35})+\lambda\omega_{36}\mathrm{d}e^6$&$\Omega_3\neq0$\\
	 $\G_{6,11,1}^{a,b,\lambda=-1}$&$\omega=\omega_{13}e^{13}+\omega_{15}e^{15}+\omega_{23}e^{23}+\omega_{34}(ae^{12}-e^{14}-e^{25}+e^{34})$&\\
	&$~~~~~~~~+\omega_{36}\mathrm{d}e^6$& $\omega_{36}(\omega_{36}-\omega_{34})\neq0$\\
	$\G^{\lambda\neq\pm1}_{6,11,2}$&$\omega=\omega_{13}e^{13}+\omega_{23}e^{23}+\omega_{34}(\lambda e^{14}+\lambda e^{25}+e^{34})+\omega_{35}(\frac{\lambda}{\lambda+1}e^{15}+e^{35})$&\\
	&$~~~~~~~~+\lambda\omega_{36}\mathrm{d}e^6$&$\omega_{36}(\lambda\omega_{34}+\omega_{36})\neq0$\\
	$\G^{\lambda=1}_{6,11,2}$&$\omega=\omega_{13}e^{13}+\omega_{23}e^{23}+\omega_{34}(e^{14}+ e^{25}+e^{34})+\omega_{15}(e^{15}+2\,e^{35})$&\\
	&$~~~~~~~~~+\omega_{26}(-e^{24}+e^{26})+\omega_{36}\mathrm{d}e^6$&$\Omega_4\neq0$\\
	$\G_{6,12}^{\lambda\neq-2}$&$\omega=\omega_{13}e^{13}+\omega_{23}e^{23}+\omega_{15}(e^{15}+\frac{1+\lambda}{\lambda}e^{35})+\omega_{34}(\lambda^2 e^{16}+\lambda e^{25}+e^{34})$&\\
	&$~~~~~~~~~+\lambda\omega_{36}\mathrm{d}e^6$&$\lambda\omega_{34}\pm\omega_{36}\neq0$
	\\
	$\G_{6,12}^{\lambda=-2}$&$\omega=\omega_{13}e^{13}+\omega_{23}e^{23}+\omega_{34}(e^{16}-2\,e^{25}+e^{34})+\omega_{35}(e^{15}+e^{35})$&\\
	&$~~~~~~~~~~+\omega_{45}(e^{45}-2\,e^{56})+2\,\omega_{36}\mathrm{d}e^6$&$\Omega_5\neq0$\\
	$\G_{6,13,1}^{\lambda}$&$\omega=\omega_{13}e^{13}+\omega_{23}e^{23}+\omega_{35}e^{35}+\omega_{34}(\lambda\,e^{14}+e^{34})+\lambda\,\omega_{36}\mathrm{d}e^6$&$\omega_{36}(\lambda\,\omega_{34}+\omega_{36})\neq0$\\
	$^{\lambda\neq\{\pm 1,-2,-\frac{1}{2}\}}$&&\\
	$\G_{6,13,1}^{\lambda=-1}$&$\omega=\omega_{13}e^{13}+\omega_{15}e^{14}+\omega_{23}e^{23}+\omega_{35}e^{35}+\omega_{34}(-e^{14}+e^{34})+\omega_{36}\mathrm{d}e^6$&$\omega_{36}(\omega_{36}-\omega_{34})\neq0$\\
	$\G_{6,13,1}^{\lambda=1}$&$\omega=\omega_{13}e^{13}+\omega_{23}e^{23}+\omega_{35}e^{35}+\omega_{34}(e^{14}+e^{34})+\omega_{26}(-e^{24}+e^{26})$&\\
	&$~~~~~~~~~~+\omega_{36}\mathrm{d}e^6$&$\Omega_6\neq0$\\
	$\G_{6,13,1}^{\lambda=-2}$&$\omega=\omega_{13}e^{13}+\omega_{23}e^{23}+\omega_{35}e^{35}+\omega_{34}(-2\,e^{14}+e^{34})+\omega_{45}(e^{45}-2\,e^{56})$&\\
	&$~~~~~~~~~~+\omega_{36}\mathrm{d}e^6$&$\Omega_7\neq0$\\
	$\G_{6,13,2}^{\lambda=-\frac{1}{2}}$&$\omega=\omega_{13}e^{13}+\omega_{23}e^{23}+\omega_{34}(-\tfrac{1}{2}\,e^{14}+e^{34})+\omega_{35}(a\,e^{12}+e^{35})+\frac{1}{2}\omega_{36}\mathrm{d}e^6$&$\omega_{36}(\omega_{34}-2\,\omega_{36})\neq0$
	
	\end{tabular}}
	\end{center}
	\begin{center}
	{\renewcommand*{\arraystretch}{1.3}
	\small\begin{tabular}{l l l}
	$\G_{6,14,1}^{\lambda}$&$\omega=\omega_{13}e^{13}+\omega_{23}e^{23}+\omega_{34}e^{34}+\omega_{35}e^{35}+\lambda\,\omega_{36}\mathrm{d}e^6$,~$\lambda\neq\{\pm1,-2,-\tfrac{1}{2}\}$&$\omega_{36}\neq0$\\
	$\G_{6,14,1}^{\lambda=-2}$&$\omega=\omega=\omega_{13}e^{13}+\omega_{23}e^{23}+\omega_{34}e^{34}+\omega_{35}e^{35}+\omega_{45}e^{45}+\omega_{36}\mathrm{d}e^6$&$\omega_{36}\neq0$\\
	$\G_{6,14,1}^{\lambda=1}$&$\omega=\omega_{13}e^{13}+\omega_{23}e^{23}+\omega_{24}e^{24}+\omega_{34}e^{34}+\omega_{35}e^{35}+\omega_{36}\mathrm{d}e^6$&$\omega_{36}\neq0$\\
	$\G_{6,14,2}^{\lambda=-1}$&$\omega=\omega_{13}e^{13}+\omega_{15}e^{15}+\omega_{23}e^{23}+\omega_{35}e^{35}+\omega_{34}(a\,e^{12}+e^{34})+\omega_{36}\mathrm{d}e^6$&$\omega_{36}\neq0$\\
	$\G_{6,14,3}^{\lambda=-\frac{1}{2}}$&$\omega=\omega_{13}e^{13}+\omega_{23}e^{23}+\omega_{34}e^{34}+\omega_{35}(a\,e^{12}+e^{35})+\frac{1}{2}\omega_{36}\mathrm{d}e^6$&$\omega_{36}\neq0$\\
	$\G_{6,15}$&$\omega=\omega_{12}e^{12}+\omega_{13}e^{13}+\omega_{23}e^{23}+\omega_{26}e^{26}+\omega_{24}(e^{16}+e^{24})-\omega_{36}\mathrm{d}e^5$&$\omega_{36}\neq0$\\
	$\G_{6,16}$&$\omega=\omega_{12}e^{12}+\omega_{13}e^{13}+\omega_{23}e^{23}+\omega_{26}e^{26}+\omega_{24}e^{24}-\omega_{36}\mathrm{d}e^5$&$\omega_{36}\neq0$\\
	$\G_{6,17,1}^{ \mu\neq\{1,\frac{1}{2},\frac{1}{3}\}}$&$\omega=\omega_{13}e^{13}+\omega_{23}e^{23}+\omega_{34}e^{34}+\omega_{35}e^{35}-\mu\omega_{36}\mathrm{d}e^6$&$\omega_{36}\neq0$\\
	$\G_{6,17,1}^{ \mu=1}$&$\omega=\omega_{13}e^{13}+\omega_{23}e^{23}+\omega_{34}e^{34}+\omega_{35}e^{35}+\omega_{45}e^{45}-\omega_{36}\mathrm{d}e^6$&$\omega_{36}\neq0$\\
	$\G_{17,2}^{\mu=\frac{1}{2}}$&$\omega=\omega_{13}e^{13}+\omega_{23}e^{23}+\omega_{34}e^{34}+\omega_{35}e^{35}e^{45}-\frac{1}{2}\omega_{36}\mathrm{d}e^6$&$\omega_{36}\neq0$
	\\
	$\G_{17,3}^{\mu=\frac{1}{3}}$&$\omega=\omega_{13}e^{13}+\omega_{23}e^{23}+\omega_{34}e^{34}+\omega_{35}(-\frac{b}{2}e^{12}+ e^{35})-\frac{1}{3}\omega_{36}\mathrm{d}e^6$&$\omega_{36}\neq0$\\
	$\G_{6,18,1}^{\eta\neq\{1,\frac{1}{2},\frac{1}{3}\}}$&$\omega=\omega_{13}e^{13}+\omega_{23}e^{23}+\omega_{34}e^{34}+\omega_{35}e^{35}-\eta\omega_{36}\mathrm{d}e^6$&$\omega_{36}\neq0$\\
$\G_{6,18,1}^{\eta=1}$&$\omega=\omega_{13}e^{13}+\omega_{23}e^{23}+\omega_{34}e^{34}+\omega_{35}e^{35}+\omega_{45}e^{45}-\omega_{36}\mathrm{d}e^6$&$\omega_{36}\neq0$\\
$\G_{6,18,2}^{\eta=\frac{1}{2}}$&$\omega=\omega_{13}e^{13}+\omega_{23}e^{23}+\omega_{34}e^{34}+\omega_{35}e^{35}-\frac{1}{2}\omega_{36}\mathrm{d}e^6$&$\omega_{36}\neq0$\\
$\G_{6,18,3}^{\eta=\frac{1}{3}}$&$\omega=\omega_{13}e^{13}+\omega_{23}e^{23}+\omega_{34}(-\frac{a}{2}e^{12}+ e^{34})+\omega_{35}(-\frac{b}{2}e^{12}+e^{35})-\frac{1}{3}\omega_{36}\mathrm{d}e^6$&$\omega_{36}\neq0$\\
$\G_{6,19,1}^{\alpha,\gamma}$&$\omega=\omega_{13}e^{13}+\omega_{23}e^{23}+\omega_{34}e^{34}+\omega_{35}e^{35}-\gamma\omega_{36}\mathrm{d}e^6$&$\omega_{36}\neq0$\\
$\G_{6,19,1}^{\alpha=-1+\frac{2}{\gamma},\gamma}$&$\omega=\omega_{13}e^{13}+\omega_{23}e^{23}+\omega_{34}e^{34}+\omega_{35}e^{35}+\omega_{45}e^{45}-\gamma\omega_{36}\mathrm{d}e^6$&$\omega_{36}\neq0$\\
$\G_{6,19,1}^{\alpha=\frac{1}{3},\gamma=\frac{3}{2}}$&$\omega=\omega_{13}e^{13}+\omega_{23}e^{23}+\omega_{24}e^{24}+\omega_{34}e^{34}+\omega_{35}e^{35}+\omega_{45}e^{45}-\frac{3}{2}\omega_{36}\mathrm{d}e^6$&$\omega_{36}\neq0$\\
$\G_{6,19,1}^{\alpha=1-\frac{1}{\gamma},\gamma}$&$\omega=\omega_{13}e^{13}+\omega_{23}e^{23}+\omega_{24}e^{24}+\omega_{34}e^{34}+\omega_{35}e^{35}-\gamma\omega_{36}\mathrm{d}e^6$&$\omega_{36}\neq0$\\
$\G_{6,19,1}^{\alpha=1+\frac{1}{\gamma},\gamma}$&$\omega=\omega_{13}e^{13}+\omega_{15}e^{15}+\omega_{23}e^{23}+\omega_{34}e^{34}+\omega_{35}e^{35}-\gamma\omega_{36}e^{36}$&$\omega_{36}\neq0$\\
$\G_{6,19,2}^{\alpha=\frac{1}{2\gamma},\gamma}$&$\omega=\omega_{13}e^{13}+\omega_{23}e^{23}+\omega_{34}e^{34}+\omega_{35}e^{35}-\gamma\omega_{36}\mathrm{d}e^{6}$&$\omega_{36}\neq0$\\
$\G_{6,19,2}^{\alpha=\frac{1}{3},\gamma=\frac{3}{2}}$&$\omega=\omega_{13}e^{13}+\omega_{23}e^{23}+\omega_{24}e^{24}+\omega_{34}e^{34}+\omega_{35}e^{35}-\frac{3}{2}\omega_{36}\mathrm{d}e^{6}$&$\omega_{36}\neq0$\\
$\G_{6,19,2,a=0}^{\alpha=\frac{1}{3},\gamma=\frac{3}{2}}$&$\omega=\omega_{13}e^{13}+\omega_{23}e^{23}+\omega_{24}e^{24}+\omega_{34}e^{34}+\omega_{35}e^{35}+\omega_{45}e^{45}-\frac{3}{2}\omega_{36}\mathrm{d}e^{6}$&$\omega_{36}\neq0$\\
$\G_{6,19,3}^{\alpha=-1+\frac{1}{\gamma},\gamma}$&$\omega=\omega_{13}e^{13}+\omega_{23}e^{23}+\omega_{34}e^{34}+\omega_{35}e^{35}-\gamma\omega_{36}\mathrm{d}e^{6}$&$\omega_{36}\neq0$\\
$\G_{6,19,4}^{\alpha,\gamma=\frac{1}{2}}$&$\omega=\omega_{13}e^{13}+\omega_{23}e^{23}+\omega_{34}e^{34}+\omega_{35}e^{35}-\frac{1}{2}\omega_{36}\mathrm{d}e^{6}$&$\omega_{36}\neq0$\\
$\G_{6,19,5}^{\alpha=-\frac{1}{2}+\frac{1}{2\,\gamma},\gamma}$&$\omega=\omega_{13}e^{13}+\omega_{23}e^{23}+\omega_{34}e^{34}+\omega_{35}(-\frac{-2\,\gamma d }{\gamma+1}e^{12}+e^{35})-\gamma\omega_{36}\mathrm{d}e^6$&$\omega_{36}\neq0$\\
$\G_{6,19,5}^{\alpha=-\frac{1}{3},\gamma=3}$&$\omega=\omega_{13}e^{13}+\omega_{23}e^{23}+\omega_{34}e^{34}+\omega_{35}(-\frac{3d}{2}e^{12}+e^{35})+\omega_{45}(-3de^{26}+e^{45})$\\
&$~~~~~~~-\frac{1}{2}\omega_{36}\mathrm{d}e^6$&$\Omega_8\neq0$
\\
$\G_{6,19,6}^{\alpha=\gamma=\frac{1}{2}}$&$\omega=\omega_{13}e^{13}+\omega_{23}e^{23}+\omega_{34}e^{34}+\omega_{35}(-\frac{2\,a}{3}e^{12}+e^{35})-\frac{1}{2}\omega_{36}\mathrm{d}e^6$&$\omega_{36}\neq0$\\
$\G^{\alpha=-2+\frac{1}{\gamma},\gamma}_{6,19,7}$&$\omega=\omega_{13}e^{13}+\omega_{23}e^{23}+\omega_{35}e^{35}+\omega_{34}(\frac{a\,\gamma}{-1+\gamma} e^{12}+e^{34})-\gamma\omega_{36}\mathrm{d}e^6$&$\omega_{36}\neq0$\\
$\G^{\alpha=-\frac{1}{2},\gamma=\frac{2}{3}}_{6,19,7}$&$\omega=\omega_{13}e^{13}+\omega_{23}e^{23}+\omega_{24}e^{24}+\omega_{35}e^{35}+\omega_{34}(-2\,a\,e^{12}+e^{34})-\frac{2}{3}\omega_{36}\mathrm{d}e^6$&$\omega_{36}\neq0$\\
$\G_{6,20,1}^{\gamma\neq\{1,\frac{2}{3}\}}$&$\omega=\omega_{13}e^{13}+\omega_{23}e^{23}+\omega_{34}e^{34}+\omega_{24}(e^{24}+\frac{2-\gamma}{2\,\gamma}e^{35})-\gamma\omega_{36}\mathrm{d}e^6$&$\omega_{36}\neq0$\\
$\G_{6,20,1}^{\gamma=-2}$&$\omega=\omega_{13}e^{13}+\omega_{23}e^{23}+\omega_{34}e^{34}+\omega_{15}(e^{15}-2\,e^{26})+2\,\omega_{36}\mathrm{d}e^6$&$\Omega_9\neq0$\\
$\G_{6,20,1}^{\gamma=\frac{4}{3}}$&$\omega=\omega_{13}e^{13}+\omega_{23}e^{23}+\omega_{34}e^{34}+\omega_{45}e^{45}+\omega_{24}(e^{24}+\frac{1}{4}e^{35})-\frac{4}{3}\omega_{36}\mathrm{d}e^6$&$\omega_{36}\neq0$
		\end{tabular}}
		\end{center}
		\newpage
\begin{center}
{\renewcommand*{\arraystretch}{1.3}
	\small\begin{tabular}{l l l}
	$\G_{6,20,2}^{\gamma=\frac{2}{3}}$&$\omega=\omega_{13}e^{13}+\omega_{23}e^{23}+\omega_{34}e^{34}+\omega_{45}e^{45}+\omega_{35}(e^{24}+e^{35})-\omega_{36}\mathrm{d}e^6$&$\omega_{36}\neq0$\\
	$\G_{6,20,3}^{\gamma=1}$&$\omega=\omega_{13}e^{13}+\omega_{23}e^{23}+\omega_{34}e^{34}+\omega_{45}e^{45}+\omega_{35}(2\,e^{24}+e^{35})-\omega_{36}\mathrm{d}e^6$&$\omega_{36}\neq0$\\
	$\G_{6,20,4}^{\gamma=\frac{2}{5}}$&$\omega=\omega_{13}e^{13}+\omega_{23}e^{23}+\omega_{24}(e^{24}+2e^{35})+\omega_{34}(-\frac{2a}{3}e^{12}+e^{34})+\omega_{36}\mathrm{d}e^6$&$\omega_{36}\neq0$
	\\
	$\G_{6,21,1}^{\nu\neq\{\frac{1}{2},1\}}$&$\omega=\omega_{12}e^{12}+\omega_{13}e^{13}+\omega_{23}e^{23}+\omega_{34}e^{34}+\omega_{35}e^{35}-\nu\,\omega_{36}\mathrm{d}e^6$&$\omega_{36}\neq0$\\
	$\G_{6,21,2}^{\nu=\frac{1}{2}}$&$\omega=\omega_{12}e^{12}+\omega_{13}e^{13}+\omega_{23}e^{23}+\omega_{34}e^{34}+\omega_{35}e^{35}-\frac{1}{2}\,\omega_{36}\mathrm{d}e^6+(\omega_{24}e^{24}~\text{if}~a=0)$&$\omega_{36}\neq0$\\
	$\G_{6,21,3}^{\nu=1}$&$\omega=\omega_{12}e^{12}+\omega_{13}e^{13}+\omega_{23}e^{23}+\omega_{34}e^{34}+\omega_{35}e^{35}-\omega_{36}\mathrm{d}e^6+(\omega_{34}e^{34}~\text{if}~b=0)$&$\omega_{36}\neq0$\\
$\G^{\beta\neq\frac{1}{2\delta_1}}_{6,22,1}$	&$\omega=\omega_{13}e^{13}+\omega_{23}e^{23}+\omega_{34}e^{34}+\omega_{35}e^{35}-\delta_1\omega_{36}\mathrm{d}e^6$&$\omega_{36}\neq0$\\
$\G^{\beta=\frac{1}{\delta_1}}_{6,22,1}$	&$\omega=\omega_{13}e^{13}+\omega_{23}e^{23}+\omega_{34}e^{34}+\omega_{35}e^{35}+\omega_{45}e^{45}-\delta_1\omega_{36}\mathrm{d}e^6$&$\omega_{36}\neq0$\\
$\G^{\beta=\frac{1}{2\delta_1}}_{6,22,2}$&$\omega=\omega_{13}e^{13}+\omega_{23}e^{23}+\omega_{34}e^{34}+\omega_{35}e^{35}-\delta_2\,\omega_{36}\mathrm{d}e^6$&$\omega_{36}\neq0$\\
$\G_{6,23}^{\delta_2}$&$\omega=\omega_{12}e^{12}+\omega_{13}e^{13}+\omega_{23}e^{23}+\omega_{34}e^{34}+\omega_{35}e^{35}-\delta_2\,\omega_{36}\mathrm{d}e^6$&$\omega_{36}\neq0$

	\\\hline
	\end{tabular}}
\end{center}
$\Omega_1=\omega_{3 5}^3b^2 + (\omega_{3 4}^3 - 2\,\omega_{3 5}\omega_{3 4}^2 + \omega_{3 5}(\omega_{3 5} - 3\,\omega_{3 6})\omega_{3 4} + \omega_{3 5}^2\omega_{3 6})b + \omega_{3 6}^2(\omega_{3 4} + \omega_{3 6})$		\\
$\Omega^{\pm}=(2\,\omega_{1 6} - \omega_{3 6})(2\,\omega_{1 6}\omega_{3 6} \pm \omega_{3 5}^2 - \omega_{3 6}^2)$\\
$\Omega_2=((a - b)\,\omega_{3 4} - b\,\omega_{3 6})(\omega_{3 4} - \omega_{3 6})$\\
$\Omega_3=(\lambda(a - b)\omega_{3 4} + \omega_{3 6}b)(\lambda\omega_{3 4} + \omega_{3 6})$\\
$\Omega_4=(\omega_{3 4} + \omega_{3 6})(\omega_{1 5}\omega_{2 6} - \omega_{3 4}\omega_{3 6} - \omega_{3 6}^2)$\\
$\Omega_5=(2\,\omega_{2 3}\omega_{4 5} - 4\,\omega_{3 4}^2 + \omega_{3 6}^2)(2\,\omega_{3 4} - \omega_{3 6})$\\
$\Omega_6=(\omega_{34}-\omega_{36})(\omega_{26}\omega_{35}-\omega_{36}^2)$
\\
$\Omega_7=(2\,\omega_{3 4} - \omega_{3 6})(2\,\omega_{2 3}\omega_{4 5} + \omega_{3 6}^2)$\\
$\Omega_8=(6\,d\,\omega_{1 3}\omega_{4 5}^2 - 9\,d\,\omega_{3 5}\omega_{3 6}\omega_{4 5} - 2\,\omega_{3 6}^3)$\\
$\Omega_9=(2\,\omega_{15}^2\omega_{34} - 3\,\omega_{15}\omega_{35}\omega_{36} - \omega_{36}^3)$

\subsection{Symplectic structures on six-dimensional Frobeniusian Lie algebras without Lagrangian ideal}
\begin{center}
\captionof{table}{Symplectic structures on six-dimensional Frobeniusian Lie algebras without Lagrangian ideal}\label{cocyclewithout}
{\renewcommand*{\arraystretch}{1.3}
	\small
\begin{tabular}{l l l}\hline
		Algebra&Symplectic form&Conditions on $\omega_{ij}$\\\hline
		$N_{6,35}^{\alpha\neq0,\beta}$&$\omega=\omega_{12}e^{12}+\omega_{25}(e^{14}+e^{25})+\omega_{24}(-e^{15}+e^{24})+\omega_{16}(e^{16}+\frac{\beta}{\alpha}e^{26})$&\\
		&$~~~~~~+\omega_{45}(2e^{23}+e^{45})$&$\omega_{16}\omega_{45}\neq0$\\
		$N_{6,37}$&$\omega=\omega_{12}e^{12}+\omega_{15}(-\alpha e^{14}+e^{15}-(1+\alpha^2)e^{24})$&\\
		&$~~~~~~+\omega_{25}(e^{14}+\alpha e^{24}+e^{25})+\omega_{16}(e^{16}+2e^{23}+e^{45})+\omega_{26}e^{26}$&$\omega_{16}\neq0$
		\\
		$\G_{6,89}$&$\omega=\sum^5_{i=2}\omega_{i6}e^{i6}+\omega_{24}\mathrm{d}e^1$&$\omega_{24}\neq0$\\
		$\G_{6,90}$&$\omega=\sum^5_{i=2}\omega_{i6}e^{i6}+\omega_{24}\mathrm{d}e^1$&$\omega_{24}\neq0$\\
	$\G_{6,92}^\ast$&$\omega=\sum^5_{i=2}\omega_{i6}e^{i6}+\omega_{24}\mathrm{d}e^1$&$\omega_{24}\neq0$\\
		$\G_{6,93}$&$\omega=\sum^5_{i=2}\omega_{i6}e^{i6}+\omega_{24}\mathrm{d}e^1$&$\omega_{24}\neq0$
		
		\\\hline
		\end{tabular}}
		\end{center}

\newpage
\subsection{The flat torsion-free connections that are derived from the symplectic structures}
\begin{center}
\captionof{table}{The flat torsion-free connections that are derived from the symplectic structures}\label{flatfromsymplectic}
{\renewcommand*{\arraystretch}{1.3}
	\small\begin{tabular}{l l }\hline
		Algebra&Flat torsion-free connections\\\hline
		\multirow{1}{*}{$\mathfrak{b}_0=\left(\frac{\G_{6,7}}{\langle e_4,e_5,e_6\rangle},\psi^\ast\nabla^{\omega}\right)$}&$\nabla_{e_1}e_j=e_j, j=1,2,3,~\nabla_{e_2}e_1=e_2,~\nabla_{e_2}e_2= \epsilon e_3,~\nabla_{e_3}e_1=e_3,$
		\\
		&$\nabla_{e_3}e_2= e_2,~\nabla_{e_3}e_3= 2e_3,~\epsilon=\pm1,$\\
		&$\psi : \B_0\longrightarrow\B_0 : e_1\mapsto e_1,~ e_2\mapsto x e_2,~e_3\mapsto-\frac{x\,\omega_{35}}{2\,\omega_{16}-\omega_{36}}e_2+e_3,~x\neq0$.\\
		\multirow{1}{*}{$\mathfrak{b}_1=\left(\frac{\G^{\lambda=1}_{6,11,2}}{\langle e_4,e_5,e_6\rangle},\psi^\ast\nabla^{\omega}\right)$}&$\nabla_{e_1}e_1=e_1,~\nabla_{e_1}e_2=-\delta_0 e_1+e_2+\delta_0 e_3,~\nabla_{e_1}e_3=e_1,~\nabla_{e_2}e_1=-\delta_0 e_1+e_2+\delta_0 e_3,$\\
		&$\nabla_{e_2}e_3=-\delta_0 e_1+e_2+\delta_0 e_3,~\nabla_{e_3}e_1=e_1,~\nabla_{e_3}e_2=-\delta_0 e_1+2\,e_2+\delta_0 e_3,~\nabla_{e_3}e_3=e_3$,\\
		&where, $\delta_0=0,1$,\\
		&$\B_1\longrightarrow \B_1 : e_1\mapsto e_1,~ e_2\mapsto x e_2,~e_3\mapsto-\frac{x\,\omega_{15}}{\omega_{34}+\omega_{36}}e_2+e_3$,~~$x\neq0$.\\
\end{tabular}}
		\end{center}
		\begin{center}
		{\renewcommand*{\arraystretch}{1.2}
		\small\begin{tabular}{ll}
		\hline
		\multirow{1}{*}{$\B_9=\left(\frac{\G_{6,17,1}^{ \mu=1}}{\langle e_1,e_2,e_6\rangle},\psi^\ast\nabla^{\omega}\right)$}&$\nabla_{\tilde{e}_1}\tilde{e}_1=\tilde{e}_1,~\nabla_{\tilde{e}_1}\tilde{e}_2=\tilde{e}_2+e_3,~\nabla_{\tilde{e}_1}e_3=e_3,~\nabla_{\tilde{e}_2}\tilde{e}_1=\tilde{e}_2,~\nabla_{e_3}\tilde{e}_1=e_3,$\\
		&$\psi:\B_9\rightarrow\B_9,~\tilde{e}_1\mapsto \tilde{e}_1+\frac{\omega_{13}}{\omega_{36}}e_3,~\tilde{e}_2\mapsto \tilde{e}_2,~~e_3\mapsto e_3$.
		\\
		$\B_{10}=\left(\frac{\G_{6,18,1}^{\eta=1}}{\langle e_1,e_2,e_6\rangle},\psi^\ast\nabla^{\omega}\right)$&$\nabla_{\tilde{e}_1}\tilde{e}_1=\tilde{e}_1,~\nabla_{\tilde{e}_1}\tilde{e}_2=\tilde{e}_2,~\nabla_{\tilde{e}_1}e_3=e_3,~\nabla_{\tilde{e}_2}\tilde{e}_1=\tilde{e}_2,~\nabla_{e_3}\tilde{e}_1=e_3$, $\psi=\mathrm{I}$.\\
\multirow{1}{*}{$\B_{11}=\left(\frac{\G_{6,19,1}^{\alpha,\gamma=1}}{\langle e_1,e_2,e_6\rangle},\psi^\ast\nabla^{\omega}\right)$}&$\nabla_{\tilde{e}_1}\tilde{e}_1=\tilde{e}_1,~\nabla_{\tilde{e}_1}\tilde{e}_2=\tilde{e}_2,~\nabla_{\tilde{e}_1}e_3=\alpha e_3,~\nabla_{\tilde{e}_2}\tilde{e}_1=\tilde{e}_2,~\nabla_{e_3}\tilde{e}_1=e_3$,\\
&$\psi:\B_{11}\rightarrow\B_{11},~\tilde{e}_1\mapsto \tilde{e}_1+\frac{\omega_{23}(\alpha-1)}{\alpha\omega_{36}}e_3,~\tilde{e}_2\mapsto \tilde{e}_2,~e_3\mapsto e_3,$\\
&$0<|\alpha|<1$, with $\alpha\neq\{\frac{1}{2},-1\}$.\\
\multirow{1}{*}{$\B_{12}=\left(\frac{\G_{6,19,1}^{\alpha=-1+\frac{2}{\gamma},\gamma}}{\langle e_1,e_2,e_6\rangle},\psi^\ast\nabla^{\omega}\right)$}&$\nabla_{\tilde{e}_1}\tilde{e}_1=\frac{1}{\gamma}\tilde{e}_1,~\nabla_{\tilde{e}_1}\tilde{e}_2=\tilde{e}_2,~\nabla_{\tilde{e}_1}e_3=\frac{2-\gamma}{\gamma} e_3,~\nabla_{\tilde{e}_2}\tilde{e}_1=\frac{1}{\gamma}\tilde{e}_2,~\nabla_{e_3}\tilde{e}_1=\frac{1}{\gamma}e_3$,\\
\end{tabular}}
\end{center}
\begin{center}
{\renewcommand*{\arraystretch}{1.2}
		\small\begin{tabular}{ll}
\multirow{1}{*}{$\B_{22}=\left(\frac{\G^{\alpha=-\frac{1}{2},\gamma=\frac{2}{3}}_{6,19,7}}{\langle e_4,e_5,e_6\rangle},\psi^\ast\nabla^{\omega}\right)$}&$\nabla_{e_1}e_3=\frac{3}{2}e_1,~\nabla_{e_2}e_3=\delta_0e_1+\frac{3}{2}e_2,~\nabla_{e_3}e_1=\frac{1}{2}e_1,~\nabla_{e_3}e_2=\delta_0 e_1+2e_2,$\\
&$\nabla_{e_3}e_3=\frac{3}{2}e_3,~\delta_0=0,1$,\\
&$\psi:\B_{22}\rightarrow\B_{22},~e_1\mapsto e_1,~e_2\mapsto x_{22}e_2, e_3\mapsto \frac{\omega_{24}\omega_{35}-4\,\omega_{34}\omega_{36}}{2\,\omega_{36}^2}e_1+\frac{x_{22}\,\omega_{35}}{4\,\omega_{36}}e_2+e_3$
		\\
\multirow{1}{*}{$\B_{23}=\left(\frac{\G_{6,20,1}^{\gamma=-2}}{\langle e_4,e_5,e_6\rangle},\psi^\ast\nabla^{\omega}\right)$}&$\B_{23,1}$ : $\nabla_{e_1}e_2=e_3,~\nabla_{e_2}e_1=e_3,~\nabla_{e_2}e_2=e_1,~\nabla_{e_3}e_1=-e_1,~\nabla_{e_3}e_2=-\frac{1}{2}e_2,$\\
&$~~~~~~~~~~~\nabla_{e_3}e_3=-\frac{3}{2}e_3$,
\\
&$\psi:\B_{23}\rightarrow\B_{23},~e_1\mapsto
\frac{\sqrt[3]{4}\,\omega_{15}^2}{\sqrt[3]{\det(\omega)^2}}  e_1,~e_2\mapsto \frac{\sqrt[3]{2}\,\omega_{15}\,\sqrt[3]{\det(\omega)^2}}{\det(\omega)}e_2,$\\& $~~~~~~~~~~~~~~~~~~~~~~~~e_3\mapsto \frac{(2\,\omega_{15}\omega_{35}+\omega_{36}^2)\,\sqrt[3]{2}}{2\,\sqrt[3]{\det(\omega)^2}}e_1+\frac{\sqrt[3]{2}\,\omega_{36}\,\sqrt[3]{\det(\omega)^2}}{2\,\det(\omega)}e_2+e_3$, $\omega_{15}\neq0$.\\
&$\B_{23,2}$ : $\nabla_{e_1}e_3=-\frac{1}{2} e_1,~\nabla_{e_2}e_2=e_1,~\nabla_{e_2}e_3=-\frac{1}{2}e_2,~\nabla_{e_3}e_1=-\frac{3}{2} e_1,~\nabla_{e_3}e_2=-e_2$,\\
\hline
\end{tabular}}
		\end{center}

		\newpage
		\begin{Le}
The Lagrangian extension cohomology group and the relative cohomology of the three-dimensional real flat  Lie algebras  $(\mathfrak{b},\nabla)$ given in Table $\ref{flatfromsymplectic}$ are presented in the following table$:$
\end{Le}
\begin{center}
\captionof{table}{Cohomological properties of  the three-dimensional real flat  Lie algebras  $(\mathfrak{b},\nabla)$ given in Table $\ref{flatfromsymplectic}$.}\label{cohoforflatsym}
{\renewcommand*{\arraystretch}{1.1}
\small\begin{tabular}{lll}
			\hline
		Flat Lie algebra& $H^2_{L,\rho}(\B,\B^\ast)$&$H^2_{\rho}(\B,\B^\ast)$\\\hline
$\B_0$	&$[\alpha_{13}\mapsto \tau^\prime e_1^\ast]$&0	\\
		$\B_1$&$[\alpha_{13}\mapsto \tau^\prime e_1^\ast+\tau^\prime e_3^\ast]$&$0$\\
		$\B_{21}$&$[\alpha_{13}\mapsto 2\,\tau^\prime e_1^\ast-\tau^\prime e_3^\ast]$&$0$\\
		$\B_{22}$&$[\alpha_{13}\mapsto \tau^\prime e_3^\ast]$&$0$\\
		$\B_3$&$\left[\begin{array}{l}\alpha_{12}\mapsto \tau^\prime e_3^\ast\\\alpha_{13}\mapsto \tau^\prime e_2^\ast+\tau^\prime e_3^\ast\\\alpha_{23}\mapsto \tau^{\prime\prime} e_3^\ast\end{array}\right]$&$0$\\
		$\B_4$&$[\alpha_{13}\mapsto \tau^\prime e_1^\ast]$&$0$\\
		$\B_5$&$[\alpha_{13}\mapsto \tau^\prime e_1^\ast]$&$0$\\
\end{tabular}}
		\small\begin{tabular}{lll}
$\B_{19}^{\delta_0=1}$&$0$&$0$\\
	$\B_{19}^{\delta_0=0}$&$\left[\begin{array}{l}
\alpha_{12}\mapsto de_1^\ast\\
	\alpha_{13}\mapsto \tau^{\prime} e_2^\ast	\\\alpha_{23}\mapsto \tau^{\prime} e_1^\ast\end{array}\right]$&$[\alpha_{12}\mapsto de_1^\ast]$\\
$\B_{20}$&$0$&$0$\\
$\B_{21}$&$0$&$0$\\
$\B_{22}$&$[\alpha_{12}\mapsto a^\prime e_1^\ast-\frac{2\,\delta_0\,a^\prime}{3}e_2^\ast]$&$[\alpha_{12}\mapsto a^\prime e_1^\ast-\frac{2\,\delta_0\, a^\prime}{3}e_2^\ast]$\\
$\B_{23}$&$0$&$0$\\
$\B_{24}$&$[\alpha_{13}\mapsto a^\prime e_1^\ast-\frac{4\,\delta_0\,a^\prime}{3}e_3^\ast]$&$[\alpha_{13}\mapsto a^\prime e_1^\ast-\frac{4\,\delta_0\,a^\prime}{3}e_3^\ast]$\\
$\B_{25}$&$\left[\begin{array}{l}
\alpha_{12}\mapsto a^\prime e_1^\ast\\
	\alpha_{13}\mapsto b^\prime e_1^\ast	\\
	\alpha_{23}\mapsto \tau^{\prime} e_2^\ast\end{array}\right]$&$\left[\begin{array}{l}
\alpha_{12}\mapsto a^\prime e_1^\ast\\
	\alpha_{13}\mapsto b^\prime e_1^\ast\end{array}\right]$\\
	$\B_{26}^{\delta_0=1}$&$\left[\begin{array}{l}
\alpha_{12}\mapsto \tau^\prime e_3^\ast\\
	
	\alpha_{23}\mapsto -\tau^{\prime} e_1^\ast\end{array}\right]$&$0$
	\\
	$\B_{26}^{\delta_0=0}$&$\left[\begin{array}{l}
\alpha_{12}\mapsto a^\prime e_1^\ast+\tau^\prime e_3^\ast\\
	
	\alpha_{23}\mapsto -\tau^{\prime} e_1^\ast\end{array}\right]$&$[\alpha_{12}\mapsto a^\prime e_1^\ast]$\\
	$\B_{27}$&$\left[\begin{array}{l}
\alpha_{13}\mapsto \tau^\prime e_2^\ast\\
	
	\alpha_{23}\mapsto \tau^{\prime} e_1^\ast\end{array}\right]$&$0$
			\\\hline
			\end{tabular}
		\end{center}
		
		\newpage
		\begin{center}
		\captionof{table}{Isomorphisms of $\B_j\oplus\B^\ast_j$ to algebras $\G_{6,j}$ given in Proposition \ref{Prclassifi}}\label{isofrom new to new}
{\renewcommand*{\arraystretch}{1.3}
\small\begin{tabular}{lll}
			\hline
		Source&Isomorphism&Target\\\hline
		\multirow{1}{*}{$\B_0\oplus\B^\ast_0$}&$f_1=e_1, f_2=e_2, f_3=e_3, f_4=e_4, f_5=e_5, f_6=-e_5+e_6$, for $\epsilon=-1$&$\G_{6,7}$\\
		&$f_j=e_j$, $j=1,\ldots 6$ when $\epsilon=1$&\\
		\multirow{1}{*}{$\B_1\oplus\B^\ast_1$}&$f_1=e_1, f_2=e_2, f_3=e_3, f_4=e_4+e_5, f_5=e_5, f_6=-e_5+e_6$, for $\delta_0=1$&$\G^{\lambda=1}_{6,11,2}$\\
		&$f_j=e_j$, $j=1,\ldots 6$ when $\delta_0=0$&\\
		\multirow{1}{*}{$\B_2\oplus\B^\ast_2$} &$\B_{21} : $ $f_1=e_1,f_2=e_5,
		f_3= e_3,f_4=e_4-e_5,f_5=-\tilde{e}_2,f_6=-2\,e_5+e_6$&$\G_{6,12}^{\lambda=-2}$\\
		&$\B_{22} : $ $f_1=e_1, f_2=-e_5, f_3=e_3, f_4=e_4, f_5=e_2, f_6=e_6$&\\
		\multirow{1}{*}{$\B_3\oplus\B^\ast_3$}&$f_1=e_1, f_2=s e_5, f_3=e_3, f_4=e_4, f_5=-\frac{1}{s}\tilde{e}_2, f_6=e_6,~~s\neq0$&$\G_{6,13,1}^{\lambda=-1}$\\
\end{tabular}}	
\end{center}		
\begin{center}
{\renewcommand*{\arraystretch}{1.3}
\small\begin{tabular}{lll}
\multirow{1}{*}{$\B_{20}\oplus\B_{20}^\ast$}&$\B_{20,1} :$ $f_1=-e_5, f_2=\frac{\sqrt[3]{-9\,d}}{3}\tilde{e}_1, f_3=e_2, f_4=\frac{\sqrt[3]{3}}{3\,\sqrt[3]{-d}}e_3, f_5=\frac{\sqrt[3]{9}}{3\,\sqrt[3]{-d}^2}e_6, f_6=\frac{\sqrt[3]{3}}{\sqrt[3]{-d}}e_5,~d\neq0$&\multirow{2}{*}{$\G_{6,19,5}^{\alpha=-\frac{1}{3},\gamma=3}$}\\
&$\B_{20,2} :$ $f_1=de_6, f_2=\tilde{e}_1, f_3=e_2, f_4=-\frac{1}{d}e_3, f_5=e_4, f_4=e_5, ~~d\neq0$&\\
\multirow{1}{*}{$\B_{21}\oplus\B_{21}^\ast$}&$f_1=-\frac{1}{x}e_6, f_2=\tilde{e}_1, f_3=e_2, f_4=xe_3, f_5=e_4-3\,\delta_0 e_6, f_6=e_5,$ &$\G_{6,19,5,d=0}^{\alpha=-\frac{1}{3},\gamma=3}$\\
&where, $x=1$ (resp. $x=-\frac{1}{3}$) if $\delta_0=0$ (resp. $\delta_0=1$)&\\
\multirow{1}{*}{$\B_{22}\oplus\B_{22}^\ast$}&$f_1=x\,e_1, f_2=\frac{a}{a^\prime\,x^2}e_2, f_3=e_3, f_4=\frac{1}{x}e_4-\frac{2\,\delta_0}{3\, x}e_5, f_5=\frac{a^\prime\,x^2}{a}, f_6=e_6$ if $(a,a^\prime)\neq(0,0)$&$\G_{6,19,7}^{\alpha=-\frac{1}{2},\gamma=\frac{2}{3}}$\\
&where, $x=1$ (resp. $x=-\frac{\sqrt[3]{-18\,a\,a^{\prime 2}}}{3\,a^\prime}$) if $\delta_0=0$ (resp. if $\delta_0=1$)&\\
&$f_1=e_1, f_2=x\,e_2, f_3=e_3, f_4=e_4-\frac{2\,\delta_0}{3}e_5, f_5=\frac{1}{x}e_5, f_6=e_6$ if $(a,b)=(0,0)$&\\
&where, $x=1$ (resp. $x=-\frac{3}{2}$) if $\delta_0=0$ (resp. if $\delta_0=1$)&
\\
\hline
		\end{tabular}}
		\end{center}
		\begin{remark}
		To find the flat Lie algebra starting point for the cases where, $\psi=\mathrm{I}$, we can modify the Lagrangian ideal and use the ideal $\mathfrak{j}=\lbrace e_4,e_5,e_6\rbrace$, calculation without repeating.
		\end{remark}

\section{Appendix B}\label{AppenB}
\subsection{Isomorphisms of $\G_{6,j}$ to decomposable Lie algebras   }
\begin{center}
\captionof{table}{Isomorphisms of $\G_{6,j}$ to decomposable Lie algebras }\label{isodecompo}
{\renewcommand*{\arraystretch}{1.1}
	\small\begin{tabular}{l l l}\hline
		Source&Isomorphism&Target\\\hline
		$\G_{6,13,1}^{\lambda=-1}$&$f_1=\frac{1}{x} e_2, f_2=\frac{1}{x} e_5, f_3=e_6, f_4=e_1+e_3, f_5=-e_1, f_6=-e_4+e_6$, $x\neq0$&$\mathfrak{d}_{4,1}\oplus\mathfrak{aff}(1,\R)$	
\\		
	$\G_{6,13,1}^{\lambda=1}$&$f_1=e_5,f_2=-e_2,f_3=-e_6,f_4=e_1-e_3,f_5=-e_1,f_6=-e_4-e_6$&$\mathfrak{d}_{4,2}\oplus\mathfrak{aff}(1,\R)$	\\
		$\G_{6,13,1}^{\lambda=-2}$&$f_1=e_2,f_2=e_5,f_3=\frac{1}{2}e_6,f_4=e_1+2e_3,f_5=-e_1,f_6=-e_4+\frac{1}{2}e_6$&
\\	
\multirow{1}{*}{$\G_{6,13,1}^{\lambda^\prime\neq-\frac{1}{2}}$}&$f_1=e_2, f_2=e_5, f_3=\frac{1}{\lambda}e_3, f_4=e_1+\lambda e_3, f_5=-e_1, f_6=-e_4+\frac{1}{\lambda}e_6$&$\mathfrak{d}_{4,\lambda}\oplus\mathfrak{aff}(1,\R)$\\
&$\lambda=-\lambda^\prime>\frac{1}{2}$, $\lambda\neq 1,2$&\\
		\multirow{1}{*}{$\G_{6,13,2}^{\lambda^\prime=-\frac{1}{2}}$}&$f_1=e_2, f_2=e_5, f_3=2e_6, f_4=e_1+\frac{1}{2}e_3, f_5=-e_1, f_6=-e_4+2e_6$&\\
		&$a=b=0$,  $\lambda=\frac{1}{2}$&
		\\
	\multirow{1}{*}{$\G_{6,13,2}^{\lambda=-\frac{1}{2}}$}	&$f_1=-\frac{\sqrt{|b|}}{\sqrt{2}}e_5, f_2=\frac{\sqrt{2}}{\sqrt{|b|}}e_2, f_3=2e_6, f_4=e_1+\frac{1}{2}e_3, f_5=-e_1, f_6=-e_4+2e_6$, &$\h_4\oplus\mathfrak{aff}(1,\R)$\\
	&$a=0, b\neq0$&\\\hline
		\end{tabular}}
		\end{center}

\subsection{Isomorphisms of $\G_{6,j}$ to Turkowski's algebras   }
\begin{center}
\captionof{table}{Isomorphisms of $\G_{6,j}$ to Turkowski's algebras }\label{isoTur}
{\renewcommand*{\arraystretch}{1.2}
	\small\begin{tabular}{l l l}\hline
		Source&Isomorphism&Target\\\hline
		\multirow{1}{*}{$\G_{6,7}$}&$f_1=-e_1+e_3, f_2=-e_3, f_3=-e_4, f_4=-\sqrt{2}e_5, f_5=-e_5-2e_6, f_6=\tfrac{\sqrt{2}}{2}e_2$&$N_{6,28}$\\
		$\G_{6,8}$&$f_1=-e_1+e_3, f_2=-e_3, f_3=-e_4, f_4=-e_5, f_5=\frac{1}{2}e_4+e_6, f_6=e_2$&
		\\\hline
		\multirow{1}{*}{$\G_{6,4}^{a\neq0,1}$} &$f_1=-e_1-\frac{1}{\alpha}e_2+e_3,f_2=-e_1,
		f_3=- e_4,f_4=\frac{1}{\alpha} e_2,f_5= \alpha e_5+e_6,$&\multirow{1}{*}{$N_{6,29}^{\alpha\neq0,\beta\neq0}$}\\
		&$f_6=-\frac{1}{\alpha} e_4+\frac{\alpha-1}{\alpha}e_6$, $a=1-\alpha$, $\alpha\neq 0,1$, $\beta=1$ or &\\
&$f_1=-e_1, f_2=-e_1-\frac{1}{\beta}e_2+e_3, f_3=- e_4, f_4=e_5+\frac{1}{\beta}e_6, f_5=-e_2, $&\\
&$f_6=-e_5+(\beta-1)e_6$, $a=1-\beta$, $\beta\neq 0,1$, $\alpha=1$&	\\
$\G_{6,12}^{\lambda\neq-2}$&$f_1=-e_1, f_2=\frac{\beta-1}{2}e_1+e_3, f_3=-e_4+\frac{1+\beta}{2}e_6, f_4=-e_5, f_5=e_2, $&\\
\hline
		\end{tabular}}
		\end{center}
		\subsection{Isomorphisms of $\G_{6,j}$ to Mubarakyzanov algebras   }
		\begin{center}
		\captionof{table}{Isomorphisms of $\G_{6,j}$ to Mubarakyzanov algebras  }\label{isoMuba}
{\renewcommand*{\arraystretch}{1.3}
	\small\begin{tabular}{l l l}\hline
		Source&Isomorphism&Target\\\hline
		$\G_{6,1}$&$f_1=-\frac{1}{2}e_6, f_2=xe_4+ye_5, f_3=ze_4, f_4=-\frac{1}{2y}e_2,f_5=-\frac{1}{2z}e_1+\frac{x}{2yz}e_2, f_6=2e_3$,&$\G_{6,82}^{\lambda,\lambda_1,\alpha}$\\
		&where, $x=-\tau_2, y=-\tau_3$, $z=1$ or $z=\frac{\tau_1}{4\tau_3}$ (when $\tau_1\neq0$), (if $\tau_3\neq0$), &
		\\
		&if ($\tau_3=0$), $x=0$, $y=\frac{\tau_1}{4\tau_2}$ (when $\tau_1\neq0$) or $y=1$, $z=-\tau_2$ (if $\tau_2\neq0$)&\\
		& otherwise $z=1$ and $y=1$ or $y=-\frac{\tau_1}{4}$ (when $\tau_1\neq0$),~~\underline{$\lambda=\lambda_1=1,\alpha=2$}&\\
		&To find the symplectic forms, both paths are equivalent.&
		\\
		$\G_{6,18,1}^{\eta=1}$&$f_1=-\frac{1}{2}e_6, f_2=xe_1+\frac{1}{2y}e_2, f_3=ze_1, f_4=ye_5, f_5=\frac{1}{2z}e_4-\frac{xy}{z}e_5, f_6=2e_3$, &\\
		&\underline{$\lambda=\lambda_1=1,\alpha=2$}&\\
\hline
		\end{tabular}}
		{\renewcommand*{\arraystretch}{1.3}
	\small\begin{tabular}{l l l}\hline
	$\G_{6,19,2}^{\alpha=\frac{1}{2\gamma},\gamma}$&$f_1=-\frac{1}{2\gamma}e_6, f_2=e_2, f_3=e_1, f_4=\frac{1}{2}e_5, f_5=\frac{1}{2}e_4, f_6=2\gamma e_3$, $a=0$, $\gamma>\frac{1}{2}$, \underline{$\lambda=0,\lambda_1=2\gamma-1,\alpha=2$}&\\
	&or $f_1=-\frac{1}{2\gamma}e_6, f_2=e_2, f_3=e_4, f_4=\frac{1}{2}e_5, f_5=-\frac{1}{2}e_1, f_6=2\gamma e_3$, $a=0$, $\gamma<\frac{1}{2}$, \underline{$\lambda=0,\lambda_1=-2\gamma+1,\alpha=2$}&
	\\
	$\G_{6,18,1}^{\eta\neq\{\frac{1}{2},\frac{1}{3}\}}$&$f_1=-\frac{1}{2\eta}e_6, f_2=\frac{1}{2}e_2, f_3=e_1, f_4=e_5, f_5=\frac{1}{2}e_4, f_6=2\eta e_3$, $\lambda=\lambda_1=2\eta-1$, $\eta>\frac{1}{2}$, &\\
	&or $f_1=-\frac{1}{2\eta}e_6, f_2=-\frac{1}{2}e_4, f_3=-\frac{1}{2}e_5, f_4=e_1, f_5=e_2, f_6=2\eta$, $\lambda=\lambda_1=-2\eta+1$, $\eta<\frac{1}{2}$&
	\\
	&\underline{$\lambda=\lambda_1, \alpha=2$}&\\
	$\G_{6,19,3}^{\alpha=-1+\frac{1}{\gamma},\gamma}$&$f_1=-\frac{1}{2\gamma}, f_2=e_5, f_3=\frac{1}{2}e_1, f_4=-\frac{1}{2}e_2, f_5=e_4, f_6=2\gamma e_3$, $\lambda=\lambda_1=2\gamma-1$, $\gamma>\frac{1}{2}$, $\gamma\neq1$&
	\\
	&\underline{$\lambda=\lambda_1, \alpha=2$}&
	\\
	$\G_{6,18,3}^{\eta=\frac{1}{3}}$&$f_1=-\frac{3}{2}e_6, f_2=e_5, f_3=e_4, f_4=-\frac{1}{2}e_2, f_5=-\frac{1}{2}e_1, f_6=\frac{2}{3}e_3$, $a=b=0$, \underline{$\lambda=\lambda_1=\frac{1}{3},\alpha=2$}&\\
\hline
		\end{tabular}}

{\renewcommand*{\arraystretch}{1.3}
	\small\begin{tabular}{l l l}\hline
		Source&Isomorphism&Target\\\hline
		$\G_{6,19,2}^{\alpha=\frac{1}{2\gamma},\gamma}$&$f_1=-\frac{1}{\lambda+1}e_6, f_2=e_1,f_3=\frac{(\lambda+1)a\sqrt{2}}{2\sqrt{-(\lambda+1)|a|}}e_5, f_4=\frac{1}{2}e_4, f_5=\frac{\sqrt{2}}{\sqrt{-(\lambda+1)|a|}}e_2, f_6=(\lambda+1)e_3$,&$\G_{6,85}^{\lambda}$\\
&$\gamma=\frac{\lambda+1}{2}, a\neq0$, $\lambda>0,\lambda\neq1,2$ &		\\
&or $f_1=\frac{1}{\lambda-1}e_6, f_2=e_4, f_3=-\frac{a \left( \lambda-1 \right) \sqrt {2}}{2\sqrt{|a|
 \left( \lambda-1 \right) }}
e_5, f_4=-\frac{1}{2}e_1, f_5=\frac{1}{\sqrt {2}\sqrt{|a| \left( \lambda-1 \right)}}e_2, f_6=(1-\lambda)e_6$, &
\\
&$\gamma=\frac{1-\lambda}{2}$, $a\neq0$, $\lambda>2$&
\\
$\G_{6,18,2}^{\eta=\frac{1}{2}}$&$f_1=e_6, f_2=e_5, f_3=\sqrt{\frac{|a|}{2}}e_4, f_4=\frac{1}{2}e_2, f_5=\frac{1}{\sqrt{2|a|}}e_1, f_6=e_3$, $a\neq0,b=0, c=0$, $\lambda=0$&\\
&or $f_1=2cx^2e_6, f_2=-xe_4+e_5, f_3=cxe_5, f_4=-cxe_1, f_5=e_1+xe_2, f_6=e_3$, $a=0,b=0, c\neq0$, $\lambda=0$&
\\
&$f_1=-e_6,f_2=e_5,f_3=-\frac{b}{\sqrt{-2|c|}}e_4+\frac{\sqrt{-2|c|}}{2}e_5,f_4=\frac{c}{2b}e_1-\frac{1}{2}e_2,f_5=\frac{\sqrt{-2|c|}}{2b}e_1,f_6=e_3$, $b^2=ac\neq0$, $\lambda=0$&
\\
\hline
\end{tabular}}

{\renewcommand*{\arraystretch}{1.3}
	\small\begin{tabular}{l l l}\hline
		$\G_{6,20,1}^{\gamma=-2}$&$f_1=-\frac{1}{2}e_6, f_2=e_5, f_3=e_4, f_4=e_1, f_5=e_2, f_6=2e_3$&$\G_{6,94}^{\lambda=-3}$\\
$\G_{6,19,5}^{\alpha=-\frac{1}{3},\gamma=3}$&$f_1=\frac{1}{3}e_6, f_2=-\frac{1}{3}e_5, f_3=\frac{1}{9d}e_1, f_4=-9de_4, f_5=-3e_2, f_6=-3e_3$&		
		\\
		$\G_{6,20,1}^{\gamma=\frac{4}{3}}$&$f_1=-\frac{1}{2}e_6, f_2=\frac{\sqrt[3]{12^2}}{6}e_5, f_3=-\frac{\sqrt[3]{12}}{2}e_4, f_4=\frac{\sqrt[3]{12^2}}{9}e_1, f_5=-\frac{\sqrt[3]{12}}{3}e_2, f_6=2e_3$&$\G_{6,94}^{\lambda=-\frac{1}{2}}$\\
		$\G_{6,19,7}^{\alpha=-\frac{1}{2},\gamma=\frac{2}{3}}$&$f_1=-e_6, f_2=-{\frac {{12}^{{\frac{2}{3}}}}{6}{\frac {a}{\sqrt [3]{{a}^{2}}}}}e_4, f_3={\frac {a\left( {a}^{2} \right) ^{-{\frac{2}{3}}}\sqrt [3]{12}}{2} }e_2, f_4={\frac {{12}^{{\frac{2}{3}}}\left( {a}^{2} \right) ^{{\frac{2}{3}}}}{9\,a} }e_5, f_5={\frac {\sqrt [3]{12}\sqrt [3]{{a}^{2}}}{3\,a}}e_1, f_6=e_3$&
		\\
$\G_{6,21,3}^{\nu=1}$&$f_1=-e_6, f_2=-e_4, f_3=\frac{1}{b}e_2, f_4=be_5, f_5=e_1, f_6=e_3$, $,b\neq0$&$\G_{6,94}^{\lambda=-1}$\\
$\G_{6,20,3}^{\gamma=1}$&$f_1=-\frac{1}{2}e_1, f_2=-\frac{1}{2x}e_5, f_3=-\frac{1}{2x^2}e_4, f_4=x^2e_1, f_5=xe_2, f_6=2e_3$, $b=0$&$\G_{6,94}^{\lambda=0}$\\
$\G_{6,14,3}^{\lambda=-\frac{1}{2}}$&$f_1=e_6, f_2=-\frac{1}{2x}e_5, f_3=-\frac{1}{2ax^2}e_1, f_4=-ax^2e_4, f_5=xe_2, f_6=-e_3$, $x\neq0$, $b=0$, $a\neq0$&
\\
$\G_{6,20,2}^{\gamma=\frac{2}{3}}$&$f_1=-\frac{1}{2}e_6, f_2=-\frac{1}{3}e_5, f_3=-\frac{1}{3}e_4, f_4=e_1, f_5=e_2, f_6=2e_3$, $a=0$&$\G_{6,94}^{\lambda=1}$\\
\hline
		\end{tabular}}
		\end{center}

\end{document}